\documentclass{amsart}

\usepackage{amssymb,amsmath,graphicx}
\usepackage{xcolor}
\usepackage{enumerate}
\newtoks\prt

\newtheorem{thm}{Theorem}[section]

\newtheorem{ques}[thm]{Question}
\newtheorem{lemma}[thm]{Lemma}
\newtheorem{prop}[thm]{Proposition}
\newtheorem{cor}[thm]{Corollary}

\newtheorem{example}[thm]{Example}

\theoremstyle{definition}

\newtheorem{remark}[thm]{Remark}

\def\eqn#1$$#2$${\begin{equation}\label#1#2\end{equation}}
 \def\J#1#2#3{ \left\{ #1,#2,#3 \right\} }

\def\C{\mathcal C}

\def\T{\mathcal T}

\def\U{\mathcal U}

\def\co{\operatorname{conv}}
\def\ep{\varepsilon}

\def\en{\mathbb N}
\def\er{\mathbb R}

\def\dist{\operatorname{dist}}

\def \g{\boldsymbol{g}}
\def \h{\boldsymbol{h}}

\def\spt{\operatorname{spt}}
\def\span{\operatorname{span}}

\def \reg {\partial _{\kern1pt\text{reg}}}

\def\ip#1#2{\left\langle#1,#2\right\rangle}

\def\di{\,\mbox{\rm d}}
\def\dh{\widehat{\operatorname{d}}}
\def\clu#1{\operatorname{clust}_{w^*}(#1)}

\newcommand{\Norm}[1]{\Bigl\|#1\Bigr\|}
\newcommand{\norm}[1]{\left\|#1\right\|}

\newcommand{\betr}[1]{| #1  |}
\renewcommand{\Re}{\operatorname{Re}}

\newcommand{\wk}[2][X]{\operatorname{wk}_{#1}\left(#2\right)}
\newcommand{\wck}[2][X]{\operatorname{wck}_{#1}\left(#2\right)}
\newcommand{\wscl}[1]{\overline{#1}^{w^*}}

\newcommand{\abs}[1]{\left|#1\right|}
\newcommand{\setsep}{;\,}
\let\subset\subseteq

\title[Measures of weak non-compactness]{Measures of weak non-compactness in preduals of von Neumann algebras and JBW$^*$-triples}

\author[J. Hamhalter]{Jan Hamhalter}

\author[O.F.K. Kalenda]{Ond\v{r}ej F.K. Kalenda}

\author[A.M. Peralta]{Antonio M. Peralta}

\author[H. Pfitzner]{Hermann Pfitzner}

\address{Czech Technical University in Prague, Faculty of Electrical Engineering, Department of Mathematics, Technicka 2, 166 27, Prague 6,
Czech Republic}
\email{hamhalte@math.feld.cvut.cz}
\address{Charles University, Faculty of Mathematics and Physics, Department of
Mathematical Analysis, Sokolovsk{\'a} 86, 186 75 Praha 8, Czech Republic}
\email{kalenda@karlin.mff.cuni.cz}
\address{Departamento de An{\'a}lisis Matem{\'a}tico, Facultad de
Ciencias, Universidad de Gra\-na\-da, 18071 Granada, Spain.}
\email{aperalta@ugr.es}
\address{Institut Denis Poisson, Universit\'{e} d'Orl\'{e}ans, Universit\'{e} de Tours,
CNRS, Rue de Chartres, BP 6759, F-45067 Orl\'{e}ans Cedex 2, France}
\email{hermann.pfitzner@univ-orleans.fr}

\thanks{The first two authors were in part supported by the Research Grant GA\v{C}R 17-00941S. The first author was partly supported further by the project OP VVV Center for Advanced Applied Science CZ.02.1.01/0.0/0.0/16\_019/000077. The third author was partially supported by the Spanish Ministry of Science, Innovation and Universities (MICINN) and European Regional Development Fund project no. PGC2018-093332-B-I00 and Junta de Andaluc\'{\i}a grant FQM375.}



\begin{document}

\begin{abstract}
We prove, among other results, that three standard measures of weak non-compactness coincide in preduals of JBW$^*$-triples. This result is new even for preduals of von Neumann algebras. We further provide a characterization of JBW$^*$-triples with strongly WCG predual and describe the order of seminorms defining the strong$^*$ topology. As a byproduct we improve a characterization of weakly compact subsets of a JBW$^*$-triple predual, providing so a proof for a conjecture, open for almost eighteen years,  on weakly compact operators from a JB$^*$-triple into a complex Banach space.
\end{abstract}

\maketitle

\section{Introduction}

Measures of weak non-compactness are an important tool for a deeper understanding of weak compactness of sets and operators.
They are used to prove more precise versions of known results and to establish new results as well. As an illustration
we mention a fixed-point theorem \cite{deblasi}, quantitative versions of Krein's theorem \cite{f-krein,Gr-krein},
James' compactness theorem \cite{CKS,Gr-James}, Eberlein-\v{S}mulyan theorem \cite{AC-jmaa} and Gantmacher's theorem \cite{AC-meas}.\smallskip

There are several procedures how one can measure weak non-compactness of sets. There are two basic non-equivalent ways -- on
the one hand the De Blasi measure introduced and used in \cite{deblasi}, and on the other hand various mutually equivalent
quantities used in the other above-quoted papers. Their non-equivalence follows from \cite{tylli-cambridge,AC-meas}.
However, the counterexample witnessing their non-equivalence is an artificially constructed Banach space -- it is constructed
as the $\ell^1$-sum of suitable renormings of the space $\ell^1$. It seems to be still an open question whether there is
a `natural' Banach space where they fail to be equivalent.\smallskip

This question seems to be rather interesting as in many classical spaces all these measures of weak non-compactness are
equivalent. This applies, for example, to the Lebesgue spaces $L^1(\mu)$ for an arbitrary non-negative $\sigma$-additive measure $\mu$ \cite[Theorem 7.5]{qdpp}, including the special case $\ell^1(\Gamma)$ \cite[Proposition 7.3]{qdpp}, the space $c_0(\Gamma)$ \cite[Proposition 10.2]{qdpp}, the space of nuclear operators $N(\ell^q(\Lambda),\ell^p(J))$ for $p,q\in(1,\infty)$ \cite[Theorem 2.1]{HK-nuclear}, and preduals of atomic von Neumann algebras \cite[Theorem 2.2]{HK-nuclear}.\smallskip

In the present paper we prove the equivalence of the measures of weak non-compactness in preduals of JBW$^*$-triples.
This covers, in particular, the preduals of general von Neumann algebras. \smallskip

The machinery developed in this paper also have some important consequences in improving our current knowledge on relatively weakly compact subsets in the predual of a JBW$^*$-triple. More concretely, the result established in \cite{peralta2006some} asserts that a bounded subset, $K$, in the predual of a JBW$^*$-triple, $M$, is relatively weakly compact if and only if there is a couple of normal functionals $\varphi_1,\varphi_2$ in $M_*$  whose associated preHilbertian seminorm $\|\cdot\|_{\varphi_1,\varphi_2}$ controls uniformly the values of all
functionals in $K$ on the closed unit ball of $M$ (see sections \ref{sec:3}, \ref{sec:tripotents}, and \ref{sec:strong*} for definitions). As shown, for example in \cite[pages 340 and 341]{Cabrera-Rodriguez-vol2}, the existence of control seminorms $\|\cdot\|_{\varphi_1,\varphi_2}$ for relatively weakly compact subsets in $M_*$ can be applied to characterize weakly compact operators from a complex Banach space into the predual of a JBW$^*$-triple and from a JB$^*$-triple into a complex Banach space. It has been conjectured that, as in the case of von Neumann algebras (see \cite{akemann1967dual,jarchow1986weakly}), a single functional $\varphi$ in the predual, $M_*$, of a JBW$^*$-triple $M$  (and the associated seminorm $\|\cdot\|_{\varphi}$) is enough to control relatively weakly compact subsets in $M_*$ and to determine those weakly compact operators from a complex Banach space into $M_*$. It was even so claimed in \cite[Theorem 11]{chu1988weakly}. However, some subtle difficulties in the Barton-Friedman's proof for the Grothendieck inequality for JB$^*$-triples, also affected the arguments in \cite{chu1988weakly} (cf. section \ref{sec:11} or \cite[page 341]{Cabrera-Rodriguez-vol2}). In this paper we also provide a complete proof for this conjecture and we show the validity of the original statement in \cite[Theorem 11]{chu1988weakly}.\smallskip

The paper is organized as follows.\smallskip

In Section~\ref{sec:2} we give the definitions of the three measures of weak non-compact\-ness we will deal with and provide their basic properties.\smallskip

Section~\ref{sec:3} contains some background on JB$^*$-triples and, more specifically, on dual spaces among them, i.e., JBW$^*$-triples. We provide the basic notions and properties, including the relationship to the classical subclasses formed by C$^*$-algebras (respectively, von Neumann algebras) and JB$^*$-algebras (respectively, JBW$^*$-algebras).\smallskip

In Section~\ref{sec:main} we formulate our main result on the equivalence of the three measure of weak non-compactness in JBW$^*$-triple preduals  (see Theorem~\ref{T:main}). We further collect some consequences -- the special cases of the main result and the analogous result for real spaces.\smallskip

Sections~\ref{sec:ssp}, \ref{sec:tripotents}, \ref{sec:strong*} and~\ref{sec:8} are devoted to the proof of the main result. On the way to the proof we establish several results which seem to be of an independent interest. In Section~\ref{sec:ssp} we show that there is a close relationship between the equality of measures of weak non-compactness and the subsequence splitting property (see Proposition~\ref{p splitting}).\smallskip

In Section~\ref{sec:tripotents} we gather some properties of projections in C$^*$-algebras and JB$^*$-algebras and of tripotents in JB$^*$-triples. Most of the results presented there are known, but we point out that Lemmata~\ref{L:Jembed C*} and~\ref{L:Jembed JB*} and Proposition~\ref{P:M2 inclusion} seem to be of independent interest.\smallskip

In Section~\ref{sec:strong*} we investigate the relation between the strong$^*$ topology on a JBW$^*$-triple and the weakly compact subsets of its predual. The main achievements there are Propositions~\ref{P:seminorms order} and~\ref{P:wcompact cofinal}, where we explore the connections between the natural partial order in the set of tripotents with the order among the associated seminorms and the first results leading to the existence of control functionals.\smallskip

Section~\ref{sec:8} contains the culminating arguments of the proof of Theorem~\ref{T:main}.\smallskip

In Section~\ref{sec:9} we focus on $\sigma$-finite JBW$^*$-triples. We characterize those $\sigma$-finite JBW$^*$-triples whose predual is strongly WCG (see Theorem~\ref{T:SWCG}). This theorem reveals, in particular, a substantial
difference between JBW$^*$-triples and JBW$^*$-algebras, as the predual of a $\sigma$-finite JBW$^*$-algebra is always strongly WCG.\smallskip

In Section~\ref{sec:10} we apply the methods from Section~\ref{sec:9} to prove Theorem~\ref{T:order}
which provides a deep understanding of the structure of the strong$^*$ topology for general JBW$^*$-triples.\smallskip

In Section~\ref{sec:11}  we present the new advances on the characterization of relatively weakly compact subsets in the predual of a JBW$^*$-triple (see Theorem \ref{t weakly compact with a single control functional}), and the subsequent consequences determining the weakly compact operators from a JB$^*$-triple into a complex Banach space (cf. Theorem \ref{t weakly compact operators in the statement of ChuIochum}), which, as we have already commented, provides a proof for a conjecture considered during the last eighteen years.

The last section contains some open problems.

\section{Measures of weak noncompactness}\label{sec:2}

Let $X$ be a (real or complex) Banach space and $A,B\subset X$ two nonempty sets.
We set
$$\dh(A,B)=\sup\{\dist(a,B)\setsep a\in A\}.$$

The \emph{Hausdorff measure of norm non-compactness} is defined by the formula
$$\chi(A)=\inf\{\dh(A,F)\setsep F\subset X\mbox{ finite}\}=\inf\{\dh(A,K)\setsep K\subset X\mbox{ compact}\}$$
for a bounded set $A\subset X$. It is clear that $\chi(A)=0$ if and only if $A$ is relatively norm compact.\smallskip

The \emph{De Blasi measure of weak non-compactness} introduced in \cite{deblasi} is defined by
$$\omega(A)=\inf\{\dh(A,K)\setsep K\subset X\mbox{ weakly compact}\}.$$
(When confusion concerning the underlying space $X$ could arise, like for example in Lemma \ref{L:measures 1-complemented}, we will write $\omega_X$ instead of $\omega$.)
It is a natural modification of the Hausdorff measure of non-compactness. Further, $\omega(A)=0$ if and only if $A$ is relatively weakly compact. As remarked in \cite{deblasi} this was proved already by Grothendieck \cite[p. 401]{gro-book}
(compare also \cite[Lemma XIII.2]{Diestel_book}).\smallskip

Another measure of weak non-compactness inspired by the Banach-Alaoglu theorem is defined by
$$\wk{A}=\dh(\wscl{A},X).$$
Here $\wscl{A}$ is the closure of $A$ in the space $(X^{**},w^*)$, where $X$ is considered to be canonically embedded into its bidual.
It is clear that $A$ is relatively weakly compact if and only if $\wscl{A}\subset X$, i.e., if and only if $\wk{A}=0$.\smallskip

We will use one more quantity, namely
$$\wck{A}=\sup\{\dist(\clu{x_n},X)\setsep (x_n)\mbox{ is a sequence in }A\},$$
where $\clu{x_n}$ denotes the set of all weak$^*$-cluster points of the sequence $(x_n)$ in the bidual $X^{**}$. It follows easily from the Eberlein-\v{S}mulyan theorem that $\wck{A}=0$ whenever $A$ is relatively weakly compact. The converse follows from the quantitative version of the Eberlein-\v{S}mulyan theorem proven in \cite{AC-jmaa}. It consists in the inequalities
$$\wck{A}\le \wk{A}\le 2\wck{A}$$
which hold for any bounded subset $A\subset X$.
Further, the following inequalities are easy to check:
$$\wk{A}\le\omega(A)\le\chi(A).$$
In general, the quantities $\wk{\cdot}$ and $\omega(\cdot)$ (hence also $\wck{\cdot}$ and $\omega(\cdot)$) are not equivalent. As mentioned above, this was proved in \cite{tylli-cambridge,AC-meas} but it follows also from \cite[Theorem 2.3]{qdpp} using the fact that weakly compactly generated spaces are not stable under taking subspaces.\smallskip

If $Y$ is a closed subspace of $X$ and $A$ is a bounded subset of $Y$, then we can consider the measures of weak non-compactness of the set $A$ either in the space $X$ or in the space $Y$. In general it is not the same but we have
$$\wk{A}\le\wk[Y]{A}\le 2\wk{A},\quad
\wck{A}\le\wck[Y]{A}\le 2\wck{A}.$$
Indeed, in both cases the first inequality is trivial and the second one follows from \cite[Lemma 11]{Gr-krein}.\smallskip

Further, we clearly have $\omega_X(A)\le\omega_Y(A)$ but the converse inequality is, in general, not true, neither up to a constant (this follows, for example, from \cite[Theorem 2.3]{qdpp} with the same remark as above).\smallskip

However, if $Y$ is $1$-complemented in $X$, then the measures considered in $X$ and in $Y$ coincide. This is the content of the following easy lemma which is proved in \cite[Lemma 3.8(b)]{HK-nuclear}.

\begin{lemma}\label{L:measures 1-complemented}
Let $X$ be a Banach space, $Y\subset X$ a $1$-complemented subspace and $A\subset Y$ a bounded set.
Then
$$\wk{A}=\wk[Y]{A},\quad\wck{A}=\wck[Y]{A},\quad\omega_X(A)=\omega_Y(A).$$
\end{lemma}

\section{JB$^*$-triples and their subclasses}\label{sec:3}

Originated as mathematical models of physical observables in quantum mechanics by authors like W. Heisenberg, P. Jordan and J. von Neumann, C$^*$-algebras constitute nowadays an area of intensive research in functional analysis. The abstract definition says that a \emph{C$^*$-algebra} is a complex Banach algebra $A$ equipped with an algebra involution $^*$ satisfying the so-called \emph{Gelfand--Naimark axiom}, that is $\|a^* a\|= \| a\|^2$ for all $a\in A$. The celebrated Gelfand--Naimark theorem represents each abstract C$^*$-algebra as a norm-closed selfadjoint subalgebra of the space $B(H)$ of all continuous linear operators on a complex Hilbert space $H$.
%with two additional properties.
\smallskip

From the point of view of functional analysis, the class of C$^*$-algebras is not very stable for several properties. For example, a C$^*$-algebra is reflexive if and only if it is finite-dimensional, and C$^*$-algebras are not stable under contractive projections as we can find a contractive projection from $B(H)$ onto the Hilbert space $H$. As observed by Poincar{\'e} in 1907, for complex Banach spaces of dimension bigger than or equal to two, there exists a wide range of simply connected domains which are not biholomorphic to the unit ball. In other words, the Riemann mapping theorem cannot be easily generalized to arbitrary complex Banach spaces. JB$^*$-triples were introduced in 1983 by W. Kaup in a successful classification of bounded symmetric domains in complex Banach spaces of arbitrary dimension (see \cite{kaup1983riemann}). A \emph{JB$^*$-triple} is a complex Banach space $E$ together with a (continuous) triple product $\{.,.,.\}: E^3 \to E$, which is symmetric and bilinear in the outer variables and conjugate-linear in the middle one, satisfying the following algebraic--analytic properties:
\begin{enumerate}[$(JB^*$-$1)$]
\item $L(x,y)L(a,b) = L(L(x,y)(a),b) $ $- L(a,L(y,x)(b))$ $+ L(a,b) L(x,y) $ for all $a,b,x,y\in E$, where given $a,b\in E$, $L(a,b)$ stands for the (linear) operator on $E$ given by $L(a,b)(x)=\J abx$, for all $x\in E$ (Jordan identity);
\item The operator $L(a,a)$ is a hermitian operator with nonnegative spectrum for each $a\in E$;
\item $\|\{a,a,a\}\|=\|a\|^3$ for $a\in E$.
\end{enumerate}

Let us observe that the third axiom $(JB^*$-$3)$ is a Jordan-geometric analogue of the \emph{Gelfand--Naimark axiom}.\smallskip

The mapping $ (x,y,z)\mapsto \J xyz =\frac12 (x y^* z +z y^* x)$ can be applied to define a structure of JB$^*$-triple on a C$^*$-algebra $A$, or on the space $B(H,K)$ of all bounded linear operators between complex Hilbert spaces $H$ and $K$ and in particular on an arbitrary Hilbert space $H$ identified with $B(H,\mathbb{C})$, in all these cases we keep the original norm of the corresponding underlying Banach spaces.\smallskip

Some of the lackings exhibited by the class of C$^*$-algebras are no longer a handicap for JB$^*$-triples. For example, JB$^*$-triples are stable under contractive perturbations (see \cite{friedman1985solution,stacho1982projection,kaup1984contractive}). One of the most interesting properties of JB$^*$-triples is that a linear bijection between JB$^*$-triples is an isometry if and only if it is a triple isomorphism, that is, it preserves triple products (see \cite[Proposition 5.5]{kaup1983riemann}). In particular, if a complex Banach space $E$ admits two triple products under which $E$ is a JB$^*$-triple with respect to its original norm and any of these products, then both products coincide.\smallskip

An intermediate class between C$^*$-algebras and JB$^*$-triples is formed by JB$^*$-algebras. The hermitian part, $A_{sa}$, of a C$^*$-algebra $A$ is not, in general, closed under products. However $A_{sa}$ is a Jordan subalgebra of $A$ when the latter is considered with its natural Jordan product defined by $a\circ b = \frac12 (a b + ba)$.\smallskip

P. Jordan released the notion of Jordan algebra to set a mathematical model for the algebra of observables in quantum mechanics. In abstract algebra, a Jordan algebra is a nonassociative algebra over a field whose multiplication, denoted by $\circ$, is commutative and satisfies the so-called \emph{Jordan identity} $( x \circ y ) \circ x^2 = x\circ ( y\circ x^2 ).$ Given an element $a$ in a Jordan algebra $B$, we shall write $U_a$ for the linear mapping on $B$ defined by $U_a (b) := 2(a\circ b)\circ a - a^2\circ b$. A Jordan Banach algebra $B$ is a Jordan algebra equipped with a complete norm satisfying $\|a\circ b\|\leq \|a\| \cdot \|b\|$ for all $a,b\in B$. A complex Jordan Banach algebra $B$ admitting an involution $^*$ satisfying that $\| U_{a} (a^*) \| = \|a\|^3$ for all $a,b\in B$ is called a \emph{JB$^*$-algebra} (see \cite{youngson1978vidav}, \cite[Definition 3.3.1]{Cabrera-Rodriguez-vol1}). In some texts the definition of JB$^*$-algebras includes the extra axiom that the involution in a JB$^*$-algebra is an isometry (cf. \cite[\S 3.8]{hanche1984jordan}). The result in \cite[Lemma 4]{youngson1978vidav} (see also \cite[Proposition 3.3.13]{Cabrera-Rodriguez-vol1}) shows that this extra axiom is redundant. The self-adjoint part of a JB$^*$-algebra $B$ will be denoted by $B_{sa}.$ It is known that (real) \emph{JB-algebras} are precisely the self-adjoint parts of JB$^*$-algebras (compare \cite{Wright1977}). Any JB$^*$-algebra  also admits a structure of a JB$^*$-triple when equipped with the triple product defined by $\J xyz = (x\circ y^*) \circ z + (z\circ y^*)\circ x - (x\circ z)\circ y^*$ in which case $U_a (b)=\{a,b^*,a\}$.\smallskip

For a general overview of JB-algebras, JB$^*$-algebras and JB$^*$-triples the reader is referred to the monographs \cite{hanche1984jordan,chubook,Cabrera-Rodriguez-vol1,Cabrera-Rodriguez-vol2}.\smallskip

A \emph{JBW$^*$-triple} (respectively, a \emph{JBW$^*$-algebra}) is a JB$^*$-triple (respectively, a \emph{JB$^*$-algebra}) which is also a dual Banach space. JBW$^*$-triples can be considered as the Jordan alter-ego of von Neumann algebras. It should be noted that, as established by Sakai's theorem for von Neumann algebras, every JBW$^*$-triple admits a unique (isometric) predual and its product is separately weak$^*$-to-weak$^*$ continuous  \cite{BaTi} (see also \cite[Theorems 5.7.20 and 5.7.38]{Cabrera-Rodriguez-vol2} for new proofs). These facts rely on the fact proved in \cite{BaTi}
(see also \cite{CR}) asserting that the bidual of any JB$^*$-triple $E$ is a JB$^*$-triple under a triple product which extends that of $E$ and is separately weak$^*$-to-weak$^*$ continuous. It is further known that every triple isomorphism between JBW$^*$-triples is automatically weak$^*$-continuous (cf. \cite[Corollary 3.22]{Horn1987}).\smallskip

A \emph{JC$^*$-algebra} is a concrete JB$^*$-algebra which materializes as a norm-closed Jordan $*$-subalgebra of  a C$^*$-algebra, and hence a norm-closed Jordan self-adjoint subalgebra of some $B(H)$. A JW$^*$-algebra is a JC$^*$-algebra which is also a dual Banach space, or equivalently, a weak$^*$-closed JB$^*$-subalgebra of some von Neumann algebra. There are examples of JB$^*$-algebras which are not JC$^*$-algebras. Macdonald's and Shirshov-Cohn's theorems are useful tools to describe certain important subalgebras of JB$^*$-algebras (see \cite[Theorems 2.4.13 and 2.4.14]{hanche1984jordan}). For example, these structure results were originally applied by J.D.M. Wright in \cite{Wright1977} to deduce that the JB$^*$-subalgebra of a JB$^*$-algebra generated by two hermitian elements (and the unit element) is a JC$^*$-algebra. This result and one of its multiple consequences are gathered in the next lemma (which is a variant of \cite[Proposition 3.4.6]{Cabrera-Rodriguez-vol1}).

\begin{lemma}\label{L:JC*}
Let $B$ be a unital JB$^*$-algebra and $a\in B$ an arbitrary element.
\begin{enumerate}[$(i)$]
\item Let $N$ be the closed Jordan $*$-subalgebra of $B$ generated by $a$ and $1$. Then $N$ is Jordan (isometrically) $*$-isomorphic to a JB$^*$-subalgebra of some $B(H)$, where $H$ is a complex Hilbert space;
\item The element $a^*\circ a$ is positive in $N$. Moreover, $a^*\circ a=0$ if and only if $a=0$.
\end{enumerate}
\end{lemma}

\begin{proof}$(i)$ Note that $N$ coincides with the unital Jordan $*$-subalgebra generated by two self-adjoint elements $\frac12(a^*+a)$ and $\frac1{2i}(a-a^*)$. Hence the assertion follows immediately from \cite[Corollary 2.2]{Wright1977} (see also \cite[Proposition 3.4.6]{Cabrera-Rodriguez-vol1}). Statement $(ii)$ follows from $(i)$ since the assertion is clearly true in $B(H)$.
\end{proof}

The literature offers a generous collection of structure results for JB$^*$-triples. The Gelfand--Naimark theorem established by Y. Friedman and B. Russo proves that every JB$^*$-triple can be isometrically embedded as a JB$^*$-subtriple of an $\ell_{\infty}$-sum of Cartan factors (see \cite[Theorem 1]{Friedman-Russo-GN} and details there). A consequence of this fact shows that every JB$^*$-triple is isometrically JB$^*$-triple isomorphic to a JB$^*$-subtriple of a JB$^*$-algebra.\smallskip

Let us fix some additional notation. Given two von Neumann algebras $A\subset {B}(H)$ and $W\subset {B}(K)$, the algebraic tensor product $A\otimes W$ is canonically embedded into ${B}(H\otimes K)$, where $H\otimes K$ is the hilbertian tensor product of $H$ and $K$ (see \cite[Definition IV.1.2]{Tak}). Then the \emph{von Neumann tensor product} of $A$ and $W$ (denoted by $A\overline{\otimes} W,$) is precisely the von Neumann subalgebra generated by the algebraic tensor product $A\otimes W$ in ${B}(H\otimes K)$, that is, the weak$^*$ closure
of $A\otimes W$ in ${B}(H\otimes K)$ (see \cite[\S IV.5]{Tak}).\smallskip

Suppose $A$ is a commutative von Neumann algebra and $M$ is a JBW$^*$-subtriple of some $B(H)$. Following the standard notation, we shall write
$A \overline{\otimes} M$ for the weak$^*$-closure of the algebraic tensor product $A\otimes M$ in the usual von Neumann tensor product $A
\overline{\otimes} B(H)$ of $A$ and $B(H)$. Clearly $A \overline{\otimes} M$ is a JBW$^*$-subtriple of $A \overline{\otimes} B(H)$. It is well
known that every Cartan factor $C$ of type 1, 2, 3, or 4 is a JBW$^*$-subtriple of some $B(H)$ (compare \cite{horn1987classification} and
section \ref{sec:9}
for a more detailed presentation of Cartan factors), and thus, the von Neumann tensor product $A \overline{\otimes} C$
is a JBW$^*$-triple.\smallskip

The exceptional Cartan factors of type 5 and 6 can not be represented as JB$^*$-subtriples of some $B(H)$. Fortunately, these factors are all finite-dimensional. Henceforth, if $C$ denotes a finite-dimensional JB$^*$-triple, $A \overline{\otimes} C$ will stand for the injective tensor product of $A$ and $C$, which is clearly identified with the space $C(\Omega, C)$, of all continuous functions on $\Omega$ with values in $C$ endowed with the pointwise operations and the supremum norm, where $\Omega$ denotes the spectrum of $A$ (cf. \cite[p. 49]{ryan2013introduction}). This convention is consistent with the definitions in the previous paragraph, because if $C$ is a finite-dimensional Cartan factor which can be also embedded as a JBW$^*$-subtriple of some $B(H)$ both definitions above give the same object (cf. \cite[Theorem IV.4.14]{Tak}).\smallskip

In the setting of JBW$^*$-triples, a much more precise description than that derived from the Gelfand-Naimark theorem was found by G. Horn and E. Neher in \cite[(1.7)]{horn1987classification}, \cite[(1.20)]{horn1988classification}, where they proved that every JBW$^*$-triple $M$ writes  in the form
\begin{equation}\label{eq decomposition of JBW*-triples}
M = \left(\bigoplus_{j\in \mathcal{J}} A_j \overline{\otimes} C_j \right)_{\ell_{\infty}}
\oplus_{\ell_{\infty}} H(W,\alpha)\oplus_{\ell_{\infty}} pV,
\end{equation} where each $A_j$ is a commutative von Neumann algebra, each $C_j$ is a Cartan factor, $W$ and $V$ are continuous von Neumann algebras, $p$ is a projection in $V$, $\alpha$ is {a linear} involution on $W$ commuting with $^*$, that is, a linear $^*$-antiautomorphism of period 2 on $W$, and $H(W,\alpha)=\{x\in W: \alpha(x)=x\}$.

We conclude this section by the following result originated from \cite[Proposition 2]{chu1988weakly} and \cite[Theorem D.20]{Rodriguez94}.

\begin{prop}\label{P:triple 1-complemented}{\rm\cite[Proposition 5.10.137]{Cabrera-Rodriguez-vol2}}
Let $M$ be any JBW$^*$-triple. Then $M$ is triple-isomorphic to a weak$^*$-closed subtriple of a JBW$^*$-algebra which is $1$-complemented by a weak$^*$-to-weak$^*$ continuous projection.
\end{prop}

\section{Main results}\label{sec:main}

Our main result is the following theorem on coincidence of
measures of weak non-compactness in preduals of JBW$^*$-triples. It is proved at the end of Section \ref{sec:8} below.
%after several technical results and arguments in page
%\pageref{eq proof of Theorem 4.1} below.

\begin{thm}\label{T:main}
Let $M$ be a  JBW$^*$-triple and let $A\subset M_*$ be a bounded set. Then
$\omega(A)=\wk[M_*]{A}=\wck[M_*]{A}$.
\end{thm}

Since JBW$^*$-algebras and von Neumann algebras are special cases of JBW$^*$-triples, the following two corollaries are immediate.

\begin{cor}\label{c:main JBW}
 Let $M$ be a JBW$^*$-algebra and let $A\subset M_*$ be a bounded set. Then
$\omega(A)=\wk[M_*]{A}=\wck[M_*]{A}$.
\end{cor}

\begin{cor}
 Let $M$ be a von Neuman algebra and let $A\subset M_*$ be a bounded set. Then
$\omega(A)=\wk[M_*]{A}=\wck[M_*]{A}$.
\end{cor}

Note that Corollary~\ref{c:main JBW} easily implies the main result Theorem~\ref{T:main} using Proposition~\ref{P:triple 1-complemented} and Lemma~\ref{L:measures 1-complemented}. This is in fact the way the main result will be proved in Section~\ref{sec:8}. \smallskip

We further easily obtain the same results for real variants of the respective structures.

\begin{cor}
Let $M$ be a real JBW$^*$-triple and let $A\subset M_*$ be a bounded set. Then
$\omega(A)=\wk[M_*]{A}=\wck[M_*]{A}$.
\end{cor}

\begin{proof} We use the following arguments explained in detail in \cite[Section 6]{BHKPP-triples}.\smallskip

Let $M$ be a real JBW$^*$-triple. Then there is a JBW$^*$-triple $N$ and a weak$^*$-to-weak$^*$ continuous
conjugate-linear isometry $\tau:N\to N$ such that $\tau^2=\operatorname{id}_N$ and
$$M=\{x\in N\setsep x=\tau(x)\}.$$
For each $\varphi\in N_*$ define
$$\tau^\#(\varphi)(x)=\overline{\varphi(\tau(x))},\quad x\in N.$$
Then $\tau^\#$ is a conjugate-linear isometry of $N_*$ onto $N_*$ such that $(\tau^\#)^2=\operatorname{id}_{N_*}$.
Set
$$N_*^\tau=\{\varphi\in N_*\setsep \tau^\#(\varphi)=\varphi\}.$$
Then $N_*^\tau$ is a closed real-linear subspace of $N_*$ and the mapping $\varphi\mapsto\Re\varphi$ is a real-linear isometry of $N_*^\tau$ onto $M_*$.\smallskip

Hence the formula $$P(\varphi)=\frac12(\varphi+\tau^\#(\varphi)),\quad \varphi\in N_*$$ defines a real-linear norm-one projection of $N_*$ onto $N_*^\tau$.\smallskip

If $X$ is a complex Banach space, denote by $X_R$ the underlying real space.\smallskip

Now, by Theorem~\ref{T:main} we know that the measures of weak non-compactness coincide for subsets of $N_*$.
By the discussion at the end of Section 2.5 of \cite{qdpp} the measures of weak non-compactness with respect to $N_*$ coincide with those with respect for $(N_*)_R$. Finally, $N_*^\tau$ is $1$-complemented in $(N_*)_R$, hence Lemma~\ref{L:measures 1-complemented} yields that the measures of weak non-compactness coincide for subsets of $N_*^\tau$. Since this space is isometric to $M_*$, the proof is complete.
\end{proof}

Since JBW-algebras are among the examples of a real JBW$^*$-triple, the following corollary follows immediately.

\begin{cor}
Let $M$ be a JBW-algebra and let $A\subset M_*$ be a bounded set. Then
$\omega(A)=\wk[M_*]{A}=\wck[M_*]{A}$.
\end{cor}

The following corollary also is a special case of our previous result, since the self-adjoint part of a von Neumann algebra $M$ is a real JBW$^*$-triple and its predual is formed exactly by the self-adjoint elements of $M_*$.

\begin{cor}
 Let $M$ be a von Neuman algebra and let $A\subset M_{*,sa}$  (where $M_{*,sa}$ denotes the subspace of $M_*$ formed by self-adjoint functionals) be a bounded set. Then
$\omega(A)=\wk[M_{*,sa}]{A}=\wck[M_{*,sa}]{A}$.
\end{cor}

\section{Subsequence splitting property}\label{sec:ssp}

In \cite[Proposition 7.1]{qdpp} a special case of Theorem~\ref{T:main} was proved -- the equality of measures of weak non-compactness for subsets of $L^1(\mu)$ for a finite measure $\mu$.
A key role in the proof is played by a variant of Rosenthal's subsequence splitting lemma.
This is a motivation to define the following
property of Banach spaces which turns out to be closely connected with the equality of measures of weak non-compactness.\smallskip

Let $X$ be a Banach space. We say that $X$ has
\begin{enumerate}[$\bullet$]
    \item the \emph{subsequence splitting property} if for any bounded sequence $(x_n)$ in $X$ and any $\varepsilon>0$ there is a subsequence $(x_{n_k})$
    which can be expressed as $x_{n_k}=y_k+z_k$,
    where $(y_k)$ is weakly convergent in $X$ and
    \begin{equation}
    \norm{\sum_{j=1}^n \alpha_j z_j}\ge (1-\varepsilon)\sum_{j=1}^n \abs{\alpha_j}\norm{z_j}\label{eq def splitting}
    \end{equation}
    for any $n\in\en$ and any choice of scalars $\alpha_1,\dots,\alpha_n$;
    \item the \emph{isometric subsequence splitting property} if for any bounded sequence $(x_n)$ in $X$ there is a subsequence $(x_{n_k})$
    which can be expressed as $x_{n_k}=y_k+z_k$,
    where $(y_k)$ is weakly convergent in $X$ and
    $$\norm{\sum_{j=1}^n \alpha_j z_j}=\sum_{j=1}^n \abs{\alpha_j}\norm{z_j}$$
    for any $n\in\en$ and any choice of scalars $\alpha_1,\dots,\alpha_n$.
\end{enumerate}

That $L^1([0,1])$ has the isometric subsequence splitting property has been mentioned by Bourgain and Rosenthal in \cite{BourgRos}
but a forerunner can be found in \cite{KadPel}. There are also splitting versions for $L^p$-spaces, $0<p<\infty$ but in this paper
we naturally restrict to the case $p=1$ because preduals of JBW$^*$-triples are generalizations of (non-commutative)
$L^1$-spaces.
The generalization of the isometric subsequence splitting property to preduals of von Neumann algebras was proved by
Randrianantoanina \cite{Randri2002} and, by different methods, by Raynaud and Xu \cite{Raynaud-Xu}.
Then Fern\'{a}ndez-Polo, Ram\'{i}rez and the third mentioned author \cite{FePeRa2015} showed the isometric subsequence splitting
property for preduals of JBW$^*$-algebras and the two last named authors for JBW$^*$-triples \cite{Pe-Pfi-splitting}.\smallskip

The predual of a JBW$^*$-triple is L-embedded (see \cite[Theorem 5.7.36]{Cabrera-Rodriguez-vol2}
or \cite{Horn1987}, \cite{BaTi}).
Recall that a Banach space $X$ is called L-embedded if there is a projection $P$ on $X^{**}$ with image $X$ such that
$\norm{x^{**}}=\norm{Px^{**}}+\norm{x^{**}-Px^{**}}$ for all $x^{**}\in X^{**}$; the standard reference for
L-embedded spaces is \cite[Chapter\ IV]{hawewe}.
An L-embedded Banach space $X$ is weakly sequentially complete \cite[Theorem IV.2.2]{hawewe}.
Hence if a bounded sequence $(x_n)$ of $X$ does not contain
any $\ell^1$-subsequence then by Rosenthal's theorem \cite[Theorem 5.37]{fhhmpz} it contains a weak Cauchy subsequence $(x_{n_k})$
and in this case the splitting property is trivially satisfied with $y_k=x_{n_k}$ and $z_k=0$.
The interesting case appears when $(x_n)$ contains an $\ell^1$-sequence.
This explains why we use Lemma \ref{l weak perturb} which says that the quality of an $\ell^1$-sequence (i.e., its James constant, see below)
remains invariant, up to a subsequence, under a perturbation by a weak Cauchy sequence.\smallskip

The aim of this section is to show the following proposition whose main purpose for this paper is part $(b)$.

\begin{prop}\label{p splitting}
Let $X$ be a Banach space. Consider the following possible properties of $X$.
\begin{itemize}
    \item[(I)] For any bounded set $A\subset X$ we have $\wck{A}=\omega(A)$.
    \item[(II)] For any bounded set $A\subset X$ and any $\ep>0$ there is a countable subset $C\subset A$ such that
    $$\omega(C')>\omega(A)-\ep\mbox{ for any infinite }C'\subset C.$$
\end{itemize}
Then the following assertions are true.
\begin{enumerate}[$(a)$]
    \item (I)$\Rightarrow$(II);
    \item If $X$ has the subsequence splitting property, then (II)$\Rightarrow$(I);
    \item If $X$ is L-embedded and enjoys (I), then $X$ has the subsequence splitting property.
\end{enumerate}
\end{prop}

\noindent
Before we pass to the proof of Proposition \ref{p splitting} we recall some known facts on $\ell^1$-sequences.\smallskip

To a bounded sequence $(x_{n})$ in a Banach space $X$ we associate its \textquoteleft \emph{James constant}\textquoteright
\begin{equation}\label{eq:c_J def}
c_J (x_{n}) =\sup c_m \quad\mbox{ where the }\quad
c_m=\inf\left\{\norm{\sum_{n\geq m} \alpha_n x_n}\setsep \sum_{n\geq m}\betr{\alpha_n} =1\right\}
\end{equation}
form an increasing sequence.
If $(x_{n})$ is equivalent to the canonical basis of $\ell^1$ then $c_J(x_{n})>0$
and more specifically, $c_J(x_{n})>0$ if and only if there is an integer $m$ such that $(x_{n})_{n\geq m}$ is
equivalent to the canonical basis of $\ell^1$.
The number $c_J (x_{n})$ may be thought of as the approximately best $l^1$-basis constant of $(x_{n})$
in the sense that for each $\ep>0$ there is $m\in\en$ such that
$\norm{\sum_{n=m}^{\infty} \alpha_n x_n}\geq (1-\ep)c_J(x_n)\sum_{n=m}^{\infty}\betr{\alpha_n}$ for all $(\alpha_n)\in \ell^1$,
and $c_J (x_{n})$ cannot be replaced by a strictly greater constant.\smallskip

Further, a sequence $(z_l)$ will be called a \emph{block sequence} of $(x_n)$ if there
are successive finite sets $A_l\subset\en$ (i.e., $\max A_l<\min A_{l+1}$ for $l\in\en$) and a sequence of scalars $(\lambda_n)$ such that for each $l\in\en$ we have
$$\sum_{k\in A_l}\betr{\lambda_k}=1\mbox{ and }z_l=\sum_{k\in A_l}\lambda_k x_k.$$

\begin{lemma}\label{L:distortion}
Let $(x_n)$ be a bounded sequence in a Banach space $X$ such that $c_J(x_n)>0$. Then there is a block sequence $(z_l)$ of $(x_n)$ such that
$\norm{z_l}\to c_J(x_n)$ as $l\to\infty$ and, moreover,
\begin{equation}\label{gl-alm-isom}
(1-2^{-m})c_J (x_{n})\sum_{l=m}^{\infty}\betr{\alpha_l}
\leq\Norm{\sum_{l=m}^{\infty} \alpha_l {z_l}}
\leq(1+2^{-m})c_J (x_{n})\sum_{l=m}^{\infty}\betr{\alpha_l}
\end{equation}
for all $m\in\en$ and all $(\alpha_n)\in \ell^1$.
\end{lemma}

\begin{proof}
Set $c=c_J(x_n)$ and let $c_m$ have the meaning from \eqref{eq:c_J def}. Since $c_m\nearrow c$, we can fix  an increasing sequence $(k_m)$ of natural numbers such that $c_{k_m}>c(1-2^{-m})$.\smallskip

Next, having in mind that $c_k$ is defined as an infimum and  $c_k\le c$ for each $k\in\en$, we can find finite sets $A_m\subset \en$ and constants $\lambda_n$, $n\in A_m$, such that for each $m\in\en$ we have
\begin{itemize}
    \item $k_m\le\min A_m\le\max A_m<\min A_{m+1}$,
    \item $\sum_{n\in A_m}\abs{\lambda_n}=1$,
    \item $\norm{\sum_{n\in A_m}\lambda_n x_n}<c(1+2^{-m}).$
\end{itemize}
Set $z_m=\sum_{n\in A_m}\lambda_n x_n$. Then clearly $(z_m)$ is a block sequence of $(x_n)$ with $\norm{z_m}\to c$ as $m\to\infty$.
Moreover, fix any sequence $(\alpha_l)\in\ell^1$ and $m\in\en$.
Then
$$\norm{\sum_{l=m}^{\infty} \alpha_l {z_l}}
\le \sum_{l=m}^\infty \abs{\alpha_l}\norm{z_l}\le (1+2^{-m})c \sum_{l=m}^\infty \abs{\alpha_l},$$
and
$$\begin{aligned}
\norm{\sum_{l=m}^{\infty} \alpha_l {z_l}}&=
\norm{\sum_{l=m}^\infty\sum_{n\in A_l}\alpha_l \lambda_n x_n}\ge (1-2^{-m})c \sum_{l=m}^\infty\sum_{n\in A_l}\abs{\alpha_l}\cdot\abs{\lambda_n}
\\&=(1-2^{-m})c\sum_{l=m}^{\infty}\betr{\alpha_l},
\end{aligned}$$
which completes the proof.
\end{proof}

It is elementary but useful that if one passes to a subsequence $(x_{n_k})$ of $(x_n)$ then $c_J(x_{n_k})\geq c_J(x_n)$;
in particular it makes sense to define
\begin{equation}
\tilde{c}_J(x_n)=\sup_{n_k}c_J(x_{n_k}).\label{eq def tilde c}
\end{equation}
A diagonal argument shows that every bounded sequence $(x_n)$ admits a subsequence $(x_{n_k})$ which is $c_J$-stable
in the sense that $c_J(x_{n_k})=\tilde c_J(x_{n_k})$. If one passes to a block sequence $(z_n)$ of $(x_n)$ then $c_J(z_n)\geq c_J(x_n)$.\

\begin{lemma}\label{l weak perturb}
Let $(x_n)$ be a bounded sequence in a Banach space $X$. Let $(y_n)$ be a weak Cauchy sequence in $X$.
Then for each $\ep\in(0,1)$ there is a subsequence $(n_k)$ such that $(x_{n_k})$ and $(x_{n_k}+y_{n_k})$ are $c_J$-stable and
\begin{eqnarray}
(1-\ep)\tilde c_J(x_n)\le c_J(x_{n_k})\le\tilde c_J(x_n).\label{eq perturb1}\\
c_J(x_{n_k}+y_{n_k})=c_J(x_{n_k}).\label{eq perturb2}
\end{eqnarray}
\end{lemma}

\begin{proof}
Let $(x_n)$ be a bounded sequence. If $\tilde c_J(x_n)=0$ then by Rosenthal's $\ell^1$-theorem
$(x_n)$ contains a weak Cauchy subsequence which remains
weak Cauchy when a weak Cauchy sequence is added; hence the $c_J$-value of the sum is still $0$ and we are done
in this case.\smallskip

Assume now that $\tilde c_J(x_n)>0$. Let $(y_n)$ be weakly Cauchy and let $0<\ep<1$. Set $z_n=x_n+y_n$.
Choose a subsequence $(x_{n_k})$ such that \eqref{eq perturb1} holds.
Passing to another subsequence, if necessary, we assume further that $(x_{n_k})$ and $(z_{n_k})$ are $c_J$-stable.
Write $c=c_J(x_{n_k})$ for short.
Note that $c\ge(1-\ep)\tilde c_J(x_n)>0$. In particular we may assume that $(x_{n_k})$ is an $\ell^1$-sequence (by omitting, if necessary, finitely many $x_{n_k}$).
Passing to another subsequence again, if necessary, we also get that $c_J(z_{n_k})>0$ because otherwise $(z_{n_k})$ would have no $\ell^1$-subsequence and would therefore contain
a weak Cauchy subsequence $(z_{n_{k_l}})$ by Rosenthal's theorem in which case $x_{n_{k_l}}=z_{n_{k_l}}-y_{n_{k_l}}$ would form a weak Cauchy sequence which is not possible.
By Lemma~\ref{L:distortion} take blocks $x_n^{(1)}$ of the sequence $(x_{n_k})$ such that
\begin{equation}
(1-2^{-m})c\sum_{l=m}^{\infty}\betr{\alpha_l}\leq
\Norm{\sum_{l=m}^{\infty} \alpha_l x_l^{(1)}}   \leq(1+2^{-m})c\sum_{l=m}^{\infty}\betr{\alpha_l}
\label{eq-alm-isom}
\end{equation}
for all $m\in\en$ and $(\alpha_l)\in\ell^1$.
For the corresponding blocks $z_n^{(1)}$ we have $c_J(z_n^{(1)})\ge c_J(z_{n_k})>0$.
By Lemma~\ref{L:distortion} take blocks $z_n^{(2)}$ of  $(z_n^{(1)})$ such that
\begin{equation}
(1-2^{-m})c_J(z_n^{(1)})\sum_{l=m}^{\infty}\betr{\alpha_l}\leq
\Norm{\sum_{l=m}^{\infty} \alpha_l z_l^{(2)}}   \leq(1+2^{-m})c_J(z_n^{(1)})\sum_{l=m}^{\infty}\betr{\alpha_l}
\label{eq-alm-isom2}
\end{equation}
for all $m\in\en$ and $(\alpha_l)\in\ell^1$.
Note that every block sequence of $(y_n)$ admits a weak Cauchy subsequence [for, if $u_n=\sum_{k\in A_n}\lambda_ky_k$,
(with finite pairwise disjoint $A_n$, $\sum_{k\in A_n}\betr{\lambda_k}=1$) and if $(u_{n_m})$ is a subsequence such that
$\lambda=\lim_m\sum_{k\in A_{n_m}}\lambda_k$ exists then, given $x^*\in X^*$, the sequence $x^*(u_{n_m})$ converges to $\lambda\mu$
where $\mu=\lim_nx^*(y_n)$:
$\betr{x^*(u_{n_m})-\lambda\mu}
=\betr{\sum_{k\in A_{n_m}}\lambda_k(x^*(y_k)-\mu) + ((\sum_{k\in A_{n_m}}\lambda_k)-\lambda)\mu}
\le \max_{k\in A_{n_m}}\betr{x^*(y_k)-\mu}+\betr{(\sum_{k\in A_{n_m}}\lambda_k)-\lambda}\betr{\mu}\to0$ as $m\to\infty]$.
Hence the blocks $y_n^{(2)}$ which correspond to  $z_n^{(2)}$ admit a weak Cauchy subsequence $(y_{n_k}^{(2)})$
whose differences $\frac12(y_{n_{2k+1}}^{(2)}-y_{n_{2k}}^{(2)})$ are weakly null and, by the Mazur theorem,
there are blocks $y_n^{(3)}$ of $(y_n^{(2)})$ such that $\norm{y_n^{(3)}}\to0$.
Note that for the corresponding blocks $x_n^{(3)}$ (which are blocks of the $x_n^{(1)}$)
we have $\norm{x_n^{(3)}}\to c$ by \eqref{eq-alm-isom}.
Hence the norm of the right hand side in
$$z_n^{(3)}=x_n^{(3)}+y_n^{(3)}$$
converges to $c$ while the norm of the left hand side converges to $c_J(z_n^{(1)})$ by \eqref{eq-alm-isom2}.
Thus $c=c_J(z_n^{(1)})$.
This shows that
$c_J(x_{n_k})=c_J(z_n^{(1)})\ge c_J(x_{n_k}+y_{n_k})$.\smallskip

Apply now what has been shown so far to the bounded sequence $(z_{n_k})$ and the weak Cauchy sequence $(-y_{n_k})$
in order to obtain subsequences $(z_{n_{k_l}})$ and $(y_{n_{k_l}})$ such that,
by $c_J$-stability of $(x_{n_k})$ and $(z_{n_k})$,
$$c_J(x_{n_k})
=c_J(x_{n_{k_l}})=c_J(z_{n_{k_l}}+(-y_{n_{k_l}}))\le c_J(z_{n_{k_l}})=c_J(z_{n_k}).$$
Now \eqref{eq perturb2} follows.
\end{proof}

\smallskip

\begin{proof}[Proof of Proposition \ref{p splitting}]~\smallskip
$(a)$ The proof of the implication (I)$\Rightarrow$(II) is almost immediate from the definition of $\operatorname{wck}_X$.
For, given $\ep>0$, it is enough to set $C=\{x_n\}$ where $(x_n)$ is a sequence in $A$ such that $\operatorname{dist}(x^{**},X)>\wck[X]{A}-\ep$
for all weak$^*$-cluster points $x^{**}$ of $(x_n)$ which leads to $\omega(C')>\wck[X]{A}-\ep=\omega(A)-\ep$ for all infinite
$C'\subset C$.\smallskip

$(b)$ Suppose that $X$ has the subsequence splitting property
and (II) is satisfied. Let $A$ be bounded in $X$. We need to prove that $\wck A\ge\omega(A)$. If $\omega(A)=0$, the inequality is trivial.
So, assume $\omega(A)>0$ and fix an arbitrary $\varepsilon\in(0,\min\{1,\omega(A)\})$.
Let $C=\{x_n\}$ be a countable subset of $A$ provided by (II).
In particular, $\omega(\{x_{n_k}\})>0$ for every subsequence $(x_{n_k})$ of $(x_n)$.\smallskip

Take a subsequence $(x_{n_k})$ of $(x_n)$ such that $\tilde c_J(x_n)-\ep<c_J(x_{n_k})\le\tilde c_J(x_n)$.
Take another subsequence (still denoted by $(x_{n_k})$) and write $x_{n_k}=y_k+z_k$ according to the subsequence splitting property
where $(y_k)$ converges weakly and $\norm{\sum\alpha_kz_k}\ge(1-\ep/2)\sum\betr{\alpha_k}\norm{z_k}$
for all $(\alpha_k)\in\ell^1$.
Note that the sequence $(z_k)$ is bounded, so without loss of generality $\lambda:=\lim\|z_k\|$ exists.
We have that $\lambda>0$ because otherwise $(x_{n_k})$ would be weakly convergent which would contradict $\omega(\{x_{n_k}\})>0$.
Hence, without loss of generality $\lambda(1-\ep/2)<\norm{z_k}<\lambda(1+\varepsilon)$ for $k\in\en$.
Since the sequence $(y_k)$ forms a relatively weakly compact set,
by the assumptions we have
$$\omega(A)-\varepsilon<\omega(\{x_{n_k},k\in\en\})\le \sup_{k}\norm{z_k}\le\lambda(1+\varepsilon).$$
Moreover, $c_J(z_k)\ge(1-\ep/2)\lambda(1-\ep/2)>(1-\ep)\lambda$.\smallskip

We apply Lemma \ref{l weak perturb} and obtain further subsequences (still denoted by $(x_{n_k})$, $(y_k)$, $(z_k)$). Then \eqref{eq perturb2} yields $$c_J(x_{n_k})=c_J(z_k)\ge\lambda(1-\varepsilon)\ge\frac{1-\varepsilon}{1+\varepsilon}(\omega(A)-\varepsilon).$$
By \cite[Lemma 5(ii)]{wesecom} $\dist(\clu{x_{n_k}},X)\ge c_J(x_{n_k})$, hence
$$\wck{A}\ge \dist(\clu{x_{n_k}},X)\ge c_J(x_{n_k})\ge\frac{1-\varepsilon}{1+\varepsilon}(\omega(A)-\varepsilon).$$
Since $\ep$ is arbitrary we are done.\smallskip

$(c)$ Assume $X$ is $L$-embedded and satisfies (I).
Fix a bounded sequence $(x_n)$ in $X$.
If $\tilde c_J(x_n)=0$ then by Rosenthal's $\ell^1$-theorem, $(x_n)$ contains a weak Cauchy subsequence $(x_{n_k})$ which,
by weak sequential completeness of L-embedded spaces \cite[Theorem IV.2.2]{hawewe} converges weakly.
The case is settled by putting $y_k=x_{n_k}$ and $z_k=0$.\smallskip

Let us now suppose that $\tilde c:=\tilde c_J(x_n)>0$.
In order to prove the subsequence splitting property in this case it is enough to produce, given $\varepsilon>0$, a decomposition $x_{n_k}=y_k+z_k$
where $(y_k)$ converges weakly and where
\begin{eqnarray}
(1-\ep)\tilde c\sum\betr{\alpha_k}&\le&\Norm{\sum\alpha_kz_k}\le (1+\ep)\tilde c\sum\betr{\alpha_k}
\label{eq converse1}
\end{eqnarray}
for all $(\alpha_k)\in\ell^1$.
Set $A=\{x_n; n\in\en\}$.\\
First we claim that $\wck{A}=\tilde c$. The inequality ``$\ge$'' follows from \cite[Lemma 5(ii)]{wesecom}.
For the other inequality note that for any $\eta>0$, the construction (which works also for complex scalars) leading to
\cite[formula (8)]{wesecom} yields an $x\in X$ and a subsequence $(x_{n_k})$ such that $c_J(x_{n_k}-x)>(1-\eta)\wck{A}$. (Note that here the assumption that $X$ is $L$-embedded is used).
It follows $c_J(x_{n_k})>(1-\eta)\wck{A}$ (cf. \cite[Proposition 4.2]{KnaustOdell-1989}
or Lemma \ref{l weak perturb} for constant $y_n=-x$).
This proves ``$\le$'' and the claim.\\
Let  $(x_{n_k})$ be a subsequence such that $c_J(x_{n_k})>(1-\frac{\ep}{2})\tilde c$.
Since $\omega(A)=\tilde c$ by (I) and the claim, there is a weakly compact set $K$ of $X$ such that $A\subset K+ (1+\ep)\tilde cB_X$.
Choose a sequence $(y_k)$ in $K$ (which can be supposed to converge weakly) such that $x_{n_k}=y_k+z_k$
with $\norm{z_k}\le(1+\ep)\tilde c$. The latter inequality implies the second one of \eqref{eq converse1}.
By Lemma \ref{l weak perturb} we pass to subsequences (still denoted by $(x_{n_k})$, $(y_k)$, $(z_k)$) in order to get
$c_J(z_k)=c_J(x_{n_k})$.
Now, by the definition of $c_J(z_k)$, if we omit at most finitely many $z_k$ then we have
$$\Norm{\sum\alpha_kz_k}\ge(1-\frac{\ep}{2})c_J(z_k)\sum\betr{\alpha_k}
\ge(1-\frac{\ep}{2})^2\tilde c\sum\betr{\alpha_k}$$
and the first inequality of \eqref{eq converse1} follows.
\end{proof}

Proposition~\ref{p splitting}$(b)$ will play a key role in the proof of Theorem~\ref{T:main} in Section~\ref{sec:8}. Let us remark that some special cases may be proved already now, as there are some spaces which are easily seen to satisfy (II).\smallskip

Recall that a Banach space $X$ is said to be \emph{weakly compactly generated} (shortly \emph{WCG}) if there is a weakly compact subset $K\subset X$ with $\overline{\span} K=X$. Further, $X$ is called \emph{strongly WCG} (see \cite{SWCG}) if there is a weakly compact set $K\subset X$ such that
$$\forall L\subset X\mbox{ weakly compact}\,\forall\varepsilon>0\,\exists n\in\en: L\subset nK+\varepsilon B_X.$$
By the Krein theorem we may assume without loss of generality that $K$ is absolutely convex.\\
We have the following easy lemma

\begin{lemma}
Let $X$ be a strongly WCG Banach space. Then for any bounded set $A\subset X$ there is a countable subset $C\subset A$ such that
$\omega(C')=\omega(A)$ for each $C'\subset C$ infinite.
\end{lemma}

\begin{proof} Let $K\subset X$ be a weakly compact set witnessing the strong WCG property. Then clearly
$$\omega(A)=\inf_{n\in\en} \dh(A,nK)$$
for any bounded set $A\subset X$. Given any bounded set $A\subset X$, we can find a sequence $(x_n)$ in $A$ such that $\dist(x_n,nK)>\omega(A)-\frac1n$. It is enough to take $C=\{x_n\setsep n\in\en\}$.
\end{proof}

Since $L^1(\mu)$ is strongly WCG for any finite measure $\mu$ (by \cite[Example 2.3(d)]{SWCG}) and $L^1$ spaces have the subsequence splitting property, then the essential part of \cite[Proposition 7.1]{qdpp} follows immediately from the previous lemma and Proposition~\ref{p splitting}. \smallskip

Further, preduals of $\sigma$-finite von Neumann algebras and, more generally, preduals of $\sigma$-finite JBW$^*$-algebras are seen to be strongly WCG
by combining \cite[Theorem 2.1]{SWCG}, \cite[Appendice 6, Lemma 2]{iochumbook} (cf. Lemma~\ref{L:JBW*algebra sigma finite} below), and \cite[Theorem D.21]{Rodriguez94} (see Proposition~\ref{P:strong*=Mackey} below).
Hence the validity of Theorem~\ref{T:main} for $\sigma$-finite JBW$^*$-algebras easily follows.
The case of $\sigma$-finite JBW$^*$-triples is more complicated, see Theorem~\ref{T:SWCG} below.

\section{Tripotents and projections}\label{sec:tripotents}

A \emph{projection} in a $C^*$-algebra $A$ is a self-adjoint idempotent, i.e., an element $p\in A$ satisfying $p^*=p=p^2$. If $A$ is represented as a $C^*$-subalgebra of $B(H)$, then $p\in A$ is a projection if and only if it is, when viewed as an operator on $H$,
an orthogonal projection.\smallskip

Similarly, a \emph{projection} in a JB$^*$-algebra $A$ is an element $p\in A$ satisfying $p^*=p$ and $p\circ p=p$; in particular, $p$ is positive (by Lemma \ref{L:JC*}).
Two projections $p,q\in A$ are said to be \emph{orthogonal} if $p\circ q=0$ or, equivalently, if $p+q$ is also a projection. In case $A$ is a $C^*$-algebra, this is just equivalent to $pq=0$, which means that the ranges of the projections are orthogonal subspaces of $H$.\smallskip

We shall consider the usual partial order on the set of projections in a JB$^*$-algebra defined by $p\le q$ if $q-p$ is a projection. In a $C^*$-algebra this order coincides with the standard one.\smallskip

In a JB$^*$-triple there is no natural notion of projections, but tripotents play a similar role.
As a motivation for the latter notion, suppose $A$ is a C$^*$-algebra regarded as a
JB$^*$-triple with respect to the triple product $\{a,b,c\}= \frac12 (a b^* c + c b^* a)$.
It is well known that an element $u$ in $A$ is a partial isometry (i.e., $u^*u$ is a projection,
or equivalently, $u u^*$ is a projection) if and only if $\{u,u,u\}= uu^*u= u$.
Given a partial isometry $u$ in a C$^*$-algebra $A$, the elements $p_i=u^*u$ and $p_f=uu^*$ are called the initial and final projection of $u$, respectively.
An element $u$ of a JB$^*$-triple $M$ is called a \emph{tripotent} if $\J uuu=u$. If $M$ is
a JB$^*$-algebra, then
$u\in M$ is a tripotent if and only if
$$ u = 2 (u\circ u^*)\circ u - (u\circ u)\circ u^*,$$
which cannot be simplified. Note that the projections of a JB$^*$-algebra $A$ are precisely those tripotents in $A$ which are positive elements.

\subsection{Peirce decomposition}

From a purely algebraic point of view, a complex linear space $E$ equipped with a triple product $\{.,.,.\}: E^3 \to E$, which is symmetric and bilinear in the outer variables and conjugate-linear in the middle one satisfying the axiom $(JB^*$-$1)$ in the definition of JB$^*$-triple is called a complex Jordan triple system.
An element $u$ in a Jordan triple system $E$ is a tripotent if $\{u,u,u\}=u$. Each tripotent $u$ in a complex Jordan triple system $E$ induces a decomposition of $E$ in terms of the eigenspaces of the mapping $L(e,e)$
(this purely algebraic result can be seen, for example, in \cite[1.3 in page 7]{neher1987jordan} or \cite[page 32]{chubook}).\smallskip

In our concrete setting, the algebraic structure assures that for each tripotent $u$ in a JB$^*$-triple $M$, the eigenvalues of the operator $L(u,u)$ are contained in the set $\{0,\frac12,1\}$. If $u\ne0$, then $1$ is always an eigenvalue, the witnessing eigenvector is $u$. For $j=0,1,2$ we shall denote by $M_j(u)$ the eigenspace of $M$ with respect to the eigenvalue $\frac{j}{2}$. Then $M$ is the direct sum $M_0(u)\oplus M_1(u)\oplus M_2(u)$  of the three eigenspaces of $L(u,u)$; this decomposition is called the \emph{Peirce decomposition} of $M$ with respect to $u$. The canonical projection, $P_j(u),$ of $M$ onto $M_j(u)$ is called the ($j$-)\emph{Peirce projection} associated with $u$. Peirce projections are explicitly determined by the following formulae:\label{eq Fla Peirce projections}
$$\begin{aligned}
P_2(u)&=L(u,u)(2L(u,u)-\operatorname{id}_M)=Q(u)^2,\\
P_1(u)&=4L(u,u)(\operatorname{id}_M-L(u,u))=2(L(u,u)-Q(u)^2),\\
P_0(u)&=(\operatorname{id}_M-L(u,u))(\operatorname{id}_M-2L(u,u))=\operatorname{id}_M-2 L(u,u)+Q(u)^2
\end{aligned}
$$
where the quadratic operator $Q(u):M\to M$ is defined by $Q(u)(x)=\{u,x,u\}$
(compare \cite[1.3 in page 7]{neher1987jordan} or \cite[page 7]{chubook}). In the setting of JB$^*$-triples, Peirce projections are all contractive (see \cite[Corollary 1.2]{Friedman-Russo} or \cite[3.2.1]{chubook}).\smallskip

In case $M$ is a $C^*$-algebra and $u\in M$ is a tripotent (i.e., a partial isometry with initial projection $p_i$ and final projection $p_f$), the Peirce projections are given by the following expressions:
$$
P_2(u)x=p_f x p_i, \quad P_1(u)x= p_f x(1-p_i)+(1-p_f)x p_i, \quad P_0(u)x=(1-p_f)x(1-p_i),
$$
where $x$ runs through $M$.\smallskip

If $i,j,k\in\{0,1,2\}$, then the so-called \emph{Peirce arithmetic} or \emph{Peirce multiplication rules} say that
\begin{equation}\label{eq Peirce arithmetic}  \begin{cases}
\{M_i(u), M_j(u), M_k(u)\} \subseteq M_{i-j+k}(u), & \mbox{ if } i-j+k\in\{0,1,2\},\\
\{M_i(u), M_j(u), M_k(u)\} =0, & \mbox{ if } i-j+k\notin\{0,1,2\},\\
\{M_2(u), M_0(u), M\}= \{M_0(u), M_2(u), M\}=0,&
\end{cases}
\end{equation}
(see \cite[$(1.20)$--$(1.22)$ in pages 7-8]{neher1987jordan} or \cite[Theorem~1.2.44]{chubook}).\smallskip
%These rules were   proved in \cite{Friedman-Russo}, see also \cite[Theorem 1.2.24]{chubook}.

It follows immediately from the Peirce multiplication rules that $M_j(u)$ is a JB$^*$-subtriple for $j=0,1,2$. In the case $j=2$ something
more can be said. In this case $M_2(u)$ is even a unital JB$^*$-algebra with respect to the product and involution given by
$$a\circ b=\J aub, \quad a^* = \J uau,\qquad a,b\in M_2(u),$$ respectively (cf.\cite[\S1.2 and Remark 3.2.2]{chubook} or \cite[Corollary~4.2.30]{Cabrera-Rodriguez-vol1}).
\label{eq Peirce2 is a JBstaralgebra}
It is further known that $u$ is the unit of this JB$^*$-algebra. (If we wish to stress that the operations are with respect to $u$, we write $a\circ_u b$ and $a^{*_u}$.)\smallskip

\subsection{Complete tripotents}

A tripotent $u$ in a JB$^*$-triple $M$ is called \emph{complete} if $M_0(u)=\{0\}$ and \emph{unitary} if $M=M_2(u)$ (that is if $\{u,u,x\}=x$ for all $x\in M$).
It follows from the above structure results that $M$ admits a unitary tripotent if and only if it admits a structure of unital JB$^*$-algebra (cf.~ \cite[Theorem~4.1.55]{Cabrera-Rodriguez-vol1}).
Further, if $M$ is a unital JB$^*$-algebra, then $u\in M$ is a unitary tripotent if and only if it is a unitary element, i.e., an element satisfying that $u$ is invertible in the Jordan sense (i.e., there exists a unique element $b= u^{-1}$ in $M$ such that $b\circ u = 1$ and $u^2\circ b = u$) and $u^{-1} = u^*$
(compare \cite[\S4.1.1]{Cabrera-Rodriguez-vol1}).
In a $C^*$-algebra this reduces to $u^*u=uu^*=1$.\smallskip

A complete tripotent need not be unitary, even in the $C^*$-algebra case. Indeed, if $u\in B(H)$ is a partial isometry whose initial projection is the identity on $H$ but its final projection is strictly smaller than the identity on $H$ (or vice versa), i.e., if $u^*u=1\ne uu^*$ (or vice versa),
then $u$ is a complete non-unitary tripotent. This is not the unique possibility, but it is an important case as witnessed by the forthcoming lemmata. We observe that the extreme points of the closed unit ball of a C$^*$-algebra $A$ are precisely the complete partial isometries (tripotents) of $A$ (see \cite[Theorem 1]{kadison1951isometries} or \cite[Theorem~I.10.2]{Tak}).
The same result remains true for the closed unit ball of a JB$^*$-triple $E$, that is, the complete tripotents in $E$ are the extreme points of its closed unit ball (cf. \cite[Lemma 4.1]{braun1978holomorphic} and \cite[Proposition 3.5]{kaup1977jordan} or \cite[Theorem 3.2.3]{chubook}).\smallskip

We recall that a \emph{conjugation} on a complex Banach space $X$ is a conjugate-linear isometry $\tau:X\to X$ of period-2 (i.e., $\tau^2 = Id_{X}$). Let $H$ be a complex Hilbert space, and let us fix an orthonormal basis $(\xi_k)_{k\in \Lambda}$ in $H$. Given $\xi\in H,$ let $\tau(\xi)\in H$ be the vector defined by $\tau(\xi)=\sum_{k\in \Lambda} \overline{\ip \xi{\xi_k}}\xi_k.$ Then $\tau$ is a  conjugation on $H$ and, moreover, any conjugation is of that form (with a properly chosen orthonormal basis). %(see e.g. \cite[7.5.5, p.167]{hanche1984jordan} ).

\begin{lemma}\label{L:Jembed C*}
Let $M$ be a unital $C^*$-algebra, and let $u\in M$ be a complete tripotent. Then there exist a complex Hilbert space $H$ and an isometric unital Jordan $*$-monomorphism $\psi:M\to B(H)$ such that $\psi(u)^*\psi(u)=1$.
\end{lemma}

\begin{proof} By applying the Gelfand-Naimark-Segal construction, we can find a family of complex Hilbert spaces $\{H_i\}_{i\in I}$ and
irreducible representations $\Phi_i:M\to B(H_i)$, such that
$$\Phi=\bigoplus_{i\in I}\Phi_i:M\to \bigoplus_{i\in I}^{\ell_{\infty}} B(H_i)\subset B(\bigoplus_{i\in I}^{\ell_2} H_i)$$
is an isometric $^*$-monomorphism
(we can consider, for example, the \emph{atomic representation} of $M$ \cite[4.3.7]{pedersen1979c}, where the family $I$ is precisely the set of all pure states of $M$ and each $\Phi_i$ is the irreducible representation associated with the pure state $i$ \cite[Theorem 3.13.2]{pedersen1979c}).

Fix $i\in I$. Since $u$ is tripotent in $M$, $\Phi_i(u)$ is a tripotent in $\Phi_i(M)$ as well. Moreover, $u$ is complete, i.e.,
$$(1-uu^*)a(1-u^*u)=0\mbox{ for }a\in M,$$
hence
$$(1-\Phi_i(u)\Phi_i(u)^*)x(1-\Phi_i(u)^*\Phi_i(u))=0\mbox{ for }x\in \Phi_i(M).$$
Since $\Phi_i$ is irreducible,  its range is weak$^*$-dense in $B(H_i)$, thus the above formula holds for all $x\in B(H_i)$. In other words, $\Phi_i(u)$ is a complete tripotent in $B(H_i)$.
Having in mind that $B(H_i)$ is a factor, it follows that $\Phi_i (u)^* \Phi_i (u) =1$ or $\Phi_i (u) \Phi_i (u)^* =1$ (this follows  e.g. from \cite[Lemma V.1.7]{Tak} applied to the projections $1-\Phi_i(u)^*\Phi_i(u)$ and $1-\Phi_i(u)\Phi_i(u)^*$).\smallskip

Let $I_1:=\{j \in I : \Phi_j (u)^* \Phi_j (u) =1_j \}$ and $I_2:= I\setminus I_1$.\smallskip

For each $j\in I_2$, we can find a $^*$-anti-homomorphism $\Psi_j : B(H_j) \to B(H_j)$ (consider, for example a transposition on $B(H_j)$ defined by $\Psi_j (a) := \tau a^* \tau$, where $\tau$ is the conjugation on $H_j$ described before the statement of the lemma).
For $j\in I_1$, $\Psi_j$ will stand for the identity on $B(H_j)$. Let $\Psi = \bigoplus_{j\in I} \Psi_j : \bigoplus_j^{\ell_{\infty}} B(H_j)\to \bigoplus_j^{\ell_{\infty}} B(H_j)$. Clearly $\Psi$ is a Jordan $^*$-isomorphism. By construction we have $\Psi(\Phi(u))^* \Psi(\Phi(u)) = 1$, the identity in $\bigoplus_j^{\ell_{\infty}} B(H_j)$.  Finally, we can embed the C$^*$-algebra $\bigoplus_j^{\ell_{\infty}} B(H_j)$ inside $B(\bigoplus_j^{\ell_{2}} H_j)$ via a $^*$-monomorphism $\theta$, and the Jordan $^*$-monomorphism  $\psi = \theta \circ \Psi \circ \Phi: M\to  B(\bigoplus_j^{\ell_{2}} H_j)$ satisfies the desired property.
\end{proof}

We continue with a technical result relating complete tripotents in a JC$^*$-algebra $A$ with the complete tripotents in the C$^*$-algebra generated by $A$.

\begin{lemma}\label{l complete tripotents in the Cstar algebra generated} Let $M$ be a unital JB$^*$-algebra. Let $u$ be a complete tripotent in $M$, and let $N$ denote the JB$^*$-subalgebra of $M$ generated by $u$ and the unit element. Then $N$ is a JC$^*$-subalgebra of some C$^*$-algebra $B$, and $u$ is a complete tripotent in the C$^*$-subalgebra of $B$ generated by $N$.
\end{lemma}

\begin{proof} The first statement follows from Lemma~\ref{L:JC*}$(i)$, so fix a $C^*$-algebra $B$ conitaing $N$ as a JC$^*$-subalgebra. Let $A$ be the C$^*$-subalgebra of $B$ generated by $N$. Furhter, let $1$ denote the unit of $M$ (which belongs to $N$). Then for any $x\in N$ we have
$1x1=\J 1x1=x$, hence $x=1x=x1$. Since $A$ is generated by $N$, it follows that $1$ is the unit of $A$.

Clearly $u$ is a tripotent (hence a partial isometry) in $A$, so its Peirce-$0$ projection is given by
$$P_0(u)(a)=(1-uu^*)a(1-u^*u),\quad a\in A.$$
To prove that $u$ is complete in $A$ it is enough to show that $P_0(u)$ vanishes on all the (associative) monomials in $u$ and $u^*$. To this end, we will consider the formal degree of such monomials in the obvious way ($1$ is the unique monomial of degree $0$, monomials of degree $1$ are $u$ and $u^*$, monomials of degree $2$ are $u^2$, $(u^*)^2$, $uu^*$ and $u^*u$ etc.).

Since $u$ is complete in $N$ and $1,u,u^*\in N$, we deduce that $P_0(u)$ vanishes on monomials of degree $0$ or $1$. Assume that $n\in\en$ and $P_0(u)$ vanishes on all the monomials of degree at most $n$. Let $a$ be a monomial of degree $n+1$. If $a=u^{n+1}$ or $a=(u^*)^{n+1}$, then $a\in N$, hence $P_0(u)(a)=0$. Otherwise there are two monomials $b,c$ with $\deg(b)+\deg(c)=n+1$ such that either $a=buu^*c$ or $a=bu^*uc$. If the first possibility takes place, then
$$\begin{aligned}P_0(u)(a)&=(1-uu^*)buu^*c(1-u^*u)\\&=(1-uu^*)bc(1-u^*u)-(1-uu^*)b(1-uu^*)c(1-u^*u)\\&=
P_0(u)(bc)-(1-uu^*)bP_0(u)(c)=0
\end{aligned}$$
by the induction hypothesis. The second case is analogous.
\end{proof}

\begin{lemma}\label{L:Jembed JB*}
Let $M$ be a unital JB$^*$-algebra and $u$ a complete tripotent in $M$. Let $N$ be the closed unital Jordan $^*$-subalgbebra of $M$ generated by $u$. Then there is a unital Jordan $^*$-monomorphism $\psi:N\to B(H),$ where $H$ is a complex Hilbert space, such that $\psi(u)^*\psi(u)=1$.
\end{lemma}

\begin{proof} By Lemma \ref{l complete tripotents in the Cstar algebra generated} we can assume that $N$ is a unital JC$^*$-subalgebra of some $C^*$-algebra $A$ such that $u$ is a complete tripotent in $A$ as well. The desired conclusion follows now from Lemma~\ref{L:Jembed C*}.
\end{proof}

\subsection{Orders on tripotents}

There is a natural partial order on tripotents which we recall below. We start by analyzing a coarser ordering
(see the subsequent Proposition~\ref{P:M2 inclusion}) which will be useful in the next section. We start by the following easy lemma.

\begin{lemma}\label{L:comuting Laa Lxx}
Let $M$ be a JB$^*$-triple, $e\in M$ a tripotent and $x\in M_j(e)$ for some $j\in\{0,1,2\}$. Then the following assertions hold.
\begin{enumerate}[$(a)$]
\item The operators $L(e,e)$ and $L(x,x)$ commute.
\item The operator $L(x,x)$ commutes with the Peirce projections induced by $e$.
\item If $x$ is moreover a tripotent, then the Peirce projections induced by $x$ commute with the Peirce projections induced by $e$.
\end{enumerate}
\end{lemma}

\begin{proof}
(a) Let $j\in\{0,1,2\}$ be such that $x\in M_j(e)$. It means that $L(e,e)x=
\frac{j}{2}x$. By the Jordan identity we deduce that given any $y\in M$ we have
$$
\begin{aligned}
L(e,e)L(x,x)y &=L(e,e)\J xxy
\\&=\J {L(e,e)x}xy- \J x{L(e,e)x}y+ \J xx{L(e,e)y}
\\&=\J {\frac j2 x}xy- \J x{\frac j2 x}y+ \J xx{L(e,e)y}
=L(x,x)L(e,e)y.
\end{aligned}$$
This completes the proof of (a). Assertions (b) and (c) follow from (a) using the formulae for Peirce projections.
\end{proof}

Let us remark that statement $(c)$ was already established by G. Horn in \cite[$(1.10)$]{Horn1987}. The previous result and its proof are included here for completeness reasons.\smallskip

A coarser ordering on the set of tripotents is considered in our next result.

\begin{prop}\label{P:M2 inclusion} Let $M$ be a JB$^*$-triple and $e,u$ be two tripotents in $M$.
The following assertions are equivalent:
\begin{enumerate}[$(1)$]
\item $u\in M_2(e)$;
\item $P_2(u)P_2(e)=P_2(e)P_2(u)=P_2(u)$, $P_1(u)P_1(e)=P_1(e)P_1(u)$ and
$P_0(u)P_0(e)=P_0(e)P_0(u)=P_0(e)$;
\item $M_2(u)\subset M_2(e)$ and $M_0(e)\subset M_0(u)$;
\item $M_2(u)\subset M_2(e)$.
\end{enumerate}
\end{prop}

\begin{proof}
(1)$\Rightarrow$(2) Assume that $u\in M_2(e)$. By Lemma~\ref{L:comuting Laa Lxx} the Peirce projections induced by $u$ commute with the Peirce projections induced by $e$. Further, given $x\in M$ the Peirce rules \eqref{eq Peirce arithmetic} yield
$$\J u{P_1(e)x}u = \J u{P_0(e)x}u=0,$$
thus
$$\J uxu=\J u{P_2(e)x}u.$$
It follows that $Q(u)=Q(u)P_2(e)$, so
$$P_2(u)=Q(u)^2=Q(u)^2P_2(e)=P_2(u)P_2(e).$$

Let $x\in M_0(e)$, another application of Peirce arithmetic yields
$$\J uu{x}\in\J {M_2(e)}{M_2(e)}{M_0(e)}=\{0\},$$ so $M_0(e)\subset M_0(u)$, and hence $P_0(u)P_0(e)=P_0(e)$.\smallskip

The implications (2)$\Rightarrow$(3)$\Rightarrow$(4)$\Rightarrow$(1) are obvious.
\end{proof}

\begin{prop}\label{P:M2 equality}
Let $M$ be a JB$^*$-triple, and let $e,u$ be two tripotents in $M$.
The following assertions are equivalent:
\begin{enumerate}[$(1)$]
\item $u\in M_2(e)$ and $e\in M_2(u)$;
\item $M_2(e)=M_2(u)$;
\item The Peirce decompositions induced by $e$ and $u$ coincide {\rm(}i.e., $M_j (e) = M_j(u)$ for all $j=0,1,2${\rm)}.
\end{enumerate}
\end{prop}

\begin{proof}
The implications (3)$\Rightarrow$(2)$\Rightarrow$(1) are obvious.

(1)$\Rightarrow$(3) Assume $u\in M_2(e)$ and $e\in M_2(u)$. It follows from Proposition~\ref{P:M2 inclusion} (the implication (1)$\Rightarrow$(2)) that $P_2(e)=P_2(u)$ and $P_0(e)=P_0(u)$. Hence also $P_1(u)=P_1(e)$.
\end{proof}

\begin{prop}\label{P:OG trip}
Let $M$ be a JB$^*$-triple, and let $e,u$ be two tripotents in $M$.
The following assertions are equivalent:
\begin{enumerate}[$(1)$]
\item $u\in M_0(e)$;
\item $e\in M_0(u)$;
\item $M_2(u)\subset M_0(e)$ and $M_2(e)\subset M_0(u)$;
\item $P_2(u)P_0(e)=P_0(e)P_2(u)=P_2(u)$ and $P_0(u)P_2(e)=P_2(e)P_0(u)=P_2(e)$.
\end{enumerate}
\end{prop}

\begin{proof}
(1)$\Rightarrow$(2) Assume $u\in M_0(e)$. Then $\J uue=\J euu=0$ by the Peirce arithmetics (note that $e\in M_2(e)$ and $u\in M_0(e)$). Hence $e\in M_0(u)$.

(2)$\Rightarrow$(1) follows by symmetry.

(1)$\Rightarrow$(4) Assume $u\in M_0(e)$. By the already proved implications we know that also $e\in M_0(u)$.
It follows from Peirce arithmetic that $M_2(u)\subseteq M_0(e)$ and $M_2(e)\subseteq M_0(u)$. Therefore $P_2(u) = P_0(e) P_2(u)$ and $P_2(e) = P_0(u) P_2(e)$.
Since, by Peirce arithmetics, we also have $\{u,M_2(e),u\}= \{u,M_1(e),u\}=\{0\}$, and $P_2(u) = Q(u)^2$, we deduce that $P_2(u) P_j(e) =0,$ for $j=1,2$.
Therefore, $P_2(u)= P_2(u) P_0(e),$ and similarly $P_2(e)= P_2(e) P_0(u).$\smallskip

The implications (4)$\Rightarrow$(3)$\Rightarrow$(2) are obvious.
\end{proof}

\begin{remark} Propositions~\ref{P:M2 inclusion} and~\ref{P:M2 equality} show that the Peirce subspace $M_2(e)$ determines the whole Peirce decomposition. This is not the case for $M_0(e)$ as there may exist complete tripotents with different Peirce decompositions.\smallskip
\end{remark}

Tripotents $u,e\in M$ satisfying any of the equivalent conditions in Proposition~\ref{P:OG trip} are called \emph{orthogonal} ($u\perp e$ in short). In particular, $e\pm u$ are again tripotents. It is actually known that given two tripotents $e,u\in M$, then $e\perp u$ if and only if $e\pm u$ is a tripotent (cf. \cite[Lemma 3.6]{isidro1995real}).\smallskip

%It seems that there is no similar easy characterization of the feature $u\in M_1(e)$. Let $e$ and $u$ be two tripotents in a JB$^*$-triple $M$. For information, we observe that, according to the standard notation in texts like \cite{horn1987coordinatization,dang1987classification,FPMaPe}, we say that $e$ and $v$ are \emph{collinear} (respectively, \emph{rigidly collinear}) if $e\in M_1(u)$ and $u\in M_1(e)$ (respectively, $M_2(e)\subset M_{1}(u)$ and $M_2(u)\subset M_{1}(e)$). If $e\in M_1 (u)$ and $u\in M_2(e)$ then $e$ \emph{governs} $u$. \smallskip

We are now in position to recall the natural partial order on the set of tripotents. If $M$ is a JB$^*$-triple and $e,u$ are two tripotents in $M$, we say that $u\le e$ if $e-u$ is a tripotent orthogonal to $u$.
This order is finer than the one derived from Proposition~\ref{P:M2 inclusion} as can be seen from the last of the
characterizations in the following proposition (originally due to Y. Friedman and B. Russo \cite[Corollary~1.7]{Friedman-Russo},
compare also \cite[Proposition~1.2.43]{chubook}, \cite[Corollary 5.10.56]{Cabrera-Rodriguez-vol2}).

\begin{prop}\label{p characterization of triple order}{\rm(essentially \cite[Corollary~1.7]{Friedman-Russo})}
Let $M$ be a JB$^*$-triple and $e,u\in M$ two tripotents. Then the following assertions are equivalent:
\begin{enumerate}[$\bullet$]
\item $u\le e$;
\item $u=P_2(u)e$;
    \item $u= \J ueu$;
    \item $u$ is a projection in the JB$^*$-algebra $M_2(e)$;
    \item $M_2(u)$ is a JB$^*$-subalgebra of $M_2(e)$.
\end{enumerate}
\end{prop}

\subsection{More on the Peirce-2 subspace}

Our next result gathers some properties of the Peirce-2 subspace associated with a tripotent. Most of the statements are part of the folklore in the theory of JB$^*$-triples, we include here the properties and basic references for completeness reasons.

\begin{lemma}\label{L:properties of M2(e)}%{\rm(\cite[Lemma 1.5(b)]{Friedman-Russo}, \cite{peralta2015positive})}
Let $M$ be a JB$^*$-triple and let $e\in M$ be a tripotent. Consider $M_2(e)$ equipped with its structure of JB$^*$-algebra.
\begin{enumerate}[$(a)$]
\item  Assume that $v\in M$ is a tripotent such that $e\le v$. Then for any $a,b\in M$ we have
$$\begin{aligned}
P_2(e)\J abv & = \J{P_2(e)a}{P_2(e)b}e + \J{P_1(e)a}{P_1(e)b}e,\\
P_1(e)\J abv & = \J{P_1(e)a}{P_2(e)b}e + \J{P_0(e)a}{P_1(e)b}e
\\ &\qquad+\J{P_2(e)a}{P_1(e)b}{v-e} + \J{P_1(e)a}{P_0(e)b}{v-e},\\
P_0(e)\J abv & =\J{P_0(e)a}{P_0(e)b}{v-e}+\J{P_1(e)a}{P_1(e)b}{v-e},\end{aligned}$$
in particular
$$\begin{aligned}
P_2(e)\J abe & = \J{P_2(e)a}{P_2(e)b}e + \J{P_1(e)a}{P_1(e)b}e,\\
P_1(e)\J abe & = \J{P_1(e)a}{P_2(e)b}e + \J{P_0(e)a}{P_1(e)b}e,\\
P_0(e)\J abe & =0.\end{aligned}$$
\item Assume $j\in\{1,2\}$ and $a,b\in M_j(e)$. Then
$$\J abe \in M_2(e)\mbox{ and }\J abe ^*=\J bae.$$
\item  Assume $a,b\in M_2(e)$. Then $a\circ b^*=\J abe$.
\item If $a\in M_2(e)\cup M_1(e)$, then $\J aae$ is a positive element of the JB$^*$-algebra $M_2(e)$. Moreover, $\J aae=0$ if and only if $a=0$.
\end{enumerate}
\end{lemma}

\begin{proof}
Fix a tripotent $v\in M$ with $e\le v$. Then
$P_2(e)v=e$, $P_1(e)v=0$, and $P_0(e)v=v-e$.
Hence the Peirce arithmetic implies that, given $x\in M_j(e)$ and $y\in M_k(e)$ for some $j,k\in \{0,1,2\}$, we have
\begin{equation}\label{eq:Peirce}
\J xye \begin{cases}
\in M_2(e) & \mbox{ if }j=k=2\mbox{ or }j=k=1,\\
\in M_1(e) & \mbox{ if }j=1,k=2\mbox{ or }j=0,k=1,\\
=0 &\mbox{ otherwise,}
\end{cases}
\end{equation}
and
\begin{equation}\label{eq:Peirce*}
\J xy{v-e} \begin{cases}
\in M_0(e) & \mbox{ if }j=k=0\mbox{ or }j=k=1,\\
\in M_1(e) & \mbox{ if }j=1,k=0\mbox{ or }j=2,k=1,\\
=0 &\mbox{ otherwise.}
\end{cases}
\end{equation}
Assertion $(a)$ now follows from \eqref{eq:Peirce} and \eqref{eq:Peirce*}.
Further, the first statement of assertion $(b)$ follows from \eqref{eq:Peirce}.
Let us continue by proving the second statement from $(b)$. We deduce from the Jordan identity, the definition of the involution in $M_2(e)$, and the fact that $\J bae\in M_2(e)$, that
$$\begin{aligned}
\J bae&=L(b,a)e=L(b,a)\J eee\\&=\J{L(b,a)e}ee-\J e{L(a,b)e}e+\J ee{L(b,a)e}
\\&=2 L(e,e)\J bae- (L(a,b)e)^*=2\J bae - \J abe ^*.
\end{aligned}$$
%Indeed, the first equality follows from the definitions, the second one uses the assumption that $e$ is a tripotent, the third one is the Jordan identity, the fourth one follows again  from the definitions, the last one uses the fact that $\J bae\in M_2(e)$ (see \eqref{eq:Peirce}).

$(c)$ The Peirce-2 subspace $M_2(e)$ is a JB$^*$-algebra, and hence it is a JB$^*$-triple with respect to the triple product given by $\{a,b,c\}_1= (a\circ b^*)\circ c + (c\circ b^*)\circ a- (a\circ c)\circ b^*$. It is also a JB$^*$-triple with the triple product inherited from $M$. Since the identity mapping from $(M_2(e),\{.,.,.\}_1)$ onto $(M_2(e),\{.,.,.\})$ is a surjective isometry, it follows from Kaup's Riemann mapping theorem (see \cite[Proposition 5.5]{kaup1983riemann} or \cite[Theorem~3.1.7]{chubook}) that $\{a,b,c\}_1= \{a,b,c\}$, for all $a,b,c\in M_2(e)$. Consequently, $\J abe = \{a,b,e\}_1= (a \circ b^*)\circ e + (e\circ b^*)\circ a - (a\circ e)\circ b^* = a \circ b^*$, because $e$ is the unit of $M_2(e)$.\smallskip

$(d)$ If $a\in M_2(e)$, then $\J aae=a\circ a^*$ by $(c)$, hence the assertion follows from Lemma~\ref{L:JC*}(ii).
The case $a\in M_1(e)$ is covered by \cite[Lemma 1.5(b)]{Friedman-Russo}, and both cases ($a\in M_1(e)$ and $a\in M_2(e)$) are fully studied in \cite{peralta2015positive} (see also \cite[Proposition 4.2.32]{Cabrera-Rodriguez-vol1}), where a simple proof based on the axioms of JB$^*$-triples can be found.
\end{proof}

When we combine the previous result with the properties of the functionals in the dual space of a JB$^*$-triple we get additional properties. We recall that a functional $\varphi$ in the dual space of a JB$^*$-algebra $M$ is called faithful if $\varphi (a) = 0$ for $a\geq 0$ implies $a=0$.

\begin{lemma}\label{L:functionals and seminorms}
Let $M$ be a JB$^*$-triple and let $e\in M$ be a tripotent. Consider $M_2(e)$ equipped with its structure of unital JB$^*$-algebra. Let $\varphi\in M^*$. Then the following assertions hold:
\begin{enumerate}[$(a)$]
\item $\norm{\varphi \circ (P_2(e)+P_0(e))}=\norm{\varphi\circ P_2(e)}+\norm{\varphi\circ P_0(e)}$.
\end{enumerate}
Moreover, if $\norm{\varphi}=\varphi(e)$, then the following assertions are valid, too:
\begin{enumerate}[$(b)$]
\item  $\varphi=\varphi\circ P_2(e)$;
\item[$(c)$]  $\varphi|_{M_2(e)}$ is a positive linear functional on the JB$^*$-algebra $M_2(e)$;
\item[$(d)$] The mapping
$$(x,y)\mapsto \varphi(\J xye),\qquad x,y\in M,$$
is a positive semidefinite sesquilinear form on $M$, and if $z\in M$ is
a norm-one element satisfying $\varphi (z) = \|\varphi\|$ then $\varphi(\J xye) = \varphi(\J xyz)$ for all $x,y\in M$;
\item[$(e)$] The formula
$$\norm{x}_{e,\varphi}=\sqrt{\varphi(\J xxe)},\quad x\in M$$
defines a pre-Hilbert seminorm on $M$ which is zero on $M_0(e)$.\smallskip

\noindent If moreover $\varphi|_{M_2(e)}$ is faithful, then the kernel of  $\norm{\cdot}_{e,\varphi}$ is exactly $M_0(e)$.
\end{enumerate}
\end{lemma}

\begin{proof} Assertions $(a)$ and $(b)$ are proved in \cite[Lemma 1.3(b) and Proposition 1$(a)$]{Friedman-Russo}, compare also \cite[Lemma 5.7.11, Fact 5.10.53]{Cabrera-Rodriguez-vol2}.\smallskip

$(c)$ Since $\norm{\varphi}=\varphi(e)$ we also have $\norm{\varphi\circ P_2(e)} =\norm{\varphi} = \norm{\varphi|_{M_2(e)}}= \varphi|_{M_2(e)}(e).$
Therefore $\varphi|_{M_2(e)}$ is a positive functional on the JB$^*$-algebra $M_2(e)$ (cf. \cite[Lemma 1.2.2]{hanche1984jordan} or
\cite[Lemma 5.10.2]{Cabrera-Rodriguez-vol2}).\smallskip

$(d)$ and $(e)$ are consequences of $(c)$, $(b)$ and Lemma \ref{L:properties of M2(e)}$(a)$ and $(d)$. They are also explicitly proved in \cite[Proposition 1.2]{barton1987grothendieck} and \cite[Lemma 4.1]{Edw-Rut-exposed}.
See also \cite[Proposition 5.10.60]{Cabrera-Rodriguez-vol2}  for the JBW$^*$-case.\smallskip
\end{proof}

\section{Strong$^*$ topology and weakly compact sets}\label{sec:strong*}

In the previous section we collected many results on projections and tripotents. However, it may happen that there are no nontrivial projections or tripotents.
For example, the $C^*$-algebra $\C_0(\er)$ contains no nonzero projection or tripotent and in the unital $C^*$-algebra $\C([0,1])$ the only projections are $0$ and $1$
and the only nonzero tripotents are the unitary ones (which coincide with the continuous functions with values in the unit circle).
The situation is different in the dual case -- in a von Neumann algebra projections form a complete lattice and their linear span is norm-dense,
%So, there is an abundance of projections (which may be constructed using the measurable calculus). There is also abundance of tripotents (i.e., partial isometries) -- they may be obtained by polar decomposition. Analogous results hold also for JBW$^*$-algebras (cf. \cite[\S 4 and Lemma 4.1.11]{hanche1984jordan}).
 and, as we previously commented, any JBW$^*$-triple provides a rich supply of tripotents %-- for example the complete tripotents coincide with the extreme points of the unit ball
(cf.  \cite[Theorem 3.2.3]{chubook} or \cite[Theorem 4.2.34]{Cabrera-Rodriguez-vol1}).\smallskip
% also \cite[Lemma 4.1]{braun1978holomorphic} and \cite[Proposition 3.5]{kaup1977jordan}).\smallskip

If $M$ is a JBW$^*$-triple and $u\in M$ is a tripotent,
then the Peirce projections are weak$^*$-to-weak$^*$ continuous and the Peirce subspaces are weak$^*$-closed.
This follows from the separate weak$^*$-to-weak$^*$ continuity of the triple product and the explicit formulae for the Peirce projections displayed in page \pageref{eq Fla Peirce projections}. In particular, $M_2(u)$
is a JBW$^*$-algebra.

\subsection{Strong$^*$ topology on JBW$^*$-triples}

Assume that $M$ is a JBW$^*$-triple and $\varphi\in M_*\setminus\{0\}$. By \cite[Proposition 2]{Friedman-Russo} (see also \cite[Proposition\ 5.10.57]{Cabrera-Rodriguez-vol2}) there is a unique tripotent $s(\varphi)\in M$, called the \emph{support tripotent} of $\varphi$, such that
\begin{enumerate}[$\bullet$]
\item $\varphi=\varphi\circ P_2(s(\varphi))$,
\item $\varphi|_{M_2(s(\varphi))}$ is a faithful positive functional on the JBW$^*$-algebra $M_2(s(\varphi))$.
\end{enumerate}
Furthermore, $\norm{\varphi}=\varphi(s(\varphi))$.
According to the notation from Lemma~\ref{L:functionals and seminorms}, we set
$\norm{\cdot}_\varphi=\norm{\cdot}_{s(\varphi),\varphi}.$\smallskip

Note that if $M$ is a JBW$^*$-algebra (or even a von Neumann algebra) and $\varphi\in M_*$ is a positive functional, then its support tripotent $s(\varphi)$ is even a projection because in such a case $\varphi$ attains its norm at a positive element. Observe that in the latter case the seminorm $\norm{\cdot}_\varphi$ writes in the form
$$\norm{x}_\varphi=\sqrt{\varphi \J xx{s(\varphi)}}
=\sqrt{\varphi \J xx1}=\sqrt{\varphi(x^*\circ x)}.$$

Introduced in \cite{BaFr}, the \emph{strong$^*$ topology} on $M$ is the locally convex topology generated by the seminorms
$\norm{\cdot}_\varphi$ where $\varphi$ runs in the set $M_*\setminus\{0\}.$ It should be noted that in the original definition (see \cite[Definition 3.1]{BaFr}) only norm-one functionals are considered, but both definitions obviously give the same notion. Since each $\|.\|_\varphi$ is a preHilbertian seminorm, it follows from the Cauchy-Schwarz inequality (cf.\ \ref{L:functionals and seminorms}(d)) and the properties of the support tripotent that, with $x_2=P_2(s(\varphi))x$,
\begin{eqnarray*}
|\varphi(x)|&=&|\varphi(x_2)|
=|\varphi(x_2\circ_{s(\varphi)}s(\varphi))|
=|\varphi(\{x_2,s(\varphi),s(\varphi)\})|\\
&\stackrel{\ref{L:properties of M2(e)}}{=}&
|\varphi(P_2(s(\varphi))\{x,s(\varphi),s(\varphi)\})|
=|\varphi \J x{s(\varphi)}{s(\varphi)}|\\
&\leq& \|x\|_{\varphi} \ \|s(\varphi)\|_{\varphi} = \sqrt{\|\varphi\|} \ \|x\|_{\varphi}, \quad  x\in M, \varphi\in M_*\setminus\{0\}.
\end{eqnarray*}
Consequently, the strong$^*$ topology is stronger than the weak$^*$ topology.\smallskip

Given $\varphi_1,\ldots,\varphi_n\in M_*$, we shall write $\|.\|_{\varphi_1,\ldots,\varphi_n}$ for the seminorm on $M$ defined by $$\|x\|_{\varphi_1,\ldots,\varphi_n}^2:=\sum_{k=1}^n \|x\|^2_{\varphi_k}\ \ (x\in M).$$

The following lemma summarizes some known properties of the strong$^*$ topology.

\begin{lemma}\label{L:strong* topology}
Let $M$ be a JBW$^*$-triple.
\begin{enumerate}[$(a)$]
\item If $M$ is even a JBW$^*$-algebra, then the strong$^*$ topology on $M$ coincides with the algebra strong$^*$ topology, i.e., with the locally convex topology generated by seminorms
$$x\mapsto \sqrt{\varphi(x^*\circ x)},\qquad \varphi\in M_*,\varphi\ge 0;$$
\item If $N$ is a weak$^*$ closed subtriple of $M$, then the strong$^*$ topology of $N$ coincides with the restriction to $N$ of the strong$^*$ topology of $M$;
\item If $M$ is even a von Neumann algebra (embedded to some $B(H)$), the strong$^*$ topology coincides on bounded sets  with the locally convex topology generated by the seminorms
$$x\mapsto\norm{x\xi}+\norm{x^*\xi},\qquad \xi\in H;$$
\item A linear functional $\varphi:M\to \mathbb{C}$ is strong$^*$ continuous if and only if it is weak$^*$ continuous. Furthermore a linear mapping between JBW$^*$-triples is strong$^*$-to-strong$^*$ continuous if and only if it is weak$^*$-to-weak$^*$ continuous.
In particular, the Peirce projections associated with a tripotent are strong$^*$-to-strong$^*$ continuous;
\item If $s(\varphi)$ is complete, then $\norm{x}_{\varphi}$ is a norm on $M$.
\end{enumerate}
\end{lemma}

\begin{proof}
Assertion $(a)$ is proved in \cite[Proposition 3]{RodrPal-strong*}, while $(b)$ is established in \cite[COROLLARY]{bunce2001norm}.\smallskip

Let us justify assertion $(c)$. In \cite[Definition II.2.3]{Tak} the name \emph{$\sigma$-strong$^*$ operator topology} is used for the algebra strong$^*$ topology in $B(H)$. By \cite[Lemma II.2.5(iii)]{Tak} this topology coincides on bounded sets with the topology generated by the given seminorms. Hence we can conclude by applying $(a)$ and $(b)$.\smallskip

Statement $(d)$ is proved in \cite[Corollary 9]{peralta2001grothendieck} and \cite[Corollary 3]{RodrPal-strong*} and the comments before \cite[Theorem 9]{peralta2001grothendieck}.\smallskip

$(e)$ If $s(\varphi)$ is a complete tripotent, then $M_0(s(\varphi))=\{0\}$, thus the statement follows from
Lemma~\ref{L:functionals and seminorms}$(e)$.
\end{proof}

The description of the strong$^*$ topology is closely related to $\sigma$-finite projections and tripotents. Recall that a projection $p$ in a JBW$^*$-algebra is called \emph{$\sigma$-finite} if any family of pairwise orthogonal smaller nonzero projections is at most countable. If the unit of a JBW$^*$-algebra is $\sigma$-finite, the respective algebra is called \emph{$\sigma$-finite}. The classical definitions in von Neumann algebras are exactly the same. Let us note that some authors also employ the term \emph{countably decomposable} to refer to $\sigma$-finite projections in a von Neumann algebra (cf. \cite[Definition 2.1.8]{sakai2012c} or \cite[Definition 5.5.14]{KR1}).\smallskip

Similarly, a tripotent $u$ in a JBW$^*$-triple is \emph{$\sigma$-finite} if any family of pairwise orthogonal nonzero smaller tripotents is at most countable. A JBW$^*$-triple is itself called \emph{$\sigma$-finite} if it admits a $\sigma$-finite complete tripotent (cf. \cite[\S 3]{Edw-Rut-exposed}).
The next result gathers some basic facts on $\sigma$-finite tripotents.\smallskip

Let us recall a couple of notions. A subspace $I$ of a JB$^*$-triple $E$ is an \emph{inner ideal} if $\J IEI\subseteq I$. Every inner ideal of $E$ is a subtriple.\smallskip

Given a norm-one element $a$ in a JBW$^*$-triple $M$, there exists a smallest tripotent $e\in M$ satisfying that $a$ is a positive element in the JBW$^*$-algebra $M_{2} (e)$, this tripotent is called the \emph{range tripotent} of $a,$ and it will be denoted by $r(a)$ (see, for example, \cite[comments before Lemma 3.1]{edwards1988facial}). For a non-zero element $b\in M$, the range tripotent of $b$, $r(b)$, is defined as the range tripotent of $\frac{b}{\|b\|}$ and we set $r(0)=0$. It follows from the same reference that if $M$ is a JBW$^*$-algebra and $x$ is a non-zero positive element, then $r(x)$ is a projection and it coincides with the range projection in \cite[Lemma 4.2.6]{hanche1984jordan}.

\begin{lemma}\label{L:sigma finite tripotent}
\begin{enumerate}[$(a)$]
\item Let $u$ be a tripotent in a JBW$^*$-triple $M$. Then $u$ is $\sigma$-finite if and only if $u=s(\varphi)$ for some norm-one functional $\varphi\in M_*$;
\item Let $M$ be a JBW$^*$-algebra and $p\in M$ a projection. Then $p$ is a $\sigma$-finite projection if and only if it is a $\sigma$-finite tripotent;
\item Let $M$ be a JBW$^*$-algebra and $p\in M$ a projection. Then $p$ is $\sigma$-finite if and only if $p=s(\varphi)$ for some normal state {\rm(}i.e., a positive norm-one functional{\rm)} $\varphi\in M_*$;
\item Let $M$ be a JBW$^*$-algebra and $e\in M$ a $\sigma$-finite tripotent. Then there is a $\sigma$-finite projection $p\in M$ such that $e\in M_2(p)$.
\end{enumerate}
\end{lemma}

\begin{proof}
Assertion $(a)$ is proved in \cite[Theorem 3.2]{Edw-Rut-exposed}. Statement $(b)$ follows from the fact that any projection is also a tripotent, and from the property that a tripotent $u$ is smaller than or equal to a projection $p$ if and only if $u$ is a projection and $u\leq p$ (cf. Proposition \ref{p characterization of triple order}).\smallskip

Statement $(c)$ is a consequence of $(a)$.% and of the property that a normal functional is positive if it attains its norm at a positive element.
\smallskip

$(d)$ Let us consider the sets
$$\begin{aligned}
S&=\{x\in M: \ \exists p\in M\mbox{ a $\sigma$-finite projection such that } x\in M_2(p)\},\\
M_{\sigma}&=\{x\in M: \ \exists u\in M\mbox{ a $\sigma$-finite tripotent such that } x\in M_2(u)\}.
\end{aligned}$$

Clearly $S\subseteq M_{\sigma}$.
By \cite[page 667 and Theorems 4.1 and 5.1]{BHKPP-triples} we have $M_{\sigma}=\{x\in M : r(x) \hbox{ is $\sigma$-finite}\}$ is (a \emph{$1$-norming $\Sigma$-subspace} and) a norm-closed inner ideal of $M=(M_*)^*$ (see the quoted paper for definitions). Since $M$ is a JBW$^*$-algebra, we get $M_{\sigma}\circ M_{\sigma} = \{M_{\sigma}, 1, M_{\sigma}\}\subseteq M_{\sigma},$ and hence $M_{\sigma}$ is a Jordan subalgebra of $M$.\smallskip

It can be seen easily that a tripotent $u\in M$ is $\sigma$-finite if and only if $u^*$ is, therefore $M_{\sigma}^*\subseteq M_{\sigma}$, and hence $M_{\sigma}^*= M_{\sigma}$ is a norm-closed JB$^*$-subalgebra of $M,$ which is also hereditary in the Jordan terminology, that is, if $0\leq a\leq b$ in $M$ with $b\in M_{\sigma}$, then $a\in M_{\sigma}$.\smallskip

Let $a= h+ i k$ be an element in $M_{\sigma}$, where $h,k\in (M_{\sigma})_{sa}$. The elements $h^2,k^2$ are positive and belong to $M_{\sigma}$, thus $h^2 + k^2\in M_{\sigma}$. Therefore, the range tripotent $p=r(h^2 + k^2)$ is a $\sigma$-finite projection in $M$. Clearly, $h^2,k^2\leq h^2+k^2\in M_2 (p)\subseteq M_{\sigma},$ and consequently, $h,k\in M_2 (p)$, which implies that $p \circ h =h$ and $p\circ k = k$. Therefore $\{p,p,a\}=p\circ a = a$ as desired.
\end{proof}

Lemme 2 in Appendice 6 in \cite{iochumbook} offers a sufficient condition to guarantee the metrizability of the strong$^*$ topology on the bounded subsets of a JBW-algebra. We shall next adapt the result to JBW$^*$-algebras.

\begin{lemma}\label{L:JBW*algebra sigma finite} Let $M$ be a $\sigma$-finite JBW$^*$-algebra.
\begin{enumerate}[$(a)$]
\item $M$ admits a faithful normal state;
\item Let $\varphi$ be a faithful normal state on $M$. Then the topology induced by the norm
$$x\mapsto\sqrt{\varphi(x^*\circ x)}= \|x\|_{\varphi}, \qquad x\in M,$$
coincides with the strong$^*$ topology on bounded sets.
\end{enumerate}
\end{lemma}

\begin{proof} $(a)$ Since $1$ is a $\sigma$-finite projection, by Lemma~\ref{L:sigma finite tripotent}(c) there is a positive norm-one functional $\varphi\in M_*$ (i.e., a normal state) such that $s(\varphi)=1$. It follows that $\varphi$ is faithful.\smallskip

$(b)$ We know that $M_{sa}$ is a JBW-algebra and a real JBW$^*$-subtriple of $M$, and that $\varphi|_{M_{sa}}$ is a faithful normal state on $M_{sa}$. \cite[Appendice 6, Lemme 2]{iochumbook} implies that the strong$^*$ topology on the closed unit ball of $M_{sa}$ is metrized by the seminorm $\|x\|^2_{\varphi}=\varphi(x\circ x)=\varphi(x^2) ,$ $x\in M_{sa}$. Let $(a_{\lambda})_{\lambda}$ be a net in $B_M$, and $a\in B_M$ such that $\|a_{\lambda}-a\|_{\varphi} \to 0$. If we write $a_{\lambda} = h_{\lambda} + i k_{\lambda}$ and $a= h +i k$ with $h_{\lambda}, k_{\lambda}, h, k\in B_{M_{sa}}$, then we get the inequalities $$\varphi ((h_{\lambda} - h)^2), \varphi ((k_{\lambda} - k)^2) \leq \varphi ((h_{\lambda} - h)^2 + (k_{\lambda} - k)^2) = \|a_{\lambda}-a\|^2_{\varphi}.$$
Therefore, $(h_{\lambda})_{\lambda}\to h$ and $(k_{\lambda})_{\lambda}\to k$ in the strong$^*$ topology of $M_{sa}$, and by \cite[COROLLARY]{bunce2001norm} they also converge to the same limits with respect to the strong$^*$ topology of $M$, and consequently, $(a_{\lambda})_{\lambda}\to a$ in the strong$^*$ topology of $M$.
\end{proof}

The previous lemma says, in particular, that the strong$^*$ topology is metrizable on bounded sets of a $\sigma$-finite JBW$^*$-algebra. The analogous statement for von Neumann algebras is proved already in \cite[Proposition III.5.3]{Tak}.
The analogy for JBW$^*$-triples fails, as we will explain below (see Example \ref{l example where the strong* topology is not metrizable on bounded sets}).\smallskip

We continue now with a lemma characterizing strong$^*$ convergence of bounded positive nets.

\begin{lemma}\label{L:strong* convergence for positive}
Let $(x_\nu)$ be a bounded net of positive elements in a JBW$^*$-algebra $M$.
Then the following assertions are equivalent.
\begin{enumerate}[$(1)$]
\item\label{it:strong* conv1} $x_\nu\overset{\mbox{strong}^*}{\longrightarrow}0$;
\item\label{it:strong* conv2} $\varphi(x_\nu)\to 0$ for each positive $\varphi\in M_*$;
\item\label{it:strong* conv3} $x_\nu\overset{\mbox{weak}^*}{\longrightarrow}0$.
\end{enumerate}
\end{lemma}

\begin{proof}\iffalse
(1)$\Leftrightarrow$(2) By Lemma~\ref{L:strong* topology}(a) the net $(x_\nu)$ strong$^*$ converges to zero if and only if for any positive $\varphi\in M_*$ we have $\varphi(x_\nu^2)\to 0$. Moreover, the inequalities
$$0\le x_\nu^2\le \norm{x_\nu}x_\nu$$
prove the implication ($\Leftarrow$). Conversely, assume that $(x_\nu)$ is not strong$^*$ convergent to zero, then it is in particular not norm convergent to zero, so, up to passing to a subnet we may assume that there is some $c>0$ such that $\norm{x_\nu}\ge c$ for each $\nu$. Then
$0\le c \cdot 1\le  x_\nu,$ hence $(2)$ fails as well.\smallskip

The implication $(3)\Rightarrow(2)$ is trivial and $(1)\Rightarrow(3)$ follows from the fact that the strong$^*$ topology is stronger than the weak$^*$ one.

{\em Alternative (or rather small shortening because $(1)\Leftarrow(2)$ is shown twice}
\fi
The implication $(1)\Rightarrow(3)$ follows from the fact that the strong$^*$ topology is stronger than the weak$^*$ one,
and $(3)\Rightarrow(2)$ is trivial.\\
(2)$\Rightarrow$(1): By Lemma~\ref{L:strong* topology}(a) the net $(x_\nu)$ strong$^*$ converges to zero if and only if for any positive $\varphi\in M_*$ we have $\varphi(x_\nu^2)\to 0$. Now the double inequality
$0\le x_\nu^2\le \norm{x_\nu}x_\nu$
proves the implication (2)$\Rightarrow$(1).
\end{proof}

\subsection{Order on seminorms generating the strong$^*$ topology}

To describe the strong$^*$ topology it is not necessary (in some cases) to use all the defining seminorms. This is witnessed by Lemma~\ref{L:strong* topology}(a) and, on bounded sets, by Lemma~\ref{L:JBW*algebra sigma finite}. In this section we investigate this feature in detail.
A key result is the following proposition relating the order on seminorms with the order from Proposition~\ref{P:M2 inclusion}.

\begin{prop}\label{P:seminorms order} Let $M$ be a JBW$^*$-triple and $\varphi,\psi\in M_*\setminus\{0\}$.
\begin{enumerate}[$(i)$]
\item Assume $s(\psi)\in M_2(s(\varphi))$ {\rm(}or, equivalently, $M_2(s(\psi))\subset M_2(s(\varphi))${\rm)}. Then the seminorm $\norm{\cdot}_\psi$ is weaker than the seminorm $\norm{\cdot}_\varphi$ on bounded sets.
\item Assume  $M_2(s(\psi))\subsetneqq M_2(s(\varphi))$. Then on $B_M$, the seminorm $\norm{\cdot}_\psi$ is strictly weaker than the seminorm $\norm{\cdot}_\varphi$.
\end{enumerate}
\end{prop}

To prove this proposition we will need some lemmata.
The first lemma characterizes convergence of sequences
in a fixed seminorm. Note that the same characterization
applies to nets, but since we are comparing seminorms, sequences are enough.

\begin{lemma}\label{L:seminorm characterization}
 Let $M$ be a JBW$^*$-triple and $\varphi\in M_*\setminus\{0\}$. Let $e=s(\varphi)$ and let $(a_n)$ be a bounded sequence in $M$.
Then the following assertions are equivalent:
\begin{enumerate}[$(1)$]
\item $\norm{a_n}_\varphi\to 0$;
\item  $\J {P_2(e)(a_n)}{P_2(e)(a_n)}e\overset{\mbox{strong}^*}{\longrightarrow}0$  and $\J {P_1(e)(a_n)}{P_1(e)(a_n)}e \overset{\mbox{strong}^*}{\longrightarrow} 0$;
\item  $\J {P_2(e)(a_n)}{P_2(e)(a_n)}e\overset{\mbox{weak}^*}{\longrightarrow}0$  and $\J {P_1(e)(a_n)}{P_1(e)(a_n)}e \overset{\mbox{weak}^*}{\longrightarrow} 0$;
\item  $\J {P_2(e)(a_n)}{P_2(e)(a_n)}e+\J {P_1(e)(a_n)}{P_1(e)(a_n)}e\overset{\mbox{weak}^*}{\longrightarrow}0$.
\end{enumerate}
\end{lemma}

\begin{proof}
First notice, that it does not matter whether the convergence is considered in the JBW$^*$-triple $M$ or in the JBW$^*$-algebra $M_2(e)$ (cf. \cite[COROLLARY]{bunce2001norm}). The strong$^*$ case follows from Lemma~\ref{L:strong* topology}$(a)$, $(b)$, the weak$^*$ case is obvious.\smallskip

$(1)\Rightarrow(2)$
Assume that $\norm{a_n}_\varphi\to 0$, i.e., $\varphi(\J {a_n}{a_n}e)\to 0$. Since $\varphi=\varphi\circ P_2(e)$, by Lemma~\ref{L:properties of M2(e)}$(a)$ we have
$$\varphi(\J {a_n}{a_n}e)=\varphi(\J {P_2(e)(a_n)}{P_2(e)(a_n)}e + \J {P_1(e)(a_n)}{P_1(e)(a_n)}e).$$ We actually know that the elements
$$w_n=\J {P_2(e)(a_n)}{P_2(e)(a_n)}e\mbox{ and }z_n=\J {P_1(e)(a_n)}{P_1(e)(a_n)}e$$
are positive in the JBW$^*$-algebra $M_2(e)$ by Lemma~\ref{L:properties of M2(e)}$(d)$. Note that $0\le w_n^2\le\norm{w_n}\cdot w_n$. Since the sequence $(w_n)$ is bounded, we deduce $\varphi(w_n^2)\to 0$, hence $w_n \overset{\mbox{strong}^*}{\longrightarrow} 0$ by Lemma~\ref{L:JBW*algebra sigma finite}$(b)$.
Similarly we get  $z_n \overset{\mbox{strong}^*}{\longrightarrow} 0$.\smallskip

$(2)\Rightarrow(3)$ This is clear, as the weak$^*$ topology is weaker than the strong$^*$ one.\smallskip

$(3)\Rightarrow(4)$ This follows by the linearity of the weak$^*$ topology.\smallskip

$(4)\Rightarrow(1)$ This follows from the fact that
$$\norm{a_n}_\varphi^2=\varphi(\J {P_2(e)(a_n)}{P_2(e)(a_n)}e + \J {P_1(e)(a_n)}{P_1(e)(a_n)}e).$$
\end{proof}

The following lemma, together with the preceding one, provides a proof of assertion $(i)$ of Proposition~\ref{P:seminorms order}.

\begin{lemma}\label{L:semin-order-positive} Let $M$ be a JBW$^*$-triple and $p,u\in M$ tripotents such that $u\in M_2(p)$. Then for any bounded sequence $(a_n)$ in $M$ we have
\begin{multline*}
\J{P_2(p)(a_n)}{P_2(p)(a_n)}p+\J{P_1(p)(a_n)}{P_1(p)(a_n)}p \overset{\mbox{weak}^*}{\longrightarrow}0\\
\Longrightarrow \J{P_2(u)(a_n)}{P_2(u)(a_n)}u+\J{P_1(u)(a_n)}{P_1(u)(a_n)}u \overset{\mbox{weak}^*}{\longrightarrow}0.
\end{multline*}
\end{lemma}

\begin{proof}
Let us start by noticing that we can identify $M$ with a JB$^*$-subtriple of a unital JB$^*$-algebra $B$ in such a way that $p$ is a projection in $B$. This is proved in the first paragraph of the proof of \cite[Proposition 2.4]{BuFPMaMoPe} (using \cite[Corollary 1]{Friedman-Russo-GN} and \cite[Lemma 2.3]{BuFPMaMoPe}). (Note that $B$ can be assumed to be a JBW$^*$-algebra -- just pass to $B^{**}$.)\smallskip

Let us consider the bounded linear mapping $G: B\to B_2(u)$ defined by
$$G(x)=P_2(u)(x\circ u),\quad (x\in B).$$

Further, observe that for any $x\in B$ we have
$$\begin{aligned}\J xxu + \J {x^*}{x^*}u&=
(x\circ x^*)\circ u + (x^*\circ u)\circ x - (x\circ u)\circ x^* \\&\qquad + (x\circ x^*)\circ u + (u\circ x)\circ x^* - (x^*\circ u)\circ x
\\&=2(x\circ x^*)\circ u,\end{aligned}$$
thus
$$\begin{aligned}
P_2(u)(\J xxu + \J {x^*}{x^*}u)&= 2 P_2(u)((x\circ x^*)\circ u)
\\&=2 G(x\circ x^*)=2 G(\J xx1).
\end{aligned}$$
Further, given any $x\in B$, we have
$$\begin{aligned}
P_2(p)(\J xx1\circ u) & =
 P_2(p) \J {\J xx1}1u
\\&=\J{P_2(p)\J xx1}{P_2(p)(1)}u + \J{P_1(p)\J xx1}{P_1(p)(1)}u \\&=\J{P_2(p)\J xx1}{p}u
=\J{P_2(p)\J xx1}{1}u\\&= (P_2(p)\J xx1)\circ u=(P_2(p)\J xxp)\circ u.
\end{aligned}$$
Indeed, the first equality is obvious, the second one follows from Peirce arithmetic as $u\in M_2(p)\subset B_2(p)$. In the third equality we use the facts that $P_2(p)(1)=p$ and $P_1(p)(1)=0$.
The fourth equality follows by Peirce arithmetic using the fact that $1-p\in B_0(p)$, and thus $1-p\perp u$. The fifth equality is obvious and the sixth one follows from Lemma~\ref{L:properties of M2(e)}$(a)$.\smallskip

Thus, for each $x\in B$ we have
$$\begin{aligned}
G(\J xx1)&=P_2(u)(\J xx1\circ u)
\stackrel{\ref{P:M2 inclusion}(2)}{=}P_2(u)P_2(p)(\J xx1\circ u)\\&=
P_2(u)((P_2(p)\J xxp)\circ u)=G(P_2(p)\J xxp),\end{aligned}$$
so
$$P_2(u)(\J xxu + \J {x^*}{x^*}u)=2 G(P_2(p)\J xxp).$$

Let us check what happens in $M$. If $x\in M$, then
$$\begin{aligned}
 G(x)&=P_2(u)(x\circ u)=P_2(u)P_2(p)\J x1u
\\&=P_2(u)(\J{P_2(p)(x)}{P_2(p)(1)}u+\J{P_1(p)(x)}{P_1(p)(1)}u)\\&=P_2(u)\J{P_2(p)x}pu.
\end{aligned}$$
It follows that $G$ maps $M$ into $P_2(u)(M)=M_2(u)$ and, moreover, $G$ restricted to $M$ is weak$^*$-to-weak$^*$ continuous.\smallskip

So, assume $(a_n)$ is a bounded sequence in $M$ such that
$$\J{P_2(p)(a_n)}{P_2(p)(a_n)}p+\J{P_1(p)(a_n)}{P_1(p)(a_n)}p \overset{\mbox{weak}^*}{\longrightarrow}0,$$
equivalently,
$$P_2(p)\J{a_n}{a_n}p\overset{\mbox{weak}^*}{\longrightarrow}0.$$
Since this sequence lives in $M$, using weak$^*$-to-weak$^*$ continuity of $G$ we get
$$G(P_2(p)\J{a_n}{a_n}p)\overset{\mbox{weak}^*}{\longrightarrow}0,$$
and this sequence is contained in $M_2(u)$. Note that by the above calculation and Lemma~\ref{L:properties of M2(e)}(a) we have
$$\begin{aligned}
 G(P_2(p)\J{a_n}{a_n}p)&=\frac12 P_2(u)(\J {a_n}{a_n}u + \J {a_n^*}{a_n^*}u)\\&=
\frac12\Big(\J{P_2(u)(a_n)}{P_2(u)(a_n)}u+\J{P_1(u)(a_n)}{P_1(u)(a_n)}u
\\&\quad+\J{P_2(u)(a_n^*)}{P_2(u)(a_n^*)}u+\J{P_1(u)(a_n^*)}{P_1(u)(a_n^*)}u
\Big).\end{aligned}$$
Moreover, all the four summands in the right hand side are positive elements in the JB$^*$-algebra $B_2(u)$ by Lemma~\ref{L:properties of M2(e)}$(d)$, hence their sum is positive as well. Moreover, the sum belongs to $M_2(u)$ and the first two summands as well, and thus  $$\J{P_2(u)(a_n^*)}{P_2(u)(a_n^*)}u+\J{P_1(u)(a_n^*)}{P_1(u)(a_n^*)}u\in M_2(u),$$ too. Since
$$\begin{aligned}
0&
\le \J{P_2(u)(a_n)}{P_2(u)(a_n)}u+\J{P_1(u)(a_n)}{P_1(u)(a_n)}u \\&\le
\J{P_2(u)(a_n)}{P_2(u)(a_n)}u+\J{P_1(u)(a_n)}{P_1(u)(a_n)}u \\&\qquad + \J{P_2(u)(a_n^*)}{P_2(u)(a_n^*)}u+\J{P_1(u)(a_n^*)}{P_1(u)(a_n^*)}u,\end{aligned}$$
the equivalence $(\ref{it:strong* conv2})\Leftrightarrow(\ref{it:strong* conv3})$ in Lemma~\ref{L:strong* convergence for positive}$(2)$ shows that
$$\J{P_2(u)(a_n)}{P_2(u)(a_n)}u+\J{P_1(u)(a_n)}{P_1(u)(a_n)}u\overset{\mbox{weak}^*}{\longrightarrow}0,$$
which completes the proof.
\end{proof}

The following lemma together with Lemma~\ref{L:seminorm characterization} provides the proof for assertion $(ii)$ in Proposition~\ref{P:seminorms order}.

\begin{lemma}
Let $M$ be a JBW$^*$-triple and $e,u\in M$ two tripotents such that $M_2(u)\subsetneqq M_2(e)$. Then there is a bounded sequence $(a_n)$ in $M$ such that
$$\J{P_2(u)(a_n)}{P_2(u)(a_n)}u\overset{\mbox{strong}^*}{\longrightarrow}0\mbox{ and }\J{P_1(u)(a_n)}{P_1(u)(a_n)}u \overset{\mbox{strong}^*}{\longrightarrow}0$$
but
$$\J{P_2(e)(a_n)}{P_2(e)(a_n)}e+\J{P_1(e)(a_n)}{P_1(e)(a_n)}e \overset{\mbox{strong}^*}{\not\hskip-3pt\longrightarrow}0.$$
\end{lemma}

\begin{proof} If $M_0(u)\cap M_2(e)$ contains a nonzero element $a$, then
%$\J{P_2(u)(a)}{P_2(u)(a)}u+\J{P_1(u)(a)}{P_1(u)(a)}u=0$ but
$$\J{P_2(u)(a)}{P_2(u)(a)}u=  \J{P_1(u)(a)}{P_1(u)(a)}u=0$$ but
$$\J{P_2(e)(a)}{P_2(e)(a)}e+\J{P_1(e)(a)}{P_1(e)(a)}e=\J{a}{a}e\ne0$$
by Lemma~\ref{L:properties of M2(e)}$(d)$. It follows that the constant sequence $a_n=a$ works.\smallskip

Next assume that $M_0(u)\cap M_2(e)$ is trivial, hence $u$ is a complete tripotent in $M_2(e)$. Since $M_2(e)$ is a JBW$^*$-algebra, by Lemma~\ref{L:strong* topology}$(b)$ it is enough to consider the case in which $M=M_2(e)$. We shall therefore assume that $M$ is a JBW$^*$-algebra, $e=1$ and $u\in M$ is a complete non-unitary tripotent.\smallskip

Let $N$ denote the unital JB$^*$-subalgebra of $M$ generated by $u$. By Lemma~\ref{L:Jembed JB*} we can assume without loss of generality that  $N$ is a $JB^*$-subalgebra of $B(H)$ for a suitable complex Hilbert space $H$ and, moreover, $u^*u=1$ in $B(H)$. Since $u$ is not unitary, necessarily $uu^*\ne 1$.\smallskip

Set $q=uu^*$. Then $q$ is a projection in $B(H)$. Moreover, $q\in N$, as
$$q=uu^*=uu^*+u^*u-1=2 u\circ u^*-1.$$
Let us define a sequence in $N$ by
$x_0=1-q$ and $x_n=x_{n-1}\circ u^*$ for $n\in\en$. We claim that
$$x_n=2^{-n} (1-q) (u^*)^n,\qquad n\in\en\cup\{0\}.$$
Indeed, for $n=0$ the equality holds. Assume that $n\in\en$ and the equality holds for $n-1$. Then
$$\begin{aligned}
x_n&= x_{n-1}\circ u^*=\frac12(x_{n-1}u^*+u^*x_{n-1})=2^{-n}((1-q)(u^*)^n+u^*(1-q)(u^*)^{n-1})\\&=2^{-n}(1-q)(u^*)^n,
\end{aligned}$$
as obviously $u^*(1-q)=0$.\smallskip

Set $a_n=2^{n}x_n=(1-q)(u^*)^n$. Then $(a_n)$ is a bounded sequence in $N$, and hence in $M$. Further,
$$\begin{aligned}
\J{a_n}{a_n}1&=a_n\circ a_n^*=\frac12(a_n a_n^*+a_n^*a_n)\\&=
\frac12(u^n(1-q)(u^*)^n+(1-q)(u^*)^n u^n(1-q))\\&=\frac12(u^n(1-q)(u^*)^n+(1-q))\end{aligned}$$
as $u^*u=1$. Therefore,
$$\J{a_n}{a_n}1\ge 1-q,$$
and it then follows from Lemma~\ref{L:strong* convergence for positive} that $\J{a_n}{a_n}1$ does not converge to zero in the strong$^*$ topology.\smallskip

Further, observe that $q=P_2(u)(1)$, hence $q\in M_2(u)$ and $1-q\in M_1(u)$. We claim that $x_n\in M_1(u)$ for every $n\in \mathbb{N}\cup \{0\}$.
The case, $n=0$ is clear. The Peirce arithmetic yields by the induction hypothesis that
$$\begin{aligned}
 x_n&=x_{n-1}\circ u^*=\J{x_{n-1}}u1=\J{x_{n-1}}uq+\J{x_{n-1}}u{1-q}\\&=\J{x_{n-1}}uq\in M_1(u),\end{aligned}$$
where we used that $\J{x_{n-1}}u{1-q}\in M_0(u)=\{0\}$. So $a_n\in M_1(u)$ for each $n\in\en\cup\{0\}$ as well.
It follows by Lemma~\ref{L:properties of M2(e)}$(a)$ that
$$\begin{aligned}
P_2(u)\J{a_n}{a_n}u&=\J{a_n}{a_n}u=\frac12(a_na_n^*u+ua_n^*a_n)
\\&=
\frac12((1-q)(u^*)^nu^n(1-q)u+u u^n(1-q)(1-q)(u^*)^n)\\&=\frac12u^{n+1}(1-q)(u^*)^n\end{aligned}$$
as $(1-q)u=0$. We shall show that the sequence 	$\frac12u^{n+1}(1-q)(u^*)^n$ strong$^*$ converges to zero.\smallskip

To this end note that $u^n$ is a partial isometry for each $n\in\en$ and, moreover, $(u^n)^* u^n=1$. Thus its final projection $q_n=u^n (u^n)^*$ belongs to $N$. Let
$$Y=\overline{\bigcup_{n\in\en}\ker ((u^n)^*)}.$$
Then $Y$ is a closed subspace of $H$, and $\displaystyle \lim_n(u^*)^n(\xi) = 0$ for each $\xi\in Y$.\smallskip

Furthermore,
$$Y^\perp=\bigcap_{n\in\en}(\ker ((u^n)^*))^\perp=\bigcap_{n\in\en}(\ker q_n)^\perp=\bigcap_{n\in\en} q_n(H).$$
Thus for any $\xi\in Y^\perp$ and $n\in\en$ we have $\xi=q_n (\xi)$ and so
$$\begin{aligned}
(1-q)(u^*)^n(\xi)&=(1-q)(u^*)^nq_{n+1}(\xi)=
(1-q)(u^*)^n u^{n+1} (u^*)^{n+1}(\xi)\\&=(1-q)u(u^*)^{n+1}(\xi)
=(1-q)q (u^*)^n(\xi)=0.\end{aligned}$$
It follows that the sequence $(1-q)(u^*)^n$ SOT converges to zero, hence clearly
$(\frac12u^{n+1}(1-q)(u^*)^n)$ strong$^*$ converges to zero.
This completes the proof.
\end{proof}

\subsection{Weakly compact sets in the predual of a JBW$^*$-triple}
There is a close connection of the strong$^*$ topology, the generating seminorms and the weakly compact subsets of the predual. It is witnessed, for example, by the following proposition which is proved in \cite[Theorem D.21]{Rodriguez94}, see also \cite[Theorem 5.10.138]{Cabrera-Rodriguez-vol2}.

\begin{prop}\label{P:strong*=Mackey}
Let $M$ be a JBW$^*$-triple. The strong$^*$ topology on bounded subsets of $M$ coincides with the Mackey topology {\rm(}i.e., with the topology of uniform convergence on weakly compact subsets of $M_*${\rm)}.
\end{prop}

We shall analyze the relationship in more detail in the following lemma.

\begin{lemma}\label{L:normal functional}
Assume that $M$ is a JBW$^*$-triple, $e\in M$ is a tripotent, and $\varphi\in M_*$ satisfies $\norm{\varphi}=\varphi(e)$.
Define the mapping $\Phi=\Phi_{e,\varphi}:M\to M_*$ by
$$\Phi(a)(x)=\varphi(\J xae),\qquad x\in M, a\in M.$$
\begin{enumerate}[$(a)$]
\item $\Phi$ is a conjugate linear mapping of $M$ into $M_*$ which is moreover weak$^*$-to-weak continuous and $\norm{\Phi}\le\norm{\varphi}$;
\item Set $K=K(e,\varphi)=\Phi(B_M)\subset \|\varphi\| \ B_{M_*}$. Then $K$ is an absolutely convex weakly compact subset of $M_*$;
\item Let $a\in M$ be arbitrary. Then
$$\sup\{\abs{\psi(a)}\setsep \psi\in K\}=\norm{\Phi(a)};$$
\item For each $x\in B_M$ we have
$$\norm{x}_{e,\varphi}^2\le \norm{\Phi(x)}\le\sqrt{\norm{\varphi}}\cdot\norm{x}_{e,\varphi}.$$
In particular, the topologies induced by the seminorms $\norm{\cdot}_{e,\varphi}$ and $\norm{\Phi(\cdot)}$ coincide on $B_M$.
\end{enumerate}
\end{lemma}

\begin{proof}
$(a)$ The mapping
$$(x,y)\mapsto \varphi(\J xye),\qquad x,y\in M,$$
is a separately weak$^*$-continuous sesquilinear form on $M$ (cf. Lemma \ref{L:functionals and seminorms}). Indeed, the separate weak$^*$ continuity follows from the assumption $\varphi\in M_*$ together with the separate weak$^*$-to-weak$^*$ continuity of the Jordan product.\smallskip

It follows that for each $a\in M$ its image $\Phi(a)$ is a weak$^*$ continuous linear functional on $M$, hence $\Phi(a)\in M_*$. Further, $\Phi$ is clearly conjugate linear. The estimate of the norm is immediate from the inequality
$\norm{\{x,y,z\}}\le\norm{x}\norm{y}\norm{z}$ ($x,y,z\in M$) \cite[Corollary\ 3]{Friedman-Russo-GN}, \cite[Corollary\ 4.1.114]{Cabrera-Rodriguez-vol1}.
Finally, $\Phi$ is weak$^*$-to-weak continuous because for each $x\in (M_*)^*=M$ the mapping
$$a\mapsto\Phi(a)(x)=\varphi(\J xae)$$
is weak$^*$ continuous on $M$.\smallskip

$(b)$ This follows from $(a)$ as $B_M$ is weak$^*$ compact and absolutely convex.\smallskip

$(c)$ Let us compute:
$$\begin{aligned}
\sup\{\abs{\psi(a)}\setsep \psi\in K\}&
=\sup\{\abs{\Phi(x)(a)}\setsep x\in B_M\}
=\sup\{\abs{\varphi(\J axe)}\setsep x\in B_M\}
\\&=\sup\{\abs{\overline{\varphi(\J xae)}}\setsep x\in B_M\}
=\sup\{\abs{\overline{\Phi(a)(x)}}\setsep x\in B_M\}
\\&=\norm{\Phi(a)},\end{aligned}$$
where we used that the sesquilinear form from the proof of $(a)$ is
hermitian (because it is even positive semidefinite by Lemma~\ref{L:functionals and seminorms}$(d)$).
\smallskip

$(d)$ Fix any $x\in B_M$. Then
$$\norm{x}_{e,\varphi}^2=\varphi(\J xxe)=\Phi(x)(x)\le \norm{\Phi(x)},$$
which proves the first inequality. Further, for any $y\in B_M,$ by the Cauchy-Schwarz inequality, we have
$$\abs{\Phi(x)(y)}=\abs{\varphi(\J yxe)}\le \norm{x}_{e,\varphi}\cdot\norm{y}_{e,\varphi}\le \norm{x}_{e,\varphi}\cdot \sqrt{\norm{\Phi(y)}}\le \sqrt{\norm{\varphi}}\cdot\norm{x}_{e,\varphi}.$$
This proves the second inequality. The \textquoteleft in particular part\textquoteright then follows immediately.
\end{proof}

If $\varphi\in M_*\setminus\{0\}$, we set
$$\Phi_\varphi=\Phi_{s(\varphi),\varphi},\  K(\varphi)=K(s(\varphi),\varphi),\ \norm{\cdot}_\varphi^K=\norm{\Phi_{\varphi}(\cdot)},$$
where we use the notation from the previous lemma.\smallskip

We can next state
a characterization of relatively weakly compact subsets in the predual of a JBW$^*$-triple.

\begin{prop}\label{P:wcompact cofinal}
Let $M$ be a JBW$^*$-triple. Let $A\subset M_*\setminus\{0\}$ be such that the topology on $M$ generated by the family $\{\norm{\cdot}_\varphi:  \varphi\in A\}$ coincides on bounded sets with the strong$^*$ topology.
Then the following assertions are satisfied.
\begin{enumerate}[$(a)$]
\item Let $L\subset M_*$ be a weakly compact subset and $\varepsilon>0$. Then there are $\varphi_1,\dots,\varphi_k\in A$ and $n\in\en$ such that $$L\subset n\cdot \co(K({\varphi_1})\cup \dots \cup K({\varphi_k})) +\varepsilon B_{M_*};$$
\item Assume moreover that the family $\{M_2(s(\varphi))\setsep \varphi\in A\}$ is up-directed by inclusion. Then for any weakly compact set $L\subset M_*$ and any $\varepsilon>0$ there are $\varphi\in A$ and $n\in\en$ such that
     $L\subset n K(\varphi)+\varepsilon B_{M_*}$.
\end{enumerate}
In particular, the assumption is satisfied if
 $$\forall\psi\in M_*\setminus\{0\},\  \exists\varphi\in A: s(\psi)\in M_2(s(\varphi)).$$
\end{prop}

\begin{proof}
$(a)$ We may assume that $0<\varepsilon \leq 1$. We will use the following notation. For a bounded set $D\subset M_*$ denote by $q_D$ the seminorm on $M$ defined by
$$q_D(x)=\sup\{\abs{\varphi(x)}\setsep \varphi\in D\}.$$
Then $q_L$ is a Mackey continuous seminorm on $M$, so $q_L|_{B_M}$ is strong$^*$ continuous by Proposition~\ref{P:strong*=Mackey}.
The assumption together with Lemma~\ref{L:normal functional}$(d)$ yields the existence of $\varphi_1,\dots,\varphi_k\in A$  and a natural number $m$ such that for $\delta= \frac{1}{m}>0$ we have
$$\{x\in B_M\setsep \norm{x}_{\varphi_j}^K\le\delta\mbox{ for }j=1,\dots,k\}\subset
\{x\in B_M\setsep q_L(x)\le\varepsilon\},$$
hence
$$\{x\in B_M\setsep \norm{x}_{\varphi_j}^K\le\delta\mbox{ for }j=1,\dots,k\}_\circ\supset
\{x\in B_M\setsep q_L(x)\le\varepsilon\}_\circ.$$
Clearly
$$\{x\in B_M\setsep q_L(x)\le\varepsilon\}_\circ\supset\frac1\varepsilon L.$$
Further, by Lemma~\ref{L:normal functional}$(c)$ we have $\norm{\cdot}_{\varphi}^K=q_{K(\varphi)}$ for any $\varphi\in A$, hence
$$
\begin{aligned}
\{x\in B_M\setsep &\norm{x}_{\varphi_j}^K\le\delta\mbox{ for }j=1,\dots,k\}_\circ
=\left(B_M\cap \bigcap_{j\le k} \{x\in M\setsep q_{K(\varphi_j)}(x)\le \delta\}\right)_\circ
\\&=\left(B_M\cap\bigcap_{j\le k}\left(\frac1\delta K(\varphi_j)\right)^\circ\right)_\circ =\ \left(\left(B_{M_*}\cup\frac1\delta\bigcup_{j\le k} K(\varphi_j)\right)^\circ\right)_\circ\\&
\subset\overline{\frac1\delta\co(K(\varphi_1)\cup\dots\cup K(\varphi_k))+B_{M_*}}.
\end{aligned}$$
It follows that
$$\begin{aligned}
 L&\subset \overline{\frac\varepsilon\delta\co(K(\varphi_1)\cup\dots\cup K(\varphi_k))+\varepsilon B_{M_*}} \\&\subset m\cdot\co(K(\varphi_1)\cup\dots\cup K(\varphi_k))+2\varepsilon B_{M_*}, \end{aligned}$$
which completes the proof.\smallskip

$(b)$ We proceed in the same way as in the proof of $(a)$. We find $\varphi_1,\dots,\varphi_k$ and $\delta$. The assumption then yields $\varphi\in A$ such that $M_2(s(\varphi))$ contains $s(\varphi_1),\dots,s(\varphi_k)$.
By Proposition~\ref{P:seminorms order}$(i)$ and Lemma~\ref{L:normal functional}$(d)$ we get some $\eta>0$ such that
$$\begin{aligned}
\{x\in B_M\setsep \norm{x}_\varphi^K\le\eta\}&\subset\{x\in B_M\setsep \norm{x}_{\varphi_j}^K\le\delta\mbox{ for }j=1,\dots,k\}\\&\subset
\{x\in B_M\setsep q_L(x)\le\varepsilon\}.\end{aligned}$$
The arguments in the second part of the proof of $(a)$ complete the proof here.\smallskip

The `in particular' statement concerning the family $\{\norm{\cdot}_\varphi:  \varphi\in A\}$ follows from  Proposition~\ref{P:seminorms order}$(i)$.
\end{proof}

\section{Proof of the main result}\label{sec:8}

In this section we provide a proof of Theorem~\ref{T:main}. We begin with a technical lemma.

\begin{lemma}\label{L:sequence of projections}
Let $M$ be a JBW$^*$-algebra and let $(p_n)$ be an increasing sequence of projections in $M$ with supremum $p$. Then for any bounded sequence $(a_k)$ in $M$ we have
$$\mbox{strong$^*$-}\lim_{k} P_2(p)\J {a_k}{a_k}p =0
\Leftrightarrow \forall n\in\en,\  \mbox{strong$^*$-}\lim_{k} P_2(p_n)\J {a_k}{a_k}{p_n} =0.$$
\end{lemma}

\begin{proof}
($\Rightarrow$) Fix $n\in\en$. By  Lemma~\ref{L:properties of M2(e)}$(a)$ we have
$$P_2(p_n) \{a_k,a_k,p_n\} = P_2(p_n) \{a_k,a_k,p\} = P_2(p_n) P_2(p) \{a_k,a_k,p\}.$$
Thus we can conclude by the strong$^*$-to-strong$^*$ continuity of $P_2(p_n)$.
\smallskip

$(\Leftarrow)$ Arguing by contradiction, we assume that $(a_k)$ is a bounded sequence in $M$ such that
$$\forall n\in\en: \mbox{strong$^*$-}\lim_{k} P_2(p_n)\J {a_k}{a_k}{p_n} =0$$
 but $P_2(p)\J {a_k}{a_k}p \overset{\mbox{strong}^*}{\not\hskip-3pt\longrightarrow}0$.  We may assume, without loss of generality that $(a_k)\subseteq B_M$.
Since
$$P_2(p)\J{a_k}{a_k}p=\J{P_2(p)(a_k)}{P_2(p)(a_k)}p+\J{P_1(p)(a_k)}{P_1(p)(a_k)}p$$
is a positive element of $M$ (by Lemma~\ref{L:properties of M2(e)}$(a)$,$(d)$), there is, due to Lemma~\ref{L:strong* convergence for positive}, a positive norm-one functional $\varphi\in M_*$ (i.e., a normal state on $M$) such that
$\varphi(P_2(p)\J{a_k}{a_k}p)\not\to0$. Up to passing to a subsequence we may assume that there is some $c>0$ such that
$$\varphi(P_2(p)\J{a_k}{a_k}p)>c,\mbox{ for all }k\in\en.$$
By \cite[Lemma 3.2]{BHK-JBW} there is some $m\in\en$ with
$\norm{P_2(p)^*\varphi-P_2(p_m)^*\varphi}<\frac c2$. Then
$$\begin{aligned}
\varphi(P_2(p_m)\J {a_k}{a_k}{p_m})&= \varphi(P_2(p_m)\J {a_k}{a_k}{p})
\\& >c + \varphi(P_2(p_m)\J {a_k}{a_k}{p})-\varphi(P_2(p)\J {a_k}{a_k}{p})
\\&=c + P_2(p_m)^*\varphi  \J {a_k}{a_k}{p} -  P_2(p)^*\varphi  \J {a_k}{a_k}{p}
\\&\ge c- \norm{P_2(p_m)^*\varphi-P_2(p)^*\varphi}>\frac c2
\end{aligned}$$
for all $k\in\mathbb{N}$.
Thus, Lemma~\ref{L:strong* convergence for positive} implies that $(P_2(p_m)\J {a_k}{a_k}{p_m})_k \overset{\mbox{strong}^*}{\not\hskip-3pt\longrightarrow}0$, leading to a contradiction.
\end{proof}

\begin{lemma}\label{L:sequence of seminorms}
Let $M$ be a $\sigma$-finite JBW$^*$-algebra and let $(\varphi_n)$ be a sequence of nonzero positive functionals in $M_*$ such that their support projections $s(\varphi_n)$ form an increasing sequence with supremum $1$. Then the strong$^*$ topology on bounded subsets of $M$ coincides with the topology generated by the seminorms $\norm{\cdot}_{\varphi_n}$, $n\in\en$.
\end{lemma}

\begin{proof} Since $M$ is $\sigma$-finite, there exists a normal state $\varphi\in M_*$ with $s(\varphi)=1$ and, moreover, the norm $\norm{\cdot}_\varphi$ generates the strong$^*$ topology on bounded sets of $M$ (cf. Lemma~\ref{L:JBW*algebra sigma finite}).
Hence we can conclude using Lemmata \ref{L:seminorm characterization} and \ref{L:sequence of projections}.
\end{proof}

\begin{lemma}\label{L:omega on JBW*algebra}
Let $M$ be a JBW$^*$-algebra and $A\subset M_*$ a bounded set. Then there is a countable set $B\subset A$ such that $\omega(B')=\omega(A)$ for any $B'\subset B$ infinite.
\end{lemma}

\begin{proof}
For any $\sigma$-finite projection $p\in M$, the Peirce-2 subspace $M_2(p)$ is a $\sigma$-finite JBW$^*$-algebra, hence by Lemma~\ref{L:JBW*algebra sigma finite} we can fix a faithful normal state $\omega_p$ on $M_2(p)$. Let us set $\varphi_p=\omega_p\circ P_2(p)$. Then $\varphi_p$ is a normal positive functional on $M$ such that $\norm{\varphi_p}=\varphi_p(p)=1$ and $\varphi_p|_{M_2(p)}$    is faithful.\smallskip

Let $\Phi_p=\Phi_{p,\varphi_p}$ using the notation from Lemma~\ref{L:normal functional}.
Let $K_p=K(p,\varphi_p)=\Phi_p(B_M)$. Then $K_p$ is a weakly compact set in $M_*$ (by Lemma~\ref{L:normal functional}(b)).\smallskip

Let $A\subset M_*$ be a bounded set such that $c=\omega(A)>0$. Let us construct, by induction, two sequences $(\gamma_n)\subset A$ and $(p_n)\subset M$ such that
\begin{enumerate}[$(i)$]
\item $\norm{\gamma_1}>c-1$;
\item $p_n$ is a $\sigma$-finite projection such that $\gamma_n=\gamma_n\circ P_2(p_n)$;
\item $p_n\ge p_k$ for $k<n$;
\item $\dist(\gamma_{n+1}, n\co(K_{p_1}\cup\dots\cup K_{p_n}))> c-\frac{1}{n+1}$.
\end{enumerate}
This construction can be done by just applying the definition of $\omega(A)$. Indeed, the existence of $\gamma_1\in A$ satisfying (i) is obvious. Assume that $n\in\en$ and we have already constructed $\gamma_j$ for $1\le j\le n$ and $p_j$ for $1\le j<n$.

By \cite[Lemma 3.6]{BHK-JBW} there is a $\sigma$-finite projection $r\in M$ such that $\gamma_n=\gamma_n\circ P_2(r)$. By  \cite[Lemma 3.5]{BHK-JBW}, there is a $\sigma$-finite projection $p_n\ge r$ satisfying (iii). Clearly $p_n$ satisfies (ii) as well. Finally, find $\gamma_{n+1}\in A$ satisfying (iv) by the definition of $\omega(A)$.\smallskip

Set $B=\{\gamma_n\setsep n\in\en\}$. Then $B$ is a countable subset of $A$ and we claim that $\omega(B')=c$ for each infinite subset $B'\subset B$.\smallskip

Let $p=\sup_n p_n$. Then $p$ is $\sigma$-finite (see, e.g., \cite[Theorem 3.4]{Edw-Rut-exposed} or \cite[Lemma 3.5]{BHK-JBW}) and $B\subset P_2(p)^*M_*$. Since $P_2(p)$ is a norm-one projection, an application of Lemma~\ref{L:measures 1-complemented}  shows that $\omega(B)=\omega_{M_2(p)}(B)$. We continue by working in the JBW$^*$-algebra $M_2(p)$.\smallskip

Lemma~\ref{L:sequence of seminorms} implies that the strong$^*$ topology on $B_{M_2(p)}$ is generated by the sequence of seminorms $\norm{\cdot}_{\varphi_{p_n}|_{M_2(p)}}$. Let $L\subset (M_2(p))_*=P_2(p)^*M_*$ a weakly compact set and $\varepsilon>0$. By Proposition~\ref{P:wcompact cofinal}(b) there are $m,n\in\en$ such that
$$L\subset n \Phi_{p_m}(B_{M_2(p)})+\varepsilon B_{M_2(p)_*}
\subset n(K_{p_m}\cap P_2(p)^*M_*)+\varepsilon B_{M_2(p)_*}.$$
It follows that
$$\dh(B,L)\ge \dh(B, n K_{p_m})-\varepsilon \ge \dh (B,k K_{p_m})-\varepsilon\ge c-\frac1{k+1}-\varepsilon$$
for each $k\ge \max\{n,m\}$. Hence $\dh(B,L)\ge c-\varepsilon$.
Since $L$ is an arbitrary weakly compact set, we get $\omega(B)\ge c-\varepsilon$, and by the arbitrariness of $\varepsilon>0$  we have
$\omega(B)\ge c$.\smallskip

The same procedure applies to each infinite subset $B'\subset B$, so the proof is completed.
\end{proof}

\begin{proof}[Proof of Theorem~\ref{T:main}]\label{eq proof of Theorem 4.1}
If $M$ is a JBW$^*$-algebra, the result follows from Lemma~\ref{L:omega on JBW*algebra}, \cite[Corollary~4.3]{FePeRa2015}, and Proposition~\ref{p splitting}$(b)$.
The general case of a JBW$^*$-triple follows from the JBW$^*$-algebra case applying Lemma~\ref{L:measures 1-complemented} and Proposition~\ref{P:triple 1-complemented}.
\end{proof}

\section{Preduals of JBW$^*$-triples which are strongly WCG}\label{sec:9}

Strongly WCG spaces (see the end of Section~\ref{sec:ssp} for definitions) are a nice class of Banach spaces in which the computation of the De Blasi measure of weak non-compactness is easy. As explained in the end of Section~\ref{sec:ssp} they include the spaces $L^1(\mu)$ for a $\sigma$-finite measure $\mu$ or, more generally, preduals of $\sigma$-finite von Neumann algebras and preduals of $\sigma$-finite JBW$^*$-algebras. In the present section we characterize JBW$^*$-triples whose preduals are strongly WCG. \smallskip

Let us explain why it is not clear. By \cite[Theorem 2.1]{SWCG}
a Banach space $X$ is strongly WCG if and only if the Mackey topology on $X^*$ is metrizable on bounded sets. Therefore, it follows from Proposition~\ref{P:strong*=Mackey} that the predual $M_*$ of a JBW$^*$-triple $M$ is strongly WCG if and only if the strong$^*$ topology on $B_M$ is metrizable. For $\sigma$-finite JBW$^*$-algebras this is the case by Lemma~\ref{L:JBW*algebra sigma finite}. However, for $\sigma$-finite JBW$^*$ triples it need not be the case as witnessed by the following example.
 \smallskip

\begin{example}\label{l example where the strong* topology is not metrizable on bounded sets}
Let $\Gamma$ be an uncountable set and $C\subset\Gamma$ be an infinite countable set. Then $M=B(\ell^2(\Gamma),\ell^2(C))$ is a $\sigma$-finite JBW$^*$-triple whose strong$^*$ topology is not metrizable on the closed unit ball of $M$.
\end{example}

\begin{proof}
A family of pairwise orthogonal partial isometries of $M=B(\ell^2(\Gamma),\ell^2(C))$
has pairwise orthogonal final projections and is therefore countable as $\ell^2(C)$ is separable.
Hence $M$ is $\sigma$-finite.
$M$ is $1$-complemented in $B(\ell^2(\Gamma))$ by a weak$^*$-to-weak$^*$ continuous projection, thus by \cite[\S 1.15]{sakai2012c} and \cite[COROLLARY]{bunce2001norm}, the strong$^*$ topology on bounded sets of $M$ is given by the seminorms
$a\mapsto \norm{a(e_\gamma)}+ \norm{a^*(e_\gamma)}$ ($a\in M$), where $\gamma$ is a fixed element in $\Gamma$ and $\{e_{\gamma}:\gamma\in \Gamma\}$ is the canonical orthonormal basis of $\ell^2(\Gamma)$.\smallskip

If the strong$^*$ topology of $B_M$ were metrizable, it would be first countable, hence there would exist a countable set $D\subset \Gamma$ such that the seminorms
$a\mapsto \norm{a(e_\gamma)}+\norm{a^*(e_\gamma)}$, $\gamma\in D$, generate this topology. We may assume without loss of generality that $D\supset C$. Let $C=(c_n)_{n\in\en}$. Let $\{\gamma_n: n\in\en\}$ be a set of pairwise distinct elements in $\Gamma\setminus D$. Define the sequence of operators $a_k\in M$ by
$$a_k (e_\gamma)=\begin{cases}
e_{c_{n+k}} & \mbox{ if }\gamma=\gamma_n\\
0&\mbox{ otherwise.}
\end{cases}
$$
Then
$$a_k^*(e_\gamma)=\begin{cases}
e_{\gamma_{n-k}} & \mbox{ if }\gamma=c_n\mbox{ and }n>k\\
0&\mbox{ otherwise.}
\end{cases}$$

In this case $\norm{a_k (e_{\gamma_n})}=1$, so $a_k$ do not converge strong$^*$ to zero. However $a_k^*\to 0$ in SOT and $a_k(e_\gamma)=0$ for $\gamma\in D$, so $\norm{a_k (e_\gamma)}+\norm{a_k^*(e_\gamma)}\to 0$ for $\gamma\in D$, leading to a contradiction.
\end{proof}

On the other hand, $\sigma$-finiteness of a JBW$^*$-triple is a necessary condition for its predual to be strongly WCG. Indeed, any strongly WCG space is clearly WCG and the predual of a JBW$^*$-triple is WCG if and only if the triple is $\sigma$-finite by \cite[Theorem 1.1]{BHKPP-triples}. \smallskip

In order to find a sufficient and necessary condition we get back to the structure results of JBW$^*$-triples due to G. Horn and E. Neher presented in \eqref{eq decomposition of JBW*-triples} in page \pageref{eq decomposition of JBW*-triples} (see \cite[(1.7)]{horn1987classification}, \cite[(1.20)]{horn1988classification}). Every JBW$^*$-triple $M$ decomposes (uniquely) as an (orthogonal) $\ell_{\infty}$-sum of the form $M = \left(\bigoplus_{j\in \mathcal{J}} A_j \overline{\otimes} C_j \right)_{\ell_{\infty}}
\oplus_{\ell_{\infty}} H(W,\alpha)\oplus_{\ell_{\infty}} pV,$ where each $A_j$ is a commutative von Neumann algebra, each $C_j$ is a Cartan factor, $W$ and $V$ are continuous von Neumann algebras, $p$ is a projection in $V$, $\alpha$ is {a linear} involution on $W$ commuting with $^*$, that is, a linear $^*$-antiautomorphism of period 2 on $W$, and $H(W,\alpha)=\{x\in W: \alpha(x)=x\}$. Clearly, $H(W,\alpha)$ is a JBW$^*$-subalgebra of $W$ when the latter is equipped whit its natural structure of JBW$^*$-algebra.\smallskip

A Cartan factor of type 1 is a JBW$^*$-triple $C_1$ which coincides with the space $B(H,K)$ of all bounded linear operators between two complex Hilbert spaces $H$ and $K$. We can always assume that $K$ is a closed subspace of $H$. Therefore, denoting by $p$ the orthogonal projection of $H$ onto $K$, we have $C_1=B(H,K) =p B(H)$. Suppose $A$ is a commutative von Neumann algebra. Now taking $\widehat{p} = 1\otimes p\in A \overline{\otimes} C_1,$ we deduce that $A \overline{\otimes} C_1= \widehat{p} \left(A \overline{\otimes} B(H)\right)$ is a right ideal of the von Neumann algebra $ A \overline{\otimes} B(H)$.\smallskip

Cartan factors of types 2 and 3 are the subtriples of $B(H)$ defined by $C_2 = \{ x\in B(H) : x=- j x^* j\} $ and
$C_3 = \{ x\in B(H) : x= j x^* j\}$, respectively, where $j$ is a conjugation (i.e., a conjugate-linear isometry of period $2$) on $H$. By a little abuse of notation, each $x\in B(H)$ can be identified with a ``\emph{matrix}'' $(x_{\gamma\delta})_{\gamma,\delta\in\Gamma}$. It is easy to check that the representing matrix of $jx^*j$ is the transpose of the representing matrix of $x$. Hence, $C_2$ consists of operators with antisymmetric representing matrix and $C_3$ of operators with symmetric ones.\smallskip

The properties around Peirce decomposition show that, if a JBW$^*$-triple $M$ admits a unitary element $u$, then $M= M_2(u)$ is a JBW$^*$-algebra with product $\circ_u$ and involution ${*_u}$ (cf. page \pageref{eq Peirce2 is a JBstaralgebra}).
It is shown in the proof of \cite[Proposition 2]{ho2002derivations} that every Cartan factor of type 2 with dim$(H)$ even, or infinite, and every Cartan factor of type 3 contains a unitary element. The same result actually proves that Cartan factors of type 2 with dim$(H)$ even, or infinite, and all Cartan factors of type 3 are JBW$^*$-algebras. Consequently, if $C$ is a Cartan factor of type 2 with dim$(H)$ even, or infinite, or a Cartan factor of type 3, and $A$ is a commutative von Neumann algebra, then $A\overline{\otimes} C$ is a JBW$^*$-algebra.\smallskip

A Cartan factor of type 4 (also called a \emph{spin factor}) is a complex Hilbert space (with inner product $\langle .,. \rangle$) equipped with a conjugation $x \mapsto \overline{x}$, triple product $$\J x y z
= \langle x , y \rangle z + \langle z , y \rangle x - \langle x ,
\overline{z} \rangle  \overline{y},$$ and norm given by $\| x\|^2=\langle x , x
\rangle+\sqrt {\langle x , x \rangle^2-|\langle x , \overline x
\rangle|^2}$. Let $u$ be an element in a spin factor $C_4$ satisfying $u = \overline{u}$ and $\|u\| = \langle u,u\rangle =1$. It is not hard to check that $\{u,u,x\} = \langle u , u \rangle x + \langle x , u \rangle u - \langle u ,
\overline{x} \rangle  \overline{u}= \langle u , u \rangle x + \langle x , u \rangle u - \langle x ,
u \rangle  u =x,$ for all $x\in C_4$. This shows that $u$ is a unitary in $C_4$, and consequently $A\overline{\otimes} C_4$ is a JBW$^*$-algebra whenever $A$ is a commutative von Neumann algebra.\smallskip

Finally, assume that $C$ is a finite-dimensional Cartan factor and $A$ is a commutative von Neumann algebra.
Then $A$ is isomorphic to $\bigoplus_{j\in J}^{\ell_\infty} L^\infty(\mu_j)$, where $(\mu_j)_{j\in J}$ is a family of finite (or, equivalently, probability) measures (cf. \cite[\S 1.18]{sakai2012c}). Thus
the JBW$^*$-triple $A\overline{\otimes} C$ can be identified with $\bigoplus_{j\in J}^{\ell_\infty} L^\infty(\mu_j,C)$, (cf. \cite{horn1984dissertation,horn1987classification}). Let us observe that the remaining Cartan factors, that is, the exceptional Cartan factors of types 5 and 6, are all finite-dimensional (they have dimensions 16 and 27, respectively).\smallskip

Combining the arguments in the previous paragraphs we get the following representation of JBW$^*$-triples.

\begin{prop}\label{P:representation} Let $M$ be any JBW$^*$-triple. Then $M$ is (isometrically) JB$^*$triple isomorphic to a JBW$^*$-triple of the form
\begin{equation}\label{eq decomp in L 9.1} \left( \bigoplus_{k\in\Lambda_1}^{\ell_{\infty}} L^\infty(\mu_k,C_k)\right)\bigoplus^{\ell_{\infty}} \left(\bigoplus_{j\in\Lambda_2}^{\ell_{\infty}} L^\infty(\mu_j,D_j)\right)\bigoplus^{\ell_{\infty}} N\bigoplus^{\ell_{\infty}} pV,
\end{equation}
where
\begin{enumerate}[$\bullet$]
    \item $(\mu_k)_{k\in\Lambda_1}$ and $(\mu_j)_{j\in\Lambda_2}$ are two (possibly empty) families of probability measures;
    \item Each $C_k$ is a Cartan factor of type 5 or 6 for any $k\in\Lambda_1$, and each $D_j$ is a finite-dimensional Cartan factor of type 2 with dim$(H)\in \mathbb{N}$ odd for any $j\in\Lambda_2$;
    \item $N$ is a JBW$^*$-algebra;
    \item $V$ is a von Neumann algebra and $p\in V$ is a projection such that the triple $pV$ has no nonzero direct summand triple-isomorphic to a JBW$^*$-algebra.
\end{enumerate}
Moreover, such a representation is unique, in the sense that if $M$ admits two such representations, the respective four summands in one of them are triple-isomorphic to the respective four summands in the second one.
\end{prop}

Thanks to the structure result in the previous proposition, the promised characterization of JBW$^*$-triples is now stated in the following theorem.

\begin{thm}\label{T:SWCG}
Let $M$ be a JBW$^*$-triple. Consider its representation provided by Proposition~\ref{P:representation}.
\begin{enumerate}[$(a)$]
\item $M$ is $\sigma$-finite if and only if the sets $\Lambda_1$ and $\Lambda_2$ are countable, the JBW$^*$-algebra $N$ is $\sigma$-finite and the projection $p$ is $\sigma$-finite.
\item $M_*$ is WCG if and only if $M$ is $\sigma$-finite.
\item The following assertions are equivalent.
\begin{enumerate}[$(i)$]
    \item $M_*$ is strongly WCG.
    \item $M$ is $\sigma$-finite and, moreover, the projection $p$ is finite.
    \item There is $\varphi\in M_*\setminus\{0\}$ such that the strong$^*$ topology on $B_M$ is generated by $\norm{\cdot}_\varphi$.
    \item There is a $\sigma$-finite tripotent $u\in M$ whose Peirce-2 subspace $M_2(u)$ is maximal with respect to inclusion.
\end{enumerate}
\end{enumerate}
\end{thm}

Assertions $(a)$ and $(b)$ follow from \cite{Edw-Rut-exposed} and \cite{BHKPP-triples}, respectively.
More concretely, the \textquoteleft only if part\textquoteright in $(a)$ is obvious; to see the `if part' it is enough to use the known fact that exceptional Cartan factors are finite-dimensional and every $D_j$ is finite-dimensional too, hence each of the summands $L^\infty(\mu_\alpha,C_\alpha)$ and $L^\infty(\mu_j,D_j)$ is $\sigma$-finite (cf. \cite[Theorem 4.4]{Edw-Rut-exposed}). Assertion $(b)$ follows from \cite[Theorem 1.1]{BHKPP-triples}.
\smallskip

It remains to prove $(c)$. In view of $(b)$ we may restrict our attention to the $\sigma$-finite case. Let us observe that some implications in $(c)$ are easy at this point. Indeed, $(iii)$ implies that the strong$^*$ topology on $B_M$ is metrizable, hence we get $(iii)\Rightarrow(i)$. Further, $(iii)\Rightarrow(iv)$ follows from Proposition~\ref{P:seminorms order}(ii).
Recall that a projection $p$ in a von Neumann algebra $V$ is finite if there is no partial isometry in $V$ with final projection $p$ and initial projection stricly less than $p$.
\smallskip
The argument will follow after considering the individual summands in the representation.
However, we first give the following corollary on JBW$^*$-triples with separable predual. Note that while any separable Banach space is trivially WCG, $c_0$ is an example of a separable space which is not strongly WCG by \cite[Theorem 2.5]{SWCG}. A similar example cannot be a predual of a JBW$^*$-triple.\smallskip

\begin{cor}
Let $M$ be a JBW$^*$-triple with separable predual $M_*$. Then $M_*$ is strongly WCG.
\end{cor}

\begin{proof} First observe that $M$ is $\sigma$-finite. Indeed, being separable, $M_*$ is WCG, thus $M$ is $\sigma$-finite by Theorem~\ref{T:SWCG}(b).
(There is also an alternative way of proving
this. Assume that $M_*$ is separable and fix $e\in M$ a complete tripotent.
Then $M_2(e)_*$ is also separable.  Since $M_2(e)$ is a JBW$^*$-algebra, we can choose a countable family of normal states $\{\varphi_n : n\in \mathbb{N}\}$ which is norm-dense in the set of normal states of $M_2(e)$. Then $\displaystyle \sum_{n=1}^{\infty}\frac{1}{2^n} \varphi_n$ is a faithful normal state of $M_2(e)$. Therefore, $M_2(e)$ is $\sigma$-finite, so $e$ is $\sigma$-finite and $M$ is $\sigma$-finite as well.)

So, assume $M_*$ is separable and fix a representation of $M$ given by Proposition~\ref{P:representation}. It follows that $(pV)_*$ is separable as well.
Without loss of generality there is no nonzero central projection in $V$ orthogonal to $p$ (if $z$ is such a projection, then $pV=p(1-z)V$). We claim that in this case necessarily $V$ is $\sigma$-finite. Assume it is not the case. Then there is an uncountable family of pairwise orthogonal nonzero projections $(r_\gamma)_{\gamma\in\Gamma}$ in $V$. It follows from \cite[Theorem V.1.8]{Tak} that for each $\gamma\in\Gamma$ there is a nonzero partial isometry $u_\gamma\in V$ such that its initial projection $p_i(u_\gamma)\le r_\gamma$ and its final projection $p_f(u_\gamma)\le p$. Then clearly $u_\gamma\in pV$ for $\gamma\in\Gamma$.
Fix $\varphi_\gamma\in (pV)_*$ of norm one with $u_\gamma=s(\varphi_\gamma)$. Since $\varphi_\gamma(u_\gamma)=1$ and for $\delta\ne\gamma$
$$\varphi_\gamma(u_\delta)=\varphi_\gamma P_2(u_\gamma)(u_\delta)=0,$$
we see that $(\varphi_\gamma)_{\gamma\in\Gamma}$ is a $1$-discrete set, contradicting the separability of $(pV)_*$.

Hence the strong$^*$ topology on $B_V$ is metrizable, so by Lemma~\ref{L:strong* topology}$(b)$ the same holds for $B_{pV}$,
thus $(pV)_*$ is strongly WCG. Using Theorem~\ref{T:SWCG} we now see that $p$ is finite and hence $M_*$ is strongly WCG as well.
\end{proof}

To prove assertion $(c)$ in Theorem~\ref{T:SWCG} we will describe the structure of all preHilbertian seminorms generating the strong$^*$ topology using Proposition~\ref{P:seminorms order} and some complements to that. We will do it first for $\sigma$-finite triples and then (in the next section) we shall discuss the general case.
We start by analyzing the individual summands appearing in Proposition \ref{P:representation}.

\subsection{JBW$^*$-algebras}

In the case of JBW$^*$-algebras we can conclude by applying the existing literature. The desired conclusion is covered by the following proposition.

\begin{prop}\label{p JBWstar algebras sigma finite}
Let $M$ be a JBW$^*$-algebra.
\begin{enumerate}[$(a)$]
    \item Let $(e_n)$ be a sequence of $\sigma$-finite tripotents in $M$. Then there is a $\sigma$-finite projection $p\in M$ such that $M_2(p)$ contains $e_n$ for each $n\in\en$.
    \item Assume $M$ is not $\sigma$-finite. Then for each $\sigma$-finite projection $p\in M$ there is a $\sigma$-finite projection $q\in M$ such that $q>p$ (and hence $M_2(p)\subsetneqq M_2(q)$).
    \item The strong$^*$ topology on $B_M$ is metrizable if and only if $M$ is $\sigma$-finite. In this case it is metrizable by $\norm{\cdot}_\omega$, where $\omega$ is any faithful normal state.
\end{enumerate}
\end{prop}

\begin{proof}
$(a)$ This follows from Lemma~\ref{L:sigma finite tripotent}$(d)$ and \cite[Lemma 3.5]{BHK-JBW}.\smallskip

$(b)$ This follows easily from the definitions. If $M$ is not $\sigma$-finite and $p$ is $\sigma$-finite, then $1-p\ne0$, hence there is a $\sigma$-finite projection $r\in M_2(1-p)$. It is enough to take $q=p+r$.\smallskip

$(c)$ The `if part' follows from Lemma~\ref{L:JBW*algebra sigma finite}. Conversely, assume that $B_M$ is metrizable in the strong$^*$ topology. Then there is a countable base of strong$^*$ neighborhoods of zero in $B_M$. It follows that the strong$^*$ topology on $B_M$ is generated by countably many seminorms. By (a) and Proposition~\ref{P:seminorms order}(i)  it is generated by one seminorm. By (b) and Proposition~\ref{P:seminorms order}(ii) we deduce that $M$ is $\sigma$-finite.
\end{proof}

\subsection{Finite dimensional Cartan factors}

In this subsection we shall deal with summands of the form $L^\infty(\mu,C)$ where $C$ is an exceptional Cartan factor (i.e., the Cartan factor of type $5$ or $6$) or a finite-dimensional Cartan factor of type 2 with dim$(H)\in \mathbb{N}$ odd.
We start with properties of a finite-dimensional JB$^*$-triple.\smallskip

Let $E$ be a JB$^*$-triple. Following the most employed notation, the symbol $\mathcal{U}(E)$ will stand for the set of all tripotents in $E$. We shall write $\mathcal{U}(E)^*$ for the set of all nonzero tripotents in $E$, and we shall employ the symbol $\mathcal{U}_{max} (E)$ to denote the set of all complete tripotents in $E$. By Kaup's Riemann mapping theorem \cite[Proposition 5.5]{kaup1983riemann}, a linear bijection between JB$^*$-triples $E$ and $F$ is a triple isomorphism if and only if it is an isometry. Henceforth, we denote by Iso$(E,F)$ the set of all surjective isometries (equivalently, triple isomorphisms) from $E$ to $F$. We write Iso$(E) =$Iso$(E,E)$ for the set of all triple automorphisms of $E$.\smallskip

Fix $\Phi\in\hbox{Iso}(E)$. Then $\Phi$, being a JB$^*$-triple automorphism, preserves all the triple structure. In particular, it maps tripotents to tripotents and complete tripotents to complete tripotents, that is, \begin{equation}\label{eq invariance of sets of tripotnets under triple automorphisms}  \Phi(E) (\mathcal{U}(E)^*) = \mathcal{U}(E)^*, \hbox{ and } \Phi(E) (\mathcal{U}_{max}(E)) = \mathcal{U}_{max}(E).
\end{equation} Moreover, the equality $\Phi(P_j(e)(x))=P_j(\Phi(e))(\Phi(x))$ holds for every $e\in\mathcal{U}(E),$ $j\in\{0,1,2\}$ and $x\in E$. In particular, $\Phi(E_2(e))=E_2(\Phi(e))$ and $\Phi$ is a (unital) JB$^*$-algebra isomorphism of $E_2(e)$ onto $E_2(\Phi(e))$. Let us fix $e\in\mathcal{U}(E)$ and $\varphi\in E^*$ a functional satisfying $\varphi = \varphi P_2(e)$. Then $\varphi\circ \Phi^{-1}=\varphi\circ P_2(e)\circ \Phi^{-1}=\varphi\circ \Phi^{-1}\circ P_2(\Phi(e))$ and $(\varphi\circ \Phi^{-1})|_{E_2(\Phi(e))}=\varphi|_{E_2(e)}\circ \Phi^{-1}$.\smallskip

It is natural to ask about the orbit of a fixed $e\in \mathcal{U}_{max}(E)$ under the group $\hbox{Iso}(E)$.
In general, $\hbox{Iso}(E) (e)$ is not easy to be determined (cf. \cite{braun1978holomorphic} and \cite{kaup1997real}). If $E$ is a finite-dimensional JB$^*$-triple,
then any two complete (maximal) tripotents in $E$ are interchanged by an element in $\hbox{Iso}(E)$ (see \cite[Theorem 5.3$(b)$]{loos1977bounded}).
This can be also seen by applying that $E$ being finite-dimensional implies that $E$ coincides with a finite $\ell_{\infty}$-sum of finite-dimensional Cartan factors,
and it is known that on a finite-dimensional Cartan factor $C$ the group  $\hbox{Iso}(C)$ acts transitively on $\mathcal{U}_{max}(C)$.
Therefore, for $e\in \mathcal{U}_{max}(E)$ and dim$(E)<\infty$ we have
\begin{equation}\label{eq orbit of a complete tripotent in finite dimension}  \hbox{Iso}(E) (e) = \mathcal{U}_{max}(E).
\end{equation}

\begin{lemma}\label{L:exceptional}
Let $E$ be a finite-dimensional JB$^*$-triple, let $e\in \mathcal{U}_{max} (E),$ and let $\varphi\in E_*$ be a norm-one functional such that $e=s(\varphi)$. Then the following assertions hold:
\begin{enumerate}[$(a)$]
\item\label{it:exceptional1} For each $\Phi\in\hbox{Iso}(E)$ we have $\Phi(e)\in \mathcal{U}_{max}(E)$ and $\Phi(e)=s(\varphi\circ \Phi^{-1})$;
\item\label{it:exceptional2} $\mathcal{U}(E)$, $\mathcal{U}(E)^*$, and $\mathcal{U}_{max} (E)$ are compact subsets of $E$ and $\hbox{Iso}(E)$ is a compact subset of $B(E)$;
\item\label{it:exceptional3} There is a constant $\alpha >0$ such that for each $\Phi\in \hbox{Iso}(E)$
    we have
    $$\alpha\norm{x}\le\norm{x}_{\varphi\circ \Phi^{-1}}\le \norm{x},\quad x\in E.$$
\item\label{it:exceptional4} There is a Borel measurable mapping
    $\theta:\mathcal{U}_{max} (E)\to\hbox{Iso}(E)$ such that $u=\theta(u)(e)$ for each $u\in\mathcal{U}_{max} (E)$.
\end{enumerate}
\end{lemma}

\begin{proof}
Since $E$ is finite-dimensional, it is a $\sigma$-finite JBW$^*$-triple, so $\varphi$ can be found.\smallskip

$(a)$ This was justified in \eqref{eq invariance of sets of tripotnets under triple automorphisms}.\smallskip

$(b)$ Since the triple product is jointly norm continuous, $\mathcal{U}(E)$ and $\mathcal{U}(E)^*= \mathcal{U}(E)\backslash\{0\}$ are closed subsets of the closed unit ball and the unit sphere of $E$, respectively, so they are compact.
Elements of $\hbox{Iso}(E)$ are precisely (surjective) isometries, so $\hbox{Iso}(E)$ is a closed subset of the unit sphere of $B(E)$, hence it is compact.\smallskip

We next consider the mapping $\Psi:\hbox{Iso}(E)\to E$ defined by
$$\Psi(\Phi)=\Phi(e), \quad \Phi\in \hbox{Iso}(E).$$
It is clearly a continuous mapping and by $(a)$ it maps $\hbox{Iso}(E)$ into $\mathcal{U}_{max} (E)$.
We deduce from \eqref{eq orbit of a complete tripotent in finite dimension} that $\Psi$ is onto, so $\mathcal{U}_{max} (E)$ is compact.\smallskip

$(c)$ For any $\Phi\in\hbox{Iso}(E)$ and $x\in E$ we have (due to $(a)$) $$\norm{x}_{\varphi\circ \Phi^{-1}}^2=(\varphi\circ \Phi^{-1}) \J xx{\Phi(e)} =\varphi \J{\Phi^{-1}(x)}{\Phi^{-1}(x)}{e} =  \norm{\Phi^{-1}(x)}_\varphi^2.$$
Since
$$(x,\Phi)\mapsto \norm{\Phi^{-1}(x)}_\varphi,\qquad x\in S_E, \Phi\in\hbox{Iso}(E)$$
is a strictly positive continuous mapping on the compact space $S_E\times \hbox{Iso}(E)$,
it has some strictly positive minimum and maximum. Thus, the existence of the constant $\alpha$ easily follows. Clearly, $\|x\|_{\phi}\leq \|x\|$ for all $x\in M$ and every norm-one functional $\phi$ in $M_*$.\smallskip

$(d)$ The mapping $\Psi$ from the proof of $(b)$ is a continuous mapping of a compact metric space $\hbox{Iso}(E)$ onto a compact metric space $\mathcal{U}_{max} (E)$,
hence the inverse set-valued map $u\mapsto \Psi^{-1}(u)$ admits a Borel-measurable selection by the Kuratowski---Ryll-Nardzewski theorem (see \cite[Theorem 18.13]{aliprantisborder}).
\end{proof}

The reader may already guess at this stage that the constant $\alpha>0$ given by Lemma~\ref{L:exceptional}$(\ref{it:exceptional3})$ is directly linked to the dimension of the JB$^*$-triple $E$.
If we have a family $\{C_k:k\in\Lambda\}$ of finite-dimensional Cartan factors for which the dim$(C_k)$ is uniformly bounded for all $k\in \Lambda$ (for example, a family of exceptional Cartan factors of types 5 and 6), then the constant $\alpha$ can be chosen to be valid for all $k\in \Lambda$.

\begin{prop}\label{P:LinfyC}
Let $E$ be a finite-dimensional JB$^*$-triple, and let $(\Omega,\Sigma,\mu)$ be a probability space.
Consider the JBW$^*$-triple $M=L^\infty(\mu,E)$ {\rm(}equipped with the  pointwise triple product{\rm)}.
Let $e,\varphi,\theta,\alpha$ be as in Lemma~\ref{L:exceptional}. Then the following assertions hold:
\begin{enumerate}[$(a)$]
\item An element $f\in M$ is a tripotent if and only if $f(\omega)\in \mathcal{U}(E)$ $\mu$-almost everywhere;
\item An element $f\in M$ is a complete tripotent if and only if $f(\omega)\in\mathcal{U}_{max} (E)$ $\mu$-almost everywhere;
\item Assume that $f\in M$ is a complete tripotent. Let $$v(\omega)=\varphi\circ \theta(f(\omega))^{-1},\quad \omega\in\Omega.$$
    Then $v\in L^1(\mu,E_*)=L^\infty(\mu,E)_*$ and  $s(v)=f$;
\item Let $f$ and $v$ be as in $(c)$. Then
$$\alpha\left(\int \norm{g(\omega)}^2\di\mu(\omega)\right)^{\frac12}\le \norm{g}_v\le \left(\int \norm{g(\omega)}^2\di\mu(\omega)\right)^{\frac12},\quad g\in M;$$
\item The strong$^*$ topology on $B_M$ coincides with the topology generated by the norm $\norm{\cdot}_v$ and also with the topology generated by the norm $g\mapsto\left(\int \norm{g(\omega)}^2\di\mu(\omega)\right)^{\frac12}$.
\end{enumerate}
\end{prop}

\begin{proof}
Assertion $(a)$ follows immediately from the fact that the triple product is defined pointwise.\smallskip

$(b)$ Since the triple product is defined pointwise, we have, for a given tripotent $f\in M$,
$$P_0(f)(g)(\omega)=P_0(f(\omega))(g(\omega))\quad \mu\mbox{-a.e.}$$
Hence, if $f(\omega)\in \mathcal{U}_{max}(E)$  $\mu$-almost everywhere, then clearly $P_0(f)=0$.\smallskip

Conversely, assume that it is not true that $f(\omega)\in \mathcal{U}_{max}(E)$ $\mu$-almost everywhere. Since $\mathcal{U}_{max} (E)$ is a closed set, there is a measurable set $A\subset\Omega$ of positive measure such $f(\omega)\notin \mathcal{U}_{max}(E)$, for all $\omega\in A$.\smallskip

For any $u\in\mathcal{U}(E)$ there is $u'\in\mathcal{U}_{max} (E)$ with $u\le u'$ (cf. \cite[Lemma 3.12]{Horn1987}).
Moreover,
the set
$$\{(u,u')\in \mathcal{U}(E)\times\mathcal{U}_{max} (E)\setsep u\le u'\}$$
is closed, hence compact, and thus the set-valued mapping
$$\mathcal{U}(E)\ni u\mapsto\{u'\in \mathcal{U}_{max} (E)\setsep u\le u'\}$$
is upper-semicontinuous and compact-valued. By the Kuratowski--Ryll-Nardzewski theorem we find a Borel-measurable
mapping $\zeta:\mathcal{U}(E)\to\mathcal{U}_{max} (E)$ such that $u\le \zeta(u)$ for $u\in \mathcal{U}(E)$.\smallskip

Then the mapping $g=\zeta\circ f$ belongs to $M$ and
$$P_0(f)(g)(\omega)=P_0(f(\omega))(g(\omega))=P_0(f(\omega))(\zeta(f(\omega))=\zeta(f(\omega))-f(\omega)$$
which is nonzero on $A$. Thus $f$ is not complete.\smallskip

$(c)$ By $(b)$ we know that $f(\omega)\in \mathcal{U}_{max} (E)$ $\mu$-almost everywhere,
so the mapping $\omega\mapsto\theta(f(\omega))$ is a $\mu$-almost everywhere defined measurable mapping from $\Omega$ into $\hbox{Iso}(E)$.
Since taking an inverse is a continuous transformation, we see that $v$ is a $\mu$-almost everywhere defined measurable mapping from $\Omega$ into $E_*$.
Moreover, since $\norm{\varphi}=1$ and elements of $\hbox{Iso}(E)$ are isometries, $\norm{v(\omega)}=1$ $\mu$-almost everywhere. Thus $v\in L^1(\mu,E_*)$ and $\norm{v}=1$ (as $\mu$ is a probability measure). Moreover,
$$\ip vf=\int \ip{v(\omega)}{f(\omega)}\di\mu =
\int \varphi \circ \theta(f(\omega))^{-1} (f(\omega))\di\mu
=\int \varphi(e)\di\mu=1.
$$
Furthermore, assume that $h\in M_2(f)$ is positive with $\ip vh=0$. Then $h(\omega)$ is a positive element of $E_2(f(\omega))$ for $\mu$-almost all $\omega\in \Omega$, hence $\theta(f(\omega))^{-1}(h(\omega))$ is a positive element of $E_2(e)$ for $\mu$-almost all $\omega$. Hence
$$0=\ip v h=\int \ip{v(\omega)}{h(\omega)}\di\mu =
\int \varphi(\theta(f(\omega))^{-1}(h(\omega)))\di\mu,
$$
so $\varphi(\theta(f(\omega))^{-1}(h(\omega)))=0$ $\mu$-a.e. Since $\varphi$ is faithful on $E_2(e)$, we deduce that $\theta(f(\omega))^{-1}(h(\omega))=0$ $\mu$-a.e., so $h(\omega)=0$ $\mu$-a.e.\smallskip

$(d)$ For any $g\in M$ we have
$$\begin{aligned}
\norm{g}^2_v&=\ip{v}{\J ggf} = \int \ip{v(\omega)}{\J {g(\omega)}{g(\omega)}{f(\omega)}}\di\mu\\&=\int \ip{\varphi\circ \theta(f(\omega))^{-1}}{\J {g(\omega)}{g(\omega)}{f(\omega)}}\di\mu
=\int \norm{g(\omega)}^2_{\varphi\circ \theta(f(\omega))^{-1}}\di\mu,
\end{aligned}$$
so we can conclude by the choice of $\alpha$.\smallskip

$(e)$ For any tripotent $h\in M$ there is a complete tripotent $f\in M$
with $f\ge h$. For any complete tripotent $f$ let $v(f)\in M_*$ be as in $(c)$. By Proposition~\ref{P:seminorms order}$(i)$ the strong$^*$ topology on $B_M$ coincides with the topology generated by the seminorms $\norm{\cdot}_{v(f)}$, $f\in M$ a complete tripotent. We deduce from $(d)$ that all these norms are equivalent to the norm $g\mapsto\left(\int \norm{g(\omega)}^2\di\mu(\omega)\right)^{\frac12}$.
\end{proof}

\subsection{Triples of the form $pV$}

It turns out that the analysis of this case is more complicated than the previous two cases. We shall employ an argument which is closely related to the notion of equivalence of projections and to the theory of types of von Neumann algebras (see, for example, \cite{KR2}).\smallskip

Given a von Neumann algebra $V$, two projections $p,q\in V$ are said to be \emph{equivalent} (we write $p\sim q$) if there is a partial isometry in $V$ with initital projection $p$ and final projection $q$. Further, a projection $p$ is called \emph{finite} if the only projection $q$ satisfying $q\le p$ and $q\sim p$ is the projection $p$ itself. A projection which is not finite is called \emph{infinite}. Finally, a projection $p$ is \emph{properly infinite} if $zp$ is infinite for any central projection $z$ such that $zp\ne0$.\smallskip

For any projection $p\in V$ its \emph{central carrier} is the smallest central projection $C_p$
satisfying $C_pp=p$. It is further known that
there is a unique central projection $z\le C_p$ such that $zp$ is properly infinite or zero and $(1-z)p=(C_p-z)p$ is finite.
Indeed, if $p$ is finite, we take $z=0$, and if $p$ is infinite we may use  \cite[Proposition 6.3.7]{KR2}.
\smallskip

Henceforth, assume that we have a JBW$^*$-triple of the form $pV$, where $V$ is a von Neumann algebra, and $p\in V$ is a projection. 
We may assume, without loss of generality, that $C_p=1$ (otherwise we may replace $V$ by $C_pV$). By the previous paragraph there is a central projection $z\in V$ such that $zp$ is properly infinite and $(1-z)p$ is finite. Then $pV =zpV \oplus (1-z)pV$, thus we discuss separately the cases in which $p$ is finite or properly infinite.\smallskip

We begin with the following lemma on equivalence of projections.

\begin{lemma}\label{L:projequi}
Let $V$ be a von Neumann algebra. Then the following assertions are true.
\begin{enumerate}[$(a)$]
    \item Let $(p_n)$ be a sequence of properly infinite projections in $V$ which are all equivalent to one projection $q\in V$. Then the supremum of the sequence $(p_n)$ is also equivalent to $q$.
    \item Let $(p_n)$ be an increasing sequence of projections in $V$ with supremum $p$. If all the projections $p_n$ are equivalent to one projection $q$, then $p\sim q$ as well.
    \item Assume that $p_1,p_2$ are two equivalent projections in $V$. Then for any projection $q_1\ge p_1$ there is a projection $q_2\ge p_2$ such that $q_1\sim q_2$.
\end{enumerate}
\end{lemma}

\begin{proof}
Assertion $(a)$ is proved in \cite[Lemma 3.2(1)]{sherman}.\smallskip

$(b)$ By \cite[Proposition 6.2.8]{KR2} we have $C_{p_n}=C_q$ for each $n\in\en$, hence $C_p=C_q$ by \cite[Proposition 5.5.3]{KR1}. So, denote by $c$ the common central carrier of all the projections in question.\smallskip

Let $z\le c$ be the central projection such that $zq$ is finite and $(c-z)q$ is properly infinite. Then
$zp_n\sim zq$ for each $n\in\en$, so $zp_n\sim zp_m$ for $m,n\in\en$. Since $zp_n$ is finite for each $n$, we deduce that $zp_n=zp_m$ for each $m,n\in\en$, thus $zp=zp_n$ for $n\in\en$, hence $zp\sim zq$.\smallskip

Further, $(c-z)p_n\sim (c-z)q$ for $n\in\en$. Since $(c-z)q$ is properly infinite, the projections $(c-z)p_n$ are properly infinite as well. Thus by $(a)$ we deduce that $(c-z)p\sim(c-z)q$, hence by \cite[Proposition 6.2.2]{KR2} $p\sim q$.\smallskip

$(c)$  By the comparability theorem \cite[Theorem V.1.8]{Tak} for the pair of projections $q_1-p_1$ and $1-p_2$, there is a central projection $z$ such that
\begin{itemize}
    \item $z(q_1-p_1)$ is equivalent to some projection  $r\le z(1-p_2)$, and
    \item $(1-z)(1-p_2)$ is equivalent to some projection $s\le (1-z)(q_1-p_1)$.
\end{itemize}

By \cite[Proposition 6.2.2]{KR2} we get that $zq_1 = zp_1 + z(q_1-p_1)$ is equivalent to $r+zp_2$ and, moreover,
$1-z= (1-z)p_2 + (1-z)(1-p_2)$ is equivalent to $(1-z)p_1 +s \le(1-z)q_1$. But this means that $(1-z)q_1$ is equivalent to $1-z$ (by \cite[Proposition 6.2.4]{KR2}). Finally, one can take $q_2=r+zp_2+1-z$.
\end{proof}

We consider first the case in which $p$ is finite.

\begin{lemma}\label{L:pV finite sigmafinite}
Let $V$ be a von Neumann algebra and $p\in V$ be a finite and $\sigma$-finite projection such that $p\ne 1$. Consider the JBW$^*$-triple $M=pV$.
\begin{enumerate}[$(a)$]
    \item There is $\tau\in M_*$ such that $s(\tau)=p$, $\tau(p)=1$ and $\tau|_{pVp}$ is a trace.
    \item Let $u\in M$ be a complete tripotent. Then $u$ can be extended to a unitary operator $\tilde{u}\in V$. Moreover, the functional
    $$\tau_u(x)=\tau(x\tilde{u}^*),\quad x\in M,$$
    belongs to $M_*$, $s(\tau_u)=u$ and
    $$\frac{1}{\sqrt{2}}\sqrt{\tau(pxx^*p)}\le\norm{x}_{\tau_u}\le \sqrt{\tau(pxx^*p)},\quad x\in M.$$
    \item The strong$^*$ topology on $B_M$ is generated by the norm $\norm{\cdot}_\tau$ and also by the norm
    $$x\mapsto\sqrt{\tau(pxx^*p)}.$$
\end{enumerate}
\end{lemma}

\begin{proof}
$(a)$ Since $pVp$ is a finite and $\sigma$-finite von Neumann algebra with unit $p$,  it admits a normal finite faithful trace $\tau$ with $\tau(p)=1$.
Indeed, such a trace can be obtained by composition of the standard canonical center valued trace on $pVp$ (see \cite[Theorem V.2.6]{Tak})
with  any norm-one faithful positive normal functional on the center of $pVp$. Then $\tau\circ P_2(p)$ (i.e., the mapping $x\mapsto \tau(xp)$) is an extension of $\tau$ to $pV$. Clearly $p=s(\tau\circ P_2(p))$, hence it is enough to denote the composition again by $\tau$.\smallskip

$(b)$ Let $u\in M$ be a complete tripotent. Then it is a partial isometry in $V$ with final projection $p_f(u)\le p$.
Since $p_f(u)$ is finite $u$ can be extended to a unitary operator $\tilde{u}\in V$ (by \cite[Proposition V.1.38]{Tak}). Moreover, observe that the final projection $p_f(u)$ must coincide with $p$.
Indeed, $p\tilde u\in M$ and, since $u$ is complete,
$$0=P_0(u)(p\tilde u)=(p-p_f(u))p\tilde u(1-p_i(u))=
(p-p_f(u))\tilde u(1-p_i(u)).$$
Since $p$ is finite, we get $p_i(u)\ne 1$. Moreover, $\tilde u$ maps the range of $1-p_i(u)$ onto the range of $1-p_f(u)$, which contains the range of $p-p_f(u)$. It follows that $p_f(u)=p$.\smallskip

Set $q=p_i(u) =u^*u$ and consider the operator $\upsilon:M\to M$ defined by $$\upsilon(x)=x\tilde{u}^*, \quad x\in M.$$ Then $\upsilon$ is a surjective isometry. Hence it is a triple isomorphism (this can be also easily checked directly), in particular, it is a weak$^*$-to-weak$^*$ homeomorphism. Since $\upsilon(u)=p$, we deduce that
$s(\tau\circ\upsilon)=u$. Thus,
$$\begin{aligned}
\norm{x}^2_{\tau\circ\upsilon}&=(\tau\circ\upsilon)(\J xxu)
=\tau(\J{\upsilon(x)}{\upsilon(x)}{\upsilon(u)})=\tau(\J{x\tilde{u}^*}{x\tilde{u}^*}p)\\&=\frac12\tau(px\tilde{u}^*\tilde{u}x^*p+
p\tilde{u}x^*x\tilde{u}^*p)=\frac12\tau(pxx^*p+
ux^*xu^*)\\
&=\frac12(\tau(pxx^*p)+\tau((ux^*p)(pxu^*)))=\frac12(\tau(pxx^*p)+\tau((pxu^*)(ux^*p))\\&=\frac12(\tau(pxx^*p)+\tau(pxqx^*p)).
\end{aligned}$$
Since $q\le 1$,  we deduce $pxqx^*p\le pxx^*p$, and thus
$$\frac12 \tau(pxx^*p)\le \norm{x}_{\tau\circ\upsilon}^2\le \tau(pxx^*p).$$
Since $\tau_u=\tau\circ\upsilon$, the proof is completed.\smallskip

$(c)$ It follows from $(b)$ combined with
Proposition~\ref{P:seminorms order}$(i)$ that the strong$^*$ topology on $B_M$ coincides with the topology generated by the norms $\norm{\cdot}_{\tau_u}$, where $u\in M$ is a complete tripotent. By a further application of $(b)$ we see that all these norms are equivalent to the one given in $(c)$.
\end{proof}

We finally consider the case in which $p$ is properly infinite.

\begin{prop}\label{p pV}
Let $V$ be a von Neumann algebra and let $p\in V$ be a $\sigma$-finite properly infinite projection. Consider the JBW$^*$-triple $M=pV$. Assume that $M$ contains no nonzero direct summand triple-isomorphic to a JBW$^*$-algebra. Then the following assertions hold:
\begin{enumerate}[$(a)$]
    \item $V$ is not $\sigma$-finite;
    \item A tripotent $u\in M$ is complete if and only if its final projection equals $p$;
    \item Let $(u_n)$ be a sequence of complete tripotents in $M$. Then there is a complete tripotent $u\in M$ such that $M_2(u_n)\subset M_2(u)$ for each $n\in\en$;
    \item If $u\in M$ is a complete tripotent, then there is a complete tripotent $v\in M$ such that $M_2(u)\subsetneqq M_2(v)$;
    \item The strong$^*$ topology on $B_M$ is not metrizable.
\end{enumerate}
\end{prop}

\begin{proof}
$(a)$ If $V$ is $\sigma$-finite, then $p\sim 1$ (as $C_p=1$ and $p$ is purely infinite), thus $M=pV$ would be triple-isomorphic to a JBW$^*$-algebra (given a partial isometry $u$ with $uu^* =p$ and $u^* u =1$, the mapping $x\mapsto x u^*$ is a surjective isometry from $M$ onto $pVp$).\smallskip

$(b)$ The `if part' is clear. Let us prove the `only if part'. Assume $p_f(u)<p$. By $(a)$ we get $p_i(u)<1$. Thus $p-p_f(u)$ and $1-p_i(u)$ are two nonzero projections in $V$, thus it follows easily from the comparability theorem \cite[Theorem V.1.8]{Tak} that there are two nonzero projections $q_1\le 1-p_i(u)$ and $q_2\le p-p_f(u)$ which are equivalent. Fix a partial isometry $v\in V$ with initial projection $q_1$ and final projection $q_2$. Then $v\in M$ and
$$P_0(u)(v)=(p-p_f(u))v(1-p_i(u))\ne0$$
as the range of $1-p_i(u)$ contains the range of $q_1$, $v$ maps it isometrically to the range of $q_2$ which  is contained in the range of $p-p_f(u)$.\smallskip

$(c)$ By $(b)$ we know that $p_f(u_n)=p$ for each $n$.
So, $p\sim p_i(u_n)$ for each $n\in\en$.
If we set $q=\sup_n p_i(u_n)$, Lemma~\ref{L:projequi}(a) yields $p\sim q$.
 Then $u$ can be any partial isometry with initial projection $q$ and final projection $p$.\smallskip

$(d)$ By $(b)$ we know that $p_f(u)=p$. Since $p_i(u)<1$ (by $(a)$), we can find a $\sigma$-finite projection $q>p_i(u)$.
Then $q$ is properly infinite, and hence $p\sim q$.
Then $v$ can be any partial isometry with initial projection $q$ and final projection $p$.\smallskip

$(e)$ Assume that the restriction of the strong$^*$ topology to $B_M$ is metrizable. Then it is first countable, hence
generated by countably many of the defining seminorms. Then $(c)$ and $(d)$ together with Proposition~\ref{P:seminorms order}$(i)$ yield a contradiction.
\end{proof}

\subsection{The case of a general $\sigma$-finite JBW$^*$-triple} We are now ready to prove assertion $(c)$ of Theorem~\ref{T:SWCG}. We will do it by proving the following two propositions (the final proof follows them). \smallskip

\begin{prop}\label{P:sigmafiniteYES}
Assume that $M$ is a nontrivial JBW$^*$-triple of the form $$\left( \bigoplus_{k\in\Lambda}^{\ell_{\infty}} L^\infty(\mu_k,C_k)\right)\bigoplus^{\ell_{\infty}} N\bigoplus^{\ell_{\infty}} pV,$$
where
\begin{enumerate}[$\bullet$]
\item $\Lambda$ is a {\rm(}possibly empty{\rm)} countable set;
\item $(\mu_k)_{k\in\Lambda}$ is a {\rm(}possibly empty{\rm)} family of probability measures;
\item Each $C_k$ is a Cartan factor of type 5 or 6 or a finite-dimensional Cartan factor of type 2 in $B(H_k)$ with dim$(H_k)$ odd;
\item $N$ is a  {\rm(}possibly trivial{\rm)} $\sigma$-finite JBW$^*$-algebra;
\item $V$ is a {\rm(}possibly trivial{\rm)} von Neumann algebra and $p\in V$ is a finite $\sigma$-finite projection such that the triple $pV$ has no nonzero direct summand triple-isomorphic to a JBW$^*$-algebra.
\end{enumerate}

Fix a faithful normal state $\phi_3\in N_*$. Let $\tau\in (pV)_*$ be as in Lemma~\ref{L:pV finite sigmafinite}$(a)$.
Then the following statements hold:
\begin{enumerate}[$(a)$]
\item We can regard $\phi_3$ as an element in $M_*$ satisfying that the strong$^*$ topology on $B_N$ is metrizable by the norm $\|\cdot\|_{\phi_3}|_{N}$;
\item We can regard $\tau$ as an element in $M_*$ satisfying that the strong$^*$ topology on $B_{(pV)}$ is metrizable by the norm $\|\cdot\|_{\tau}|_{pV}$;
\item Let $\displaystyle C= \bigoplus_{k\in\Lambda}^{\ell_{\infty}} L^\infty(\mu_k,C_k)$. Fix any $\varphi\in C_*\setminus\{0\}$ such that $s(\varphi)\in \mathcal{U}_{max}(C)$. Then the norm $\norm{\cdot}_\varphi$ is equivalent to the norm
    $$(a_k)_{k\in\Lambda}\mapsto \left(\sum_{n=1}^{\infty} 4^{-n}\ \int \norm{a_{k_n}}^2 \di\mu_{k_n}\right)^{\frac12}$$ on bounded sets of $C$ {\rm(}where $(k_n)$ is an enumeration of $\Lambda${\rm)}. The strong$^*$ topology on $B_{C}$ is metrized by the norm displayed above;

\item The strong$^*$ topology on $B_M$ is metrized by the norm $\norm{\cdot}_{\tau+\phi_3+\varphi}$ {\rm(}where the functional $\varphi$ from $(c)$ is considered as an element of $M_*${\rm)} which is equivalent to the norm $$ \norm{((a_{k})_{k\in\Lambda},x,y)}^2= \left(\sum_{n=1}^{\infty} \frac1{4^n} \int \norm{a_{k_n}}^2 \di\mu\right)+\|x\|_{\phi_3}^2+\tau(py^*yp).$$
\end{enumerate}
\end{prop}

\begin{proof}$(a)$ and $(b)$ are proved in Lemmata \ref{L:JBW*algebra sigma finite}$(b)$ and \ref{L:pV finite sigmafinite}, respectively.\smallskip

$(c)$ Fix any $\varphi\in C_*\setminus\{0\}$ such that $s(\varphi)=(f_k)_{k\in\Lambda}\in\U_{max}(C)$.

Fix $k\in\Lambda$. Then $f_k$ is a maximal tripotent in $L^\infty(\mu_k,C_k)$, hence we can fix $h_k\in L^1(\mu_k,(C_k)_*)$ provided by Proposition~\ref{P:LinfyC}$(c)$. Let $(k_n)$ be an enumeration of $\Lambda$ and set $$\phi_1 ((a_k)_{k\in\Lambda})=\sum_{n=1}^{\infty} 2^{-n} \ip{h_{k_n}}{a_{k_n}}, \ \  ((a_k)_{k\in\Lambda}\in C).$$ Clearly, $s(\phi_1) =(f_k)_{k\in\Lambda}= s(\varphi)$, so $\norm{\cdot}_{\phi_1}$ and $\norm{\cdot}_\varphi$ are equivalent on $B_C$ by Proposition~\ref{P:seminorms order}. By Proposition~\ref{P:LinfyC} the norm $\norm{\cdot}_{h_k}$ is equivalent to the norm
$$f\mapsto \left(\int\norm{f}^2\di\mu_k\right)^{1/2}$$
on the unit ball of $L^\infty(\mu_k,C_k)$ for each $k\in\Lambda$.
Hence, the norm
$$\norm{(a_k)_{k\in\Lambda}}_{\phi_1}=\left(\sum_{n=1}^\infty 4^{-n}\norm{a_{k_n}}_{h_{k_n}}^2\right)^{1/2}$$
is equivalent on $B_C$ to the norm
$$(a_k)_{k\in\Lambda}\mapsto \left(\sum_{n=1}^{\infty} 4^{-n} \int \norm{a_{k_n}}^2\right)^{1/2}.$$
Indeed, both norms are well defined. Moreover, a bounded sequence $((a_k^j)_{k\in\Lambda})_{j=1}^\infty$ converges to zero in the first norm if and only if
$$\norm{a_k^j}_{h_k}\overset{j}{\to}0 \mbox{ for each }k\in\Lambda,$$ which takes place if and only if $$\int\Norm{a_k^j}^2\di\mu_k\overset{j}{\to}0\mbox{ for each }k\in\en,$$ which is in turn equivalent to the convergence to zero in the second norm.\smallskip

Finally, it follows from Proposition~\ref{P:seminorms order} that the strong$^*$ topology on $B_C$ is generated by the mentioned norm.\smallskip

Lastly, statement $(d)$ follows from the previous statements.
\end{proof}

The remaining case is treated in our next result.

\begin{prop}\label{P:sigmafiniteNO} Assume that $M$ is a JBW$^*$-triple of the form $$\left( \bigoplus_{k\in\Lambda}^{\ell_{\infty}} L^\infty(\mu_k,C_k)\right)\bigoplus^{\ell_{\infty}}  N\bigoplus^{\ell_{\infty}} pV\bigoplus^{\ell_\infty}qW,$$
where
\begin{enumerate}[$\bullet$]
\item $\Lambda$ is a {\rm(}possibly empty{\rm)} countable set;
\item $(\mu_k)_{k\in\Lambda}$ is a {\rm(}possibly empty{\rm)} family of probability measures;
\item Each $C_k$ is a Cartan factor of type 5 or 6 or a finite-dimensional Cartan factor of type 2 in $B(H_j)$ with dim$(H_j)$ odd;
\item $N$ is a {\rm(}possibly trivial{\rm)} $\sigma$-finite JBW$^*$-algebra;
\item $V$ is a {\rm(}possibly trivial{\rm)} von Neumann algebra and $p\in V$ is a finite $\sigma$-finite projection such that the triple $pV$ has no nonzero direct summand triple-isomorphic to a JBW$^*$-algebra;
\item $W$ is a nontrivial von Neumann algebra and $q\in W$ is a properly infinite $\sigma$-finite projection such that the triple $qW$ has no nonzero direct summand triple-isomorphic to a JBW$^*$-algebra.
\end{enumerate}
Then the strong$^*$ topology on $B_M$ is not metrizable.
\end{prop}

\begin{proof} Proposition \ref{p pV}$(e)$ assures that the strong$^*$ topology on $B_{qW}$ is not metrizable, and the desired conclusion follows from Lemma \ref{L:strong* topology}$(b)$.
\end{proof}

\begin{proof}[Proof of Theorem~\ref{T:SWCG}$(c)$]
$(ii)\Rightarrow(iii)$ This follows from Proposition~\ref{P:sigmafiniteYES}.\smallskip 

$(iii)\Rightarrow(i)$ This follows from Proposition~\ref{P:strong*=Mackey} and \cite[Theorem 2.1]{SWCG} as explained above.\smallskip

$(i)\Rightarrow(ii)$ Assume $p$ is not finite. By Proposition~\ref{P:sigmafiniteNO} the strong$^*$ topology on $B_M$ is not metrizable. Hence $M_*$ is not strongly WCG by Proposition~\ref{P:strong*=Mackey} and \cite[Theorem 2.1]{SWCG}.\smallskip

$(iii)\Rightarrow(iv)$ This follows from Proposition~\ref{P:seminorms order}(ii).\smallskip

$(iv)\Rightarrow(ii)$ Assume (iv) holds but $p$ is not finite. It follows from (iv) that $M$ is $\sigma$-finite, hence $M$ has the form from Proposition~\ref{P:sigmafiniteNO}. Let $u=((a_k),(b_j),x,v,w)$ be any tripotent in $M$. Then $w$ is a tripotent in $qW$. It follows from Proposition~\ref{p pV} that there is a tripotent $\tilde w\in qW$ with $(qW)_2(w)\subsetneqq (qW)_2(\tilde w)$. Set $\tilde u=((a_k),(b_j),x,v,\tilde w)$. Then $\tilde u$ is a tripotent in $M$ and $M_2(u)\subsetneqq M_2(\tilde u)$.
\end{proof}

\begin{remark}
It follows from the analysis of the individual cases in this section that any $\sigma$-finite JBW$^*$-triple can be expressed as a direct sum of (countably many) summands of three different types.
\begin{description}
\item[Type 1 -- JBW$^*$-algebra] If $N$ is a $\sigma$-finite JBW$^*$-algebra, it admits a unit, i.e., a ($\sigma$-finite) unitary element. Then $N_2(1_N)=N$, hence the Peirce-2 subspace is the largest possible.
\item[Type 2 -- $L^\infty(\mu,C)$ or $pV$ with $p$ finite] Assume $M=L^\infty(\mu,C)$, where $\mu$ is a probability measure and $C$ is a finite-dimensional Cartan factor without
a unitary element, or $M=pV$, where $V$ is a von Neumann algebra and $p\in V$ is a finite $\sigma$-finite projection such that $M$ has no direct summand isomorphic to a JBW$^*$-algebra.
Then there are tripotents whose Peirce-2 subspaces are maximal with respect to inclusion, but mutually different. But all the norms $\norm{\cdot}_\varphi$, where $s(\varphi)$ is such a tripotent, are equivalent on bounded sets.
\item[Type 3 -- $pV$ for $p$ properly infinite] Assume that
$M=pV$, where $V$ is a von Neumann algebra and $p\in V$ is a properly infinite $\sigma$-finite projection such that $M$ has no direct summand isomorphic to a JBW$^*$-algebra. Then the family of Peirce-2 subspaces $M_2(u)$, $u\in \U(M)$, is upwards $\sigma$-directed by inclusion and has no maximal element.
\end{description}
In the next section we give some consequences of this trichotomy to the structure of general (not necessarily $\sigma$-finite) JBW$^*$-triples.
\end{remark}

\section{On seminorms generating the strong$^*$ topology}\label{sec:10}

In this section we provide a characterization of the natural ordering of the seminorms generating the strong$^*$ topology for a JBW$^*$-triple,
which is defined by inclusion of the respective topologies on the unit ball. The case of $\sigma$-finite triples is covered by Propositions~\ref{P:sigmafiniteYES} and~\ref{P:sigmafiniteNO}, here we deal with general triples. The promised result is contained in the following theorem.

\begin{thm}\label{T:order}
Let $M$ be a JBW$^*$-triple. Then it can be represented in the form
$$M=\bigoplus_{\alpha\in\Lambda}L^\infty(\mu_\alpha,C_\alpha)\oplus N\oplus pV\oplus qW,$$
where
\begin{enumerate}[$\bullet$]
\item $\Lambda$ is an arbitrary {\rm(}possibly empty{\rm)} set;
\item $\mu_\alpha$ is a probability measure and $C_\alpha$ is a finite-dimensional Cartan factor not containing a unitary element for any $\alpha\in\Lambda$;
\item $N$ is a {\rm(}possibly trivial{\rm)} JBW$^*$-algebra;
\item $V$ is a {\rm(}possibly trivial{\rm)} von Neumann algebra and $p\in V$ is a finite projection such that the triple $pV$ has no nonzero direct summand triple-isomorphic to a JBW$^*$-algebra;
\item $p=\sum_{j\in J}p_j$, where $(p_j)_{j\in J}$ is an orthogonal family of {\rm(}finite{\rm)} $\sigma$-finite projections in the center of $pVp$;
\item $W$ is a {\rm(}possibly trivial{\rm)} von Neumann algebra and $q\in W$ is a properly infinite  projection such that the triple $qW$ has no nonzero direct summand triple-isomorphic to a JBW$^*$-algebra.
\end{enumerate}
For an element
$$\h= ((f_\alpha(\h))_{\alpha\in\Lambda},x(\h),v(\h),w(\h))= ((f_\alpha)_{\alpha\in\Lambda},x,v,w)\in M$$
we denote
$$\spt_1\h=\{\alpha\in\Lambda\setsep f_\alpha\ne0\},\quad
\spt_2\h=\{j\in J\setsep p_j v\ne0\}.$$
Further, set
$$\begin{aligned}
\T(M)=\{ \h=((f_\alpha&)_{\alpha\in\Lambda},x,v,w)\in M\setsep \spt_1\h,\spt_2\h\mbox{ are countable},\\
&\forall \alpha\in\spt_1\h: f_\alpha\mbox{ is a complete tripotent in }L^\infty(\mu_\alpha,C_\alpha),
\\& x\mbox{ is a $\sigma$-finite projection in }N,\
v v^*=\sum_{j\in\spt_2\h}p_j,\\&
ww^*\mbox{ is a properly infinite $\sigma$-finite projection below }q \}.
\end{aligned}$$
Then the following assertions hold:
\begin{enumerate}[$(a)$]
      \item The elements of $\T(M)$ are $\sigma$-finite tripotents in $M$. Moreover, for any $\sigma$-finite tripotent $\g\in M$ there is $\h\in\T(M)$ with $M_2(\g)\subset M_2(\h)$.
    \item  Let $\varphi,\psi\in M_*\setminus\{0\}$ such that the support tripotents of these functionals belong to $\T(M)$. Set $\h=s(\varphi)$ and $\g=s(\psi)$. Then $\norm{\cdot}_\varphi$ is weaker than $\norm{\cdot}_\psi$ on $B_M$ if and only if the following assertions hold:
    \begin{itemize}
        \item[$\circ$] $\spt_1\h\subset\spt_1\g$ and $\spt_2\h\subset\spt_2\g$;
        \item[$\circ$] $x(\h)\le x(\g)$ as projections in $N$;
        \item[$\circ$] $\norm{\cdot}_\varphi$ is weaker than $\norm{\cdot}_\psi$ on $B_{qW}$.
    \end{itemize}
\end{enumerate}
\end{thm}

Before proving this theorem let us formulate some consequences.

\begin{cor}\label{cor:sigma-directed}
Let $M$ be  a JBW$^*$-triple. Given a sequence $(\varphi_n)$ in $M_*\setminus\{0\}$, there is $\psi\in M_*\setminus\{0\}$, such that $\norm{\cdot}_{\varphi_n}$ is weaker than $\norm{\cdot}_\psi$ on $B_M$ for each $n\in\en$, i.e., the family of topologies on $B_M$ generated by the seminorms $\norm{\cdot}_\varphi$, $\varphi\in M_*\setminus\{0\}$ is upwards $\sigma$-directed by inclusion.
\end{cor}

The proof of this corollary will use one of the lemmata below, so we postpone the 	 the end of the section.

\begin{cor}\label{c 10.3}
Let $M$ be a JBW$^*$-triple and $L\subset M_*$ any weakly compact set. Then there is $\varphi\in M_*$ such that for any $\varepsilon>0$ there is $n\in\en$ satisfying $L\subset n K(\varphi)+\varepsilon B_{M_*}$.
\end{cor}

\begin{proof}
We will imitate the proof of Proposition~\ref{P:wcompact cofinal} using the same notation. The seminorm $q_L$ is Mackey continuous, so $q_L|_{B_M}$ is strong$^*$-continuous.
Hence, given $m\in\en$, there are $\varphi_1^m,\dots,\varphi_{k_m}^m\in M_*\setminus\{0\}$ and $\delta_m>0$ such that
$$\{x\in B_M\setsep \norm{x}_{\varphi_j^m}\le\delta_m\mbox{ for }j=1,\dots,k_m\}\subset\{x\in B_M\setsep q_L(x)\le\frac1m\}.$$
By Corollary~\ref{cor:sigma-directed} there is $\varphi\in M_*\setminus\{0\}$ such that $\norm{\cdot}_{\varphi_j^m}$ is weaker than $\norm{\cdot}_\varphi$ on $B_M$ for each $m\in\en$ and $j=1,\dots,k_m$. Since $\norm{\cdot}_\varphi$ is on $B_M$ equivalent to $\norm{\cdot}_\varphi^K=q_{K(\varphi)}$ (by Lemma~\ref{L:normal functional}$(c,d)$) we deduce that
for each $m\in\en$ there is $\eta_m>0$ such that
$$\{x\in B_M\setsep q_{K(\varphi)}(x)\le\eta_m\}\subset \{x\in B_M\setsep q_L(x)\le\frac1m\}.$$
The calculation of polars (see the proof of Proposition~\ref{P:wcompact cofinal}) shows that
$$L\subset \frac1{m\eta_m} K(\varphi)+\frac2m B_{M_*}.$$
This completes the proof.
\end{proof}

Now we proceed with the proof of Theorem~\ref{T:order}. This will be done in several steps.\smallskip

Let us start by explaining the existence of the respective representation. Let $M$ be any JBW$^*$-triple. By Proposition~\ref{P:representation} $M$ can be represented as
$$M=\bigoplus_{\alpha\in\Lambda}L^\infty(\mu_\alpha,C_\alpha)\oplus N\oplus sU,$$
where $\Lambda$ is a set,  $\mu_\alpha$ is a probability measure and $C_\alpha$ is a finite-dimensional Cartan factor without unitary element for each $\alpha\in\Lambda$, $N$ is a JBW$^*$-algebra,
$U$ is a von Neumann algebra, and $s\in U$ is a projection such that $sU$ admits no nonzero direct summand isomorphic to a JBW$^*$-algebra. We may assume without loss of generality that $C_s=1_U$.
By \cite[Proposition 6.3.7]{KR2} there is a unique central projection $z\in U$ such that $zs$ is poperly infinite or zero and $(1-z)s$ is finite,
hence $sU=zsU\oplus (1-z)sU$. Take $W=zU$, $q=zs$, $V=(1-z)U$, $p=(1-z)s$. Then we have the representation of the form from Theorem~\ref{T:order}, where $p$ is finite and $q$ properly infinite.
Finally, the existence of the relevant decomposition of $p$ follows from  \cite[Corollary V.2.9]{Tak}.\smallskip

We continue by proving assertion (a). It is clear that all the elements of $\T(M)$ are $\sigma$-finite tripotents. To prove the second statement we will use two lemmata.

\begin{lemma}\label{L:sigmafiniteOGfamily}
Let $V$ be a von Neumann algebra, $u\in V$ a $\sigma$-finite tripotent and $(r_j)_{j\in J}$ an orthogonal family of projections. Then the sets
$$\{j\in J\setsep r_j u\ne0\}\mbox{ and }\{j\in J\setsep ur_j \ne0\}$$
are countable.
\end{lemma}

\begin{proof}
Note that $u$, being a tripotent, is a partial isometry with initial projection $p_i(u)=u^*u$ and final projection $p_f(u)=uu^*$. Moreover, since $u$ is $\sigma$-finite, both $p_i(u)$ and $p_f(u)$ are $\sigma$-finite.
Further, it is clear that $r_ju\ne0$ if and only if $r_j p_f(u)\ne0$, and that $ur_j\ne0$ if and only if $p_i(u)r_j\ne0$. We can therefore assume, without loss of generality, that $u$ is a projection.\smallskip

So, assume $u$ is a projection. Now consider two orthogonal families of cyclic projections in $V$ with sum equal to $1$, say $(q_\gamma)_{\gamma\in\Gamma}$ and
$(s_\delta)_{\delta\in\Delta}$, such that $u$ is the sum of a subfamily of the first one and $r_j$ is the sum of a subfamily of the second one for each $j\in J$. The existence of these families follows easily from \cite[Proposition 5.5.9]{KR1}.\smallskip

Since $\{\gamma\in \Gamma\setsep q_\gamma u\ne0\}$ is countable, \cite[Proposition 4.1]{BHK-vN} implies that both sets
$\{\delta\in\Delta\setsep us_\delta\ne0\}$
and
$\{\delta\in\Delta\setsep s_\delta u\ne0\}$
are countable. Now the assertion easily follows.
\end{proof}

\begin{lemma}\label{L:infinitecofinality}
Let $V$ be a von Neumann algebra
and let $p\in V$ be a properly infinite projection. Then for any $\sigma$-finite projection $q\le p$ there is a properly infinite $\sigma$-finite projection $r$ such that $q\le r\le p$.
\end{lemma}

\begin{proof}
Without loss of generality we can assume that $p=1$. According to assumption $V$ is properly infinite. Therefore there is a sequence $(q_n)$ of mutually orthogonal projections such that $\sum_n q_n =1$ and $q_n\sim 1$ for each $n$ (see e.g. Proposition 4.12, page 97 in \cite{Stratila}). Therefore, there are projections $r_n'$ with $r_n'\le q_n$ and  $q\sim r'_n$ for each $n$ (cf. Lemma~\ref{L:projequi}$(c)$). Then $r'=\sum_n r_n'$ is a properly infinite $\sigma$-finite projection (cf. \cite[Proposition 4.12]{Stratila}) to which $q$ is subequivalent (see Lemma~\ref{L:projequi}). Now by Lemma~\ref{L:projequi}$(c)$ there is a projection $r\ge q$ with $r\sim r'$. This projection is $\sigma$-finite and properly infinite.
\end{proof}

Now we are ready to prove the second statement of assertion $(a)$. Let
$$\h=((f_\alpha)_{\alpha\in\Lambda},x,v,w)\in M$$
be a $\sigma$-finite tripotent. It is clear that $\spt_1\h$ is countable. For each $\alpha\in\spt_1\h$ choose a complete tripotent $g_\alpha\in L^\infty(\mu_\alpha,C_\alpha)$ such that $g_\alpha\ge f_\alpha$, and
for each $\alpha\in \Lambda\setminus \spt_1\h$ set $g_\alpha=0$.\smallskip

Further, $x$ is a $\sigma$-finite tripotent in $N$, hence by Lemma~\ref{L:sigma finite tripotent}$(d)$ there is a $\sigma$-finite projection $y\in N$ with $x\in N_2(y)$.\smallskip

Since $v$ is a $\sigma$-finite tripotent in $pV$,  by Lemma~\ref{L:sigmafiniteOGfamily} the set $\spt_2\h$ is countable. The final projection of $v$ satisfies $p_f(v)\le\sum_{j\in \spt_2\h}p_j$, so by Lemma~\ref{L:projequi}$(c)$ there is a projection $r\ge p_i(v)$ such that $r\sim \sum_{j\in \spt_2\h}p_j$, so we can choose a partial isometry $\tilde{v}\in V$ with $p_i(\tilde{v})=r$ and $p_f(\tilde{v})=\sum_{j\in \spt_2\h}p_j$.\smallskip

Finally, $w$ is a $\sigma$-finite tripotent in $qW$, thus $p_f(w)$ is a $\sigma$-finite projection below $q$. It follows from Lemma~\ref{L:infinitecofinality} that there is a $\sigma$-finite properly infinite projection $s_1$ with $p_f(w)\le s_1\le q$. By Lemma~\ref{L:projequi}(c) there is a projection $s_2\ge p_i(w)$ equivalent to $s_1$. Let $\tilde{w}$ be any partial isometry with $p_i(\tilde{w})=s_2$ and $p_f(\tilde{w})=s_1$.\smallskip

Now it is clear that
$$\g=((g_\alpha)_{\alpha\in\Lambda},y,\tilde{v},\tilde{w})\in \T(M)$$
and $M_2(\h)\subset M_2(\g)$. This completes the proof of assertion $(a)$.\smallskip

We continue by proving $(b)$. Fix $\varphi,\psi\in M_*\setminus\{0\}$ such that
$$\begin{aligned}
s(\varphi)&=\h=((h_\alpha)_{\alpha\in\Lambda},x(\h),v(\h),w(\h)),\\
s(\psi)&=\g=((g_\alpha)_{\alpha\in\Lambda},x(\g),v(\g),w(\g)).
\end{aligned}$$
It is clear that $\norm{\cdot}_\varphi$ is weaker than $\norm{\cdot}_\psi$ on $B_M$ if and only if
\begin{itemize}
    \item $\norm{\cdot}_\varphi$ is weaker than $\norm{\cdot}_\psi$ on $B_{L^\infty(\mu_\alpha,C_\alpha)}$ for each $\alpha\in\Lambda$,
     \item $\norm{\cdot}_\varphi$ is weaker than $\norm{\cdot}_\psi$ on $B_N$,
      \item $\norm{\cdot}_\varphi$ is weaker than $\norm{\cdot}_\psi$ on $B_{pV}$,
       \item $\norm{\cdot}_\varphi$ is weaker than $\norm{\cdot}_\psi$ on $B_{qW}$.
\end{itemize}

Observe that Proposition~\ref{P:LinfyC} yields that $\norm{\cdot}_\varphi$ and $\norm{\cdot}_\psi$ are equivalent on the closed unit ball of ${L^\infty(\mu_\alpha,C_\alpha)}$ whenever both $h_\alpha$ and $g_\alpha$ are nonzero. Further, clearly $\norm{\cdot}_\varphi$ is weaker than $\norm{\cdot}_\psi$ on $B_N$ if and only if $x(\h)\le x(\g)$ (by Proposition~\ref{P:seminorms order}, but in fact this is an easy case).\smallskip

Further, set
$$u(\g)=\sum_{j\in\spt_2\g} p_j,\qquad u(\h)=\sum_{j\in\spt_2\h} p_j.$$
These are $\sigma$-finite projections in $V$ which belong to $pV$, thus there are $\tilde{\varphi},\tilde{\psi}\in (pV)_*\setminus\{0\}$ with $s(\tilde{\varphi})=u(\h)$ and $s(\tilde{\psi})=u(\g)$. By Lemma~\ref{L:pV finite sigmafinite} and Proposition~\ref{P:seminorms order} we see that $\norm{\cdot}_{\varphi}$ and $\norm{\cdot}_{\tilde{\varphi}}$ are equivalent on $B_{pV}$ (and $\norm{\cdot}_{\psi}$ and $\norm{\cdot}_{\tilde{\psi}}$ as well). Now we deduce, via Proposition~\ref{P:seminorms order}, that $\norm{\cdot}_\varphi$ is weaker than $\norm{\cdot}_\psi$ on $B_{pV}$ if and only if $\spt_2\h\subset\spt_2\g$.\smallskip

Now assertion $(b)$ follows easily.\smallskip

Next we provide the following postponed proof.

\begin{proof}[Proof of Corollary~\ref{cor:sigma-directed}.]
We use the notation from Theorem~\ref{T:order}. By assertion $(a)$ in the just quoted theorem, for each $n\in\en$, we can find an element $\h^n\in \T(M)$ such that $M_2(s(\varphi_n))\subset M_2(\h^n)$ for each $n\in\en$. Fix the notation
$$\h^n=((f_\alpha^n)_{\alpha\in\Lambda},x^n,v^n,w^n),\quad n\in\en.$$
For any $\alpha\in\bigcup_n \spt_1 \h^n$ choose a complete tripotent $f_\alpha\in L^\infty(\mu_\alpha,C_\alpha)$ and set $f_\alpha=0$ for the remaining $\alpha\in \Lambda$. Further, set
$x=\sup_n x^n$ and
$$v=\sum_{j\in\bigcup_n\spt_2\h^n} p_j\; .$$

Finally, $w^n$ is a partial isometry in $W$ such that its final projection $p_f(w^n)$ is a properly infinite $\sigma$-finite projection below $q$ for each $n\in\en$. By Lemma~\ref{L:infinitecofinality} there is a $\sigma$-finite properly infinite projection $r\in W$ with
$$\sup_n p_f(w^n)\le r\le q.$$
By Lemma~\ref{L:projequi}$(c)$ we can find, for each $n\in\en$, a projection $s_n\ge p_i(w^n)$ such that $s_n\sim r$. Then $s=\sup_n s_n$ is equivalent to $r$ by Lemma~\ref{L:projequi}$(a)$. So, we can fix a partial isometry $w\in W$ with initial projection $s$ and final projection $r$. Then
$$\h=((f_\alpha)_{\alpha\in\Lambda},x,v,w)\in \T(M).$$
Choose $\varphi\in M_*\setminus\{0\}$ with $s(\varphi)=\h$.
Then $\norm{\cdot}_{\varphi_n}$ is weaker than $\norm{\cdot}_\varphi$ on $B_M$ by Theorem~\ref{T:order}$(b)$
(and Proposition~\ref{P:seminorms order}).
\end{proof}

\section{Characterizations of weakly compact sets and operators}\label{sec:11}

As a byproduct of our investigation we improve characterizations of weakly compact sets in preduals of JBW$^*$-triples and of weakly compact operators on such spaces. We start by recalling the following known result.

\begin{thm}\label{weakly compact characterization}\cite[Theorem 1.1, Corollary 1.4 and Theorem 1.5]{peralta2006some}
Let $K$ be a bounded subset in the predual of a JBW$^*$-triple $M$.
Then the following are equivalent:
\begin{enumerate}[$(a)$]\item $K$ is relatively weakly compact;
\item There exist norm-one normal functionals $\varphi_1, \varphi_2\in M_{*}$ satisfying the following property:
Given $\varepsilon >0,$ there exists $\delta > 0$ such that for every
$x\in M$ with $\|x\|\leq 1$ and $\|x\|_{\varphi_1,\varphi_2} <
\delta,$ we have $| \phi (x)| < \varepsilon$ for every $\phi\in K;$
\item The restriction, $K|_{C},$ of $K$ to each maximal abelian subtriple $C$ of $M$ is relatively weakly compact in $C_*$;
\item For each tripotent $e\in M$ the restriction of $K$ to $M_{2} (e)$ is relatively
weakly compact in $\left(M_{2}(e)\right)_{*}$;
\item For any monotone decreasing sequence of tripotents
$(e_n)$ in $M$ with $(e_n) \to 0$ in the weak$^*$ topology, we have $\displaystyle \lim_{n\to +\infty} \phi (e_n) =0$ uniformly for $\phi\in K$.
\end{enumerate}
If $M$ is a JBW$^*$-algebra then statement $(b)$ can be replaced with the following:
\begin{enumerate}[$(b')$]
\item There exists a normal state $\psi\in M_{*}$ satisfying the following
property: Given $\varepsilon >0,$ there exists $\delta > 0$ such
that for every $x\in M$ with $\|x\|\leq 1$ and $\|x\|_{\psi} <
\delta,$ we have $| \phi (x)| < \varepsilon$ for each $\phi\in K.$
\end{enumerate}
\end{thm}

The equivalence $(a)\Leftrightarrow(b)$ is a generalization of Akemann's theorem \cite{akemann1967dual} characterizing weakly compact sets in predual of von Neumann algebras. Recall that $\norm{x}_{\varphi_1,\varphi_2}^2=\norm{x}_{\varphi_1}^2+\norm{x}_{\varphi_2}^2$, hence it gives also a more precise version of Proposition~\ref{P:strong*=Mackey} on the relationship of strong$^*$ and Mackey topologies.\smallskip

We also notice that a triple $C$ is \emph{abelian} if the operators $L(a,b)$ and $L(x,y)$ commute for any choice $a,b,x,y\in C$ (cf. \cite[p. 468]{Cabrera-Rodriguez-vol1}).\smallskip

As observed in \cite[pages 340--342]{Cabrera-Rodriguez-vol2} the previous theorem can be applied to characterize weakly compact operators from a complex Banach space into the predual of a JBW$^*$-triple and from a JB$^*$-triple into a complex Banach space. The concrete result in the latter case reads as follows.

\begin{thm}\label{t weakly compact charac}\cite[Theorem 10]{peralta2001grothendieck}
Let $E$ be a JB$^*$-triple, $X$ a complex Banach space, and $T:E \rightarrow X$ a
bounded linear operator. Then the following assertions are equivalent:
\begin{enumerate}[$(i)$]
\item $T$ is weakly compact;
\item There exist norm-one functionals $\varphi_1, \varphi_2
\in E^{*}$ and a function $N:(0,+\infty) \rightarrow (0,+\infty)$
such that $$\|T (x) \| \leq N(\varepsilon)\
\|x\|_{\varphi_1,\varphi_2} + \varepsilon \|x\|$$ for all $x\in E$
and $\varepsilon>0$;
\item There exist a bounded linear operator $G$ from $E$ to a real
(respectively, complex) Hilbert space and a function
$N:(0,+\infty) \rightarrow (0,+\infty)$ such that $$\|T (x) \|
\leq N(\varepsilon) \|G(x)\| + \varepsilon \|x\|$$ for all $x\in
E$ and $\varepsilon >0$.
\end{enumerate}
\end{thm}

This result is collected in the recent monograph \cite{Cabrera-Rodriguez-vol2} as Theorem 5.10.141. By quoting \cite{Cabrera-Rodriguez-vol2}, it should be noted that \emph{
``{The above theorem is established in \cite[Theorem 11]{chu1988weakly}, with $\|\cdot\|_{\varphi_1,\varphi_2}$ in condition $(ii)$ replaced with $\|\cdot\|_{\varphi}$ for a single functional $\varphi$ in the unit sphere of $E^*$. Since this refinement depends on an affirmative answer to \cite[Problem 5.10.131]{Cabrera-Rodriguez-vol2}, it should remain in doubt.}''} Problem 5.10.131 refers to the so-called Barton-Friedman conjecture\label{label BF conjecture} for JB$^*$-triples and the subtle difficulties appearing around the original statement of Grothendieck's inequality for JB$^*$-triples published in \cite{barton1987grothendieck} (see \cite{peralta2001little,peralta2001grothendieck,peralta2005new}, \cite[Subsection 5.10.4]{Cabrera-Rodriguez-vol2}, \cite{HKPP} and the final remark in page \pageref{label new final remark} for more details). Summarizing, the problem whether in Theorem \ref{t weakly compact charac}$(ii)$ (respectively, Theorem \ref{weakly compact characterization}$(b)$) the seminorm of the form $\|\cdot\|_{\varphi_1,\varphi_2}$ can be replaced with a seminorm of the form $\|\cdot\|_{\varphi}$ for a single norm-one functional $\varphi\in E^*$ remains as an open question. Our next result provides a positive solution to these problems and proves the validity of the original statement in \cite[Theorem 11]{chu1988weakly}.

\begin{thm}\label{t weakly compact with a single control functional} Let $K$ be a bounded subset in the predual of a JBW$^*$-triple $M$.
Then $K$ is relatively weakly compact if and only if there exists a norm-one normal functional $\varphi\in M_{*}$ satisfying the following property:
Given $\varepsilon >0,$ there exists $\delta > 0$ such that for every
$x\in M$ with $\|x\|\leq 1$ and $\|x\|_{\varphi} < \delta,$ we have $| \phi (x)| < \varepsilon$ for every $\phi\in K.$
\end{thm}

\begin{proof} The `if part' follows from the  implication $(b)\Rightarrow(a)$ in Theorem~\ref{weakly compact characterization},
whereas the `only if part follows from the  implication $(a)\Rightarrow(b)$ in Theorem~\ref{weakly compact characterization}
and Corollary~\ref{cor:sigma-directed}.
 \end{proof}

We can now provide a proof of the statement in \cite[Theorem 11]{chu1988weakly} and close a conjecture which has remained open for over eighteen years. The proof dissipates the commented doubts expressed in \cite[page 341]{Cabrera-Rodriguez-vol2}.

\begin{thm}\label{t weakly compact operators in the statement of ChuIochum} Let $M$ be a JBW$^*$-triple, $E$ a JB$^*$-triple, and let $X$ be a complex Banach space. Then the following statements hold:
\begin{enumerate}[$(a)$]
\item A bounded linear operator $T: X\to M_*$ is weakly compact if and only if there exists a norm-one functional $\varphi\in M_{*}$ and a function $N:(0,+\infty) \rightarrow (0,+\infty)$ such that $$\|T^* (a) \| \leq N(\varepsilon)\ \|a\|_{\varphi} + \varepsilon \|a\|$$ for all $a\in M$ and $\varepsilon>0$;
\item A bounded linear operator $T: E\to X$ is weakly compact if and only if there exists a norm-one functional $\varphi\in E^{*}$ and a function $N:(0,+\infty) \rightarrow (0,+\infty)$ such that $$\|T (a) \| \leq N(\varepsilon)\ \|a\|_{\varphi} + \varepsilon \|a\|$$ for all $a\in E$ and $\varepsilon>0$.
\end{enumerate}
\end{thm}

\begin{proof} In both statements the `if parts'  follow from Theorem \ref{t weakly compact charac}.\smallskip

$(a)$ Suppose $T: X\to M_*$ is weakly compact. Since $T(B_X)\subseteq M_*$ is relatively weakly compact, Theorem \ref{t weakly compact with a single control functional} implies the existence of a norm-one normal functional $\varphi\in M_{*}$ satisfying the following property:
Given $\varepsilon >0,$ there exists $\delta > 0$ such that for every
$a\in M$ with $\|a\|\leq 1$ and $\|a\|_{\varphi} < \delta,$ we have $| T(x) (a) | < \varepsilon$ for every $x\in B_X.$\smallskip

Given $a\in M\backslash\{0\},$ the element $b= \frac{a}{\|a\|+\delta^{-1} \|a\|_{\varphi}}\in B_M$ and satisfies $\|b\|_{\varphi}<\delta$, therefore
$| T(x) (b) | < \varepsilon$ for every $x\in B_X,$ equivalently,
$$ |T^*(a) (x) | = |T(x) (a)| < \varepsilon \delta^{-1} \|a\|_{\varphi} + \varepsilon \|a\|,$$ for all $x\in B_X$, and thus $$ \|T^*(a) \| \leq \varepsilon \delta^{-1} \|a\|_{\varphi} + \varepsilon \|a\|, \hbox{ for all $a\in M$}.$$

Statement $(b)$ follows from $(a)$ since, by virtue of Gantmacher's theorem, an operator $T: E\to X$ is weakly compact if and only if $T^*: X^*\to E^*$ is.
\end{proof}

\section{Final remarks and open problems}%\label{sec:12}

Theorem~\ref{T:main} says that the three measures of weak non-compactness considered in this paper coincide in preduals of JBW$^*$-triples which complements the previous results from \cite{qdpp,HK-nuclear}. However, the mentioned results include explicit formulas for these measures. In fact, these formulas are substantially used in the proofs. In the present paper we do not get an explicit formula due to the procedure of the proof -- we use subsequence splitting property, Lemma~\ref{L:omega on JBW*algebra}, Proposition~\ref{P:triple 1-complemented} and Lemma~\ref{L:measures 1-complemented}. So it is natural to ask whether there are some natural formulas for the De Blasi measure $\omega$. The first question deals with a special case of von Neumann algebras.

\begin{ques}
Let $M$ be a semifinite von Neumann algebra with a fixed normal semifinite faithful trace $\tau$. Is there a formula for the De Blasi measure of weak noncompatness in $M_*$ in terms of the trace $\tau$? Is it so at least for finite $\sigma$-finite von Neumann algebras?
\end{ques}

We note that the special case of commutative spaces $L^1(\mu)$ is settled in \cite[Section 7]{qdpp}.\smallskip

Another possibility is to try to get  quantitative versions of some characterizations of weakly compact sets in the predual of a JBW$^*$-triple
given in Theorem~\ref{weakly compact characterization} (or in the improvement of its assertion (b) contained in Theorem~\ref{t weakly compact with a single control functional}). More precisely, we have the following question.

\begin{ques} Let $M$ be a JBW$^*$-triple and let $A\subset M_*$ be a bounded set. Can $\omega(A)$ be expressed using quantitative versions of the characterizations from Theorem~\ref{weakly compact characterization}? In particular, is $\omega(A)$ equal (or at least equivalent) to the following quantities?
\begin{enumerate}[$(a)$]\setcounter{enumi}{1}
    \item $\displaystyle\inf_{\varphi\in S_{M_*}}\, \sup_{\phi\in A}\, \inf_{\delta>0}\, \sup\, \{ \abs{\phi(x)}\setsep  x\in B_M, \norm{x}_{\varphi}<\delta\}$;
    \item $\sup\,\{\omega(A|_C)\setsep C\subset M\mbox{ a maximal abelian subtriple}\}$;
    \item $\sup\,\{\omega(A|_{M_2(e)})\setsep e\in M\mbox{ a tripotent}\}$;
    \item  $\displaystyle
    %\smash[t]{\begin{aligned} &\\
       \sup\, \left\{\smash[b]{ \limsup_{n\to\infty} \sup_{\phi\in A}}\abs{\phi(e_n)} \setsep \begin{array}{c}
  (e_n) \mbox{ a decreasing sequence of tripotents} \\
   \mbox{ weak$^*$-converging to }0
\end{array}\right\}. %\end{aligned}}
$
\end{enumerate}
\end{ques}

Note that these quantities naturally correspond to the respective characterizations in Theorem~\ref{weakly compact characterization} (or in the improvement of assertion (b) contained in Theorem~\ref{t weakly compact with a single control functional}). It is easy to check that all these quantities are bounded above by the De Blasi measure $\omega$, but the converse inequalities seem not to be obvious.

We investigated measures of weak non-compactness in preduals of JBW$^*$-triples. Another possibility is to look at subsets of JB$^*$-triples themselves. In this direction there is just one positive result -- coincidence of measures of weak non-compactness for subsets of $c_0(\Gamma)$ (see \cite[Theorem C]{HK-nuclear}) and no negative result up to now. So, the following question seems to be natural.

\begin{ques}
Are the above-considered measures of weak non-compactness equi\-va\-lent (or even equal) for bounded subsets of a JBW$^*$-triple?
\end{ques}

We have no idea how to approach this general question, so we formulate two important special cases.

\begin{ques} Let $K$ be a compact space.
Are the above-considered measures of weak non-compactness equivalent (or even equal) for bounded subsets of $\C(K)$?
Is it true for $K=[0,1]$?\end{ques}

\begin{ques}
Let $H$ be a Hilbert space.
Are the  above-considered measures of weak non-compactness equi\-va\-lent (or even equal) for bounded subsets of
$K(H)$, the space of compact operators on $H$?\end{ques}

\noindent\textbf{Added during the revision process:}\label{label new final remark} Months after the submission of this paper we discovered a complete proof of the so-called Barton-Friedman conjecture for general JB$^*$-triples, and a solution to Problem 5.10.131 in \cite{Cabrera-Rodriguez-vol2} (treated in page \pageref{label BF conjecture}). The result is included in the recent preprint \cite{HKPP}. This proof of the Barton-Friedman conjecture offers an alternative approach to derive Theorems \ref{t weakly compact with a single control functional} and \ref{t weakly compact operators in the statement of ChuIochum} as a straightforward consequences of \cite[Theorem 10]{peralta2001grothendieck}. \medskip\medskip

\noindent
\textbf{\it Acknowledgements:} We would like to thank the anonymous referee for the time employed in writing a professional and thorough report with a wide list of constructive and enriching comments and suggestions.

%\bibliography{bibliography}\bibliographystyle{acm}

\def\cprime{$'$} \def\cprime{$'$}

\end{document}